\documentclass[12pt]{amsart}

\usepackage{amssymb,amscd,amsthm, verbatim,amsmath,color,fancyhdr, mathrsfs}
\usepackage{graphicx}
\usepackage{turnstile}
\usepackage[disable]{todonotes}
\usepackage[bookmarks=true]{hyperref}
\usepackage{csquotes}

\usepackage[letterpaper, left=2.5cm, right=2.5cm, top=2.5cm,bottom=2.5cm,dvips]{geometry}

\newtheorem{thm}{Theorem}[section]
\newtheorem{prop}[thm]{Proposition}
\newtheorem{lem}[thm]{Lemma}
\newtheorem{rem}[thm]{Remark}
\newtheorem{cor}[thm]{Corollary}

\newtheorem{ex}[thm]{Example}
\newtheorem{defn}[thm]{Definition}
\newtheorem{fact}[thm]{Fact}
\newtheorem{notation}[thm]{Notation}
\newtheorem{conv}[thm]{Convention}

\newtheorem{construction}[thm]{Construction}
\newtheorem{question}[thm]{Question}
\newtheorem{assumption}[thm]{Assumption}

\title{Quasi-homogeneity of the moduli space of stable maps to homogeneous spaces}

\author{Christoph B\"arligea}
\address{Ruhr-Universit\"at Bochum\\ Fakult\"at f\"ur Mathematik\\ Arbeitsgruppe Transformationsgruppen\\ Universit\"atsstra{\ss}e 150\\ 44780 Bochum}
\email{christoph.baerligea@rub.de}
\thanks{The research was partially supported by the German Research Foundation (DFG)}

\date{June 20, 2017}	

\begin{document}

\begin{abstract}

Let $G$ be a connected, simply connected, simple, complex, linear algebraic group. Let $P$ be an arbitrary parabolic subgroup of $G$. Let $X=G/P$ be the $G$-homogeneous projective space attached to this situation. Let $d\in H_2(X)$ be a degree. Let $\overline{M}_{0,3}(X,d)$ be the (coarse) moduli space of three pointed genus zero stable maps to $X$ of degree $d$. We prove under reasonable assumptions on $d$ that $\overline{M}_{0,3}(X,d)$ is quasi-homogeneous under the action of $G$. 

The essential assumption on $d$ is that $d$ is a minimal degree, i.e. that $d$ is a degree which is minimal with the property that $q^d$ occurs with non-zero coefficient in the quantum product $\sigma_u\star\sigma_v$ of two Schubert cycles $\sigma_u$ and $\sigma_v$, where $\star$ denotes the product in the (small) quantum cohomology ring $QH^*(X)$ attached to $X$.
%$(QH^*(X),\star)$ denotes the (small) quantum cohomology ring attached to $X$. 
%and $q$ denotes the quantum parameter. 
We prove our main result on quasi-homogeneity by constructing an explicit morphism which has a dense open $G$-orbit in $\overline{M}_{0,3}(X,d)$. To carry out the construction of this morphism, we develop a combinatorial theory of generalized cascades of orthogonal roots which is interesting in its own right.

\end{abstract}

\maketitle

\section{Introduction}

Let $G$ be a connected, simply connected, simple, complex, linear algebraic group. Let $P$ be a fixed but arbitrary parabolic subgroup of $G$. Let $X=G/P$ be the $G$-homogeneous projective space attached to this situation. We select once and for all a maximal torus $T$ and a Borel subgroup $B$ of $G$ such that
$$
T\subseteq B\subseteq P\subseteq G\,.
$$

We call an effective homology class in $H_2(X)$ a degree. Let $d$ be a degree. Let $\overline{M}_{0,3}(X,d)$ be the (coarse) moduli space of three pointed genus zero stable maps to $X$ of degree $d$. By definition, the moduli space $\overline{M}_{0,3}(X,d)$ parametrizes isomorphism classes
$
[C,p_1,p_2,p_3,\mu\colon C\to X]
$
where:
\begin{itemize}

\item 

$C$ is a complex, projective, connected, reduced, (at worst) nodal curve of arithmetic genus zero.

\item

The marked points $p_i\in C$ are distinct and lie in the nonsingular locus.

\item

$\mu$ is a morphism such that $\mu_*[C]=d$.

\item

The pointed map $\mu$ has no infinitesimal automorphisms.

\end{itemize}

\begin{rem}

The stability condition in the fourth item is often reformulated in terms of a condition on the special points on each irreducible component of $C$ (cf. \cite[Subsection~0.4]{fultonpan}).

\end{rem}

\todo[inline,color=green]{More detailed conventions in the next section: coefficients in $\mathbb{Z}$, degree conventions (I will repeat doubly that a degree is an effective homology class in $H_2(X)$ -- once here in the intro, later in the next section), parabolic conventions, so that $Q$ corresponds to $\Delta_Q\subset\Delta$, Chern class convention. The point class is well-introduced in the first example of this introduction.}

Basic properties of the moduli space $\overline{M}_{0,3}(X,d)$ can be found in \cite{fultonpan}. It is a consequence of more general results in \cite{fultonpan,pand}, namely of \cite[Theorem~2(i)]{fultonpan} and \cite[Corollary~1]{pand}, that $\overline{M}_{0,3}(X,d)$ is a normal projective irreducible variety.

In this work, we ask the question if it is possible to prove stronger properties of $\overline{M}_{0,3}(X,d)$ than irreducibility. Note that the group $G$ acts on $\overline{M}_{0,3}(X,d)$ by translation. Hence, the following question makes sense.

\begin{question}
\label{quest:quasi}

For which degrees $d$ is the moduli space $\overline{M}_{0,3}(X,d)$ quasi-homogeneous under the action of $G$?

\end{question}

To give an affirmative answer to Question~\ref{quest:quasi} for specific degrees, it turns out that the class of all minimal degrees is a good starting point. Minimal degrees and their properties were studied in \cite{minimaldegrees}. We build on this work in various ways.

\begin{defn}

Let $(QH^*(X),\star)$ be the (small) quantum cohomology ring attached to $X$ as defined in \cite[Section~10]{fultonpan}. For a Weyl group element $w$, we denote by $\sigma_w$ the Schubert cycle attached to $w$. We say that a degree $d$ is a minimal degree if there exist Weyl group elements $u$ and $v$ such that $d$ is a minimal degree in $\sigma_u\star\sigma_v$, i.e. if the power $q^d$ appears with non-zero coefficient in the expression $\sigma_u\star\sigma_v$ and if $d$ is minimal with this property.

\end{defn}

\begin{rem}

Using the theory of curve neighborhoods \cite{curvenbhd2} reviewed in Subsection~\ref{subsec:nbhd}, one can give a quantum cohomology independent description of minimal degrees (cf. Definition~\ref{def:minimaldegrees}). We will mostly choose the approach to minimal degrees through curve neighborhoods because it is more explicit and more computable in terms of combinatorics. In this introduction, we gave the definition of minimal degrees in terms of quantum cohomology because it is short and it shows where the ideas originate. Nevertheless, the rest of the text is basically independent from the context of quantum cohomology and does not assume familiarity with it.

\end{rem}

\begin{notation}

We denote by $\Pi_P$ the set of all minimal degrees. We use this notation to make clear relative to which $P$ the minimal degrees are computed. In particular, the set of all minimal degrees in $H_2(G/B)$ is denoted by $\Pi_B$. The sets $\Pi_P$ and $\Pi_B$ will be often simultaneously in use.

\end{notation}

\begin{ex}

By Corollary~\ref{cor:uniqueness}, Remark~\ref{rem:uniqueness}, there exists a unique minimal degree in the quantum product $\mathrm{pt}\star\mathrm{pt}$ of two point classes. We denote this degree by $d_X\in\Pi_P$. In particular, we have $d_{G/B}\in\Pi_B$. By Corollary~\ref{cor:dsmallerdx}, we have an inclusion $\Pi_P\subseteq\{0\leq d\leq d_X\}$. Note that this inclusion is possibly strict unless $P$ is maximal (cf. \cite[Example~10.14]{minimaldegrees}). For more informations on the minimal degree $d_X$, we refer to \cite{minimaldegrees}.

\todo[inline,color=green]{To explain this, I can refer directly (or doubly) to the stronger Postnikov results.}

\end{ex}

Starting from the class of minimal degrees, we impose one further assumption on a minimal degree, namely Assumption~\ref{ass:main}, to give an affirmative answer to Question~\ref{quest:quasi}. For this introduction, it suffices to know that Assumption~\ref{ass:main} is satisfied for all minimal degrees if the root system $R$ associated to $G$ and $T$ is simply laced or if $P=B$, i.e. if $X$ is a generalized complete flag variety. We are able to obtain the following theorem.

%Starting from the class of minimal degrees, we define a more restrictive subclass of so-called $P$-admissible degrees. For those degrees, we are able to obtain the following theorem.

\begin{thm}[Theorem~\ref{thm:main}]
\label{thm:main_intro}

The moduli space $\overline{M}_{0,3}(X,d)$ is quasi-homogeneous under the action of $G$ for all 
%$P$-admissible degrees $d$.
minimal degrees $d$ which satisfy Assumption~\ref{ass:main}.

\end{thm}

%Since the definition of $P$-admissible degrees is a bit technical in detail, we rather give the main examples of $P$-admissible degrees instead of defining them properly in this introduction. For a precise definition, the reader can go to Definition~\ref{def:Padmissible}.

We prove Theorem~\ref{thm:main_intro} by constructing for every minimal degree $d$ an explicit morphism $f_{P,d}$. Then, we show that $f_{P,d}$ has a dense open orbit in $\overline{M}_{0,3}(X,d)$ under the action of $G$ if $d$ satisfies Assumption~\ref{ass:main}. We do so by comparing the dimension of the orbit of $f_{P,d}$ with the dimension of the moduli space. For the convenience of the reader, we sketch the construction of $f_{P,d}$ in this introduction.

\begin{construction}
\label{construction:dia}

Let $d\in\Pi_P$. In Section~\ref{sec:lifting}, we uniquely associate to $d$ a minimal degree $e\in\Pi_B$ -- the so-called lifting of $d$. The precise definition of $e$ is given in Definition~\ref{def:lifting}. For the moment, it suffices to know that the image of $e$ under the natural map $H_2(G/B)\to H_2(X)$ is $d$ (cf. Fact~\ref{fact:lifting}\eqref{item:ePequalsd}). In this way, the lifting helps us to transport the situation from $X$ to $G/B$.

In the next step, we associate to the degree $e\in\Pi_B$ a set of (strongly) orthogonal roots $\mathcal{B}_{R,e}$ -- a so-called generalized cascade of orthogonal roots. We systematically investigate the properties of these generalized cascades of orthogonal roots in Section~\ref{sec:gencascade}. The reader finds the precise definition of $\mathcal{B}_{R,e}$ in Definition~\ref{def:gencascade}. 

\todo[inline,color=green]{I assume familiarity with the notion \enquote{strongly orthogonal roots} although I recall it in the text later, but without referring to it now. It is basic.}

Let $R_P$ be the root system associated to the Levi factor of $P$ and $T$. Let $R_P^+$ be the positive roots of $R_P$ induced by $B$. For a root $\alpha\in\mathcal{B}_{R,e}\setminus R_P^+$, there is a unique irreducible $T$-invariant curve $C_\alpha\subseteq X$ containing the $T$-fixed points associated to $1$ and $s_\alpha$, where $s_\alpha$ is the reflection along $\alpha$ (cf. \cite[Lemma~4.2]{fulton}). Each of the curves $C_\alpha$ is isomorphic to $\mathbb{P}^1$ (cf. \cite[Lemma~4.2]{fulton}).

With these preliminaries, we can now define the morphism $f_{P,d}$ as in Definition~\ref{def:dia}.
%With this notation, we can proceed as in Definition~\ref{def:dia}
We define the morphism $f_{P,d}$ by the assignment 
$$
f_{P,d}\colon\mathbb{P}^1\hookrightarrow\prod_{\alpha\in\mathcal{B}_{R,e}\setminus R_P^+}C_\alpha\hookrightarrow X
$$
where the first morphism is the diagonal embedding of $\mathbb{P}^1$ into $\mathrm{card}(\mathcal{B}_{R,e}\setminus R_P^+)$ isomorphic copies of $\mathbb{P}^1$ and the second morphism is the embedding into $X$ which is well-defined due to the fact that two distinct elements of $\mathcal{B}_{R,e}$ are (strongly) orthogonal (cf. Theorem~\ref{thm:gencascade}\eqref{item:stronglyorthogonal}). Also, by the very same fact, the definition of $f_{P,d}$ is independent of the ordering of the product $\prod_{\alpha\in\mathcal{B}_{R,e}\setminus R_P^+}$. Hence, the morphism $f_{P,d}$ is well-defined. We call the image $f_{P,d}(\mathbb{P}^1)$ the diagonal curve (associated to $d$).

\end{construction}

One can show that $f_{P,d}$ has degree $(f_{P,d})_*[\mathbb{P}^1]=d$ (cf. Fact~\ref{fact:dia}). Hence, the morphism $f_{P,d}$ 
%gives indeed a candidate for an element of $\overline{M}_{0,3}(X,d)$ which has a dense open orbit under the action of $G$.
gives indeed a candidate for an element of $\overline{M}_{0,3}(X,d)$ which is interesting for an answer of Question~\ref{quest:quasi}.

We give an example for a class of degrees for which the diagonal curve turns out to be particularly simple and explicit.

\begin{ex}
\label{ex:dia_pcosmall}

Let $\alpha_1,\ldots,\alpha_k$ be $P$-cosmall roots (cf. Definition~\ref{def:pcosmall}) such that the supports of $\alpha_1,\ldots,\alpha_k$ are pairwise totally disjoint (cf. Notation~\ref{not:support}, Definition~\ref{def:totallydisjoint}). Let $d=\sum_{i=1}^k[C_{\alpha_i}]$, where $C_{\alpha_i}$ is the curve associated to $\alpha_i$ as in Construction~\ref{construction:dia}. Then we have $d\in\Pi_P$ (cf. Theorem~\ref{thm:splitting2}). Let $e$ be the lifting of $d$. By Theorem~\ref{thm:splitting2}, we then have $\mathcal{B}_{R,e}\setminus R_P^+=\{\alpha_1,\ldots,\alpha_k\}$. We see directly from the assumption on $\alpha_i$ that two distinct elements of $\mathcal{B}_{R,e}\setminus R_P^+$ are (strongly) orthogonal (cf. Fact~\ref{fact:stronglyorthogonal}). Hence, in this case, we can verify directly that the diagonal curve (associated to $d$) is well-defined. The morphism $f_{P,d}$ is given by
$$
f_{P,d}\colon\mathbb{P}^1\hookrightarrow\prod_{i=1}^kC_{\alpha_i}\hookrightarrow X
$$
where the two arrows are defined as in Construction~\ref{construction:dia}.

\end{ex}

\begin{rem}

The proof of Theorem~\ref{thm:main_intro} for degrees as in Example~\ref{ex:dia_pcosmall} (i.e. degrees which satisfy Assumption~\ref{ass:main}\eqref{item:ass_pcosmall}) is very simple and can be done directly without any further considerations. 
(As an exercise, the reader may provide the arguments that the morphism described in Example~\ref{ex:dia_pcosmall} has a dense open $G$-orbit.)
The point is not that we get an affirmative answer to Question~\ref{quest:quasi} for degrees which satisfy Assumption~\ref{ass:main}\eqref{item:ass_pcosmall}, but rather that the general construction of $f_{P,d}$ subsumes this case as a trivial sub-case. In this way, we gain confidence that Construction~\ref{construction:dia} is reasonable even in a more general setting.

\end{rem}

\todo[inline,color=green]{{\color{black} I cite \cite{pand} to make the warning go away. It is about sharpening these results, since quasi-homogeneity implies irreducibility. Give a summary, and motivation (like the curve associated to a $P$-cosmall root, and, as a direct generalization, to a bunch of $P$-cosmall roots with pairwise disjoint support). Mention the two appendices which are logically independent from the main purpose, and only build on the discussion in the preliminary section.}}

\todo[inline,color=green]{The reader only interested in the case of generalized complete flag varieties can skim through the sections a--b verifying each time that the results are trivial for $P=B$. The construction of the diagonal curve simplifies a lot (e.g. the step of passing from $d$ to its lifting is superfluous). I will amply explain this in the organization section.} 

\todo[inline,color=green]{Maybe, I need to introduce the curve $C_\alpha$ for $\alpha\in R^+\setminus R_P^+$ directly in the introduction, as well as the $T$-fixed points $x(w)$, because I need those objects to explain and motivate. I probably introduce the evaluation maps already in the introduction, or in the beginning of the main section.

I probably have to introduce $\Pi_P,\Pi_B,d_X,d_{G/B}$ doubly, once in the introduction and once in Section~\ref{sec:preliminaries}. In the history section, I was avoiding the $\Pi$-notation, but I am definitely using $d_X$.

I should check later if every term from the history is explained in the introduction before.}

\subsection*{History}

The problem to prove quasi-homogeneity of the moduli space of stable maps to homogeneous spaces was first posed in \cite[Section~3.2]{dmax}. A preliminary construction of the diagonal curve $f_{P,d}(\mathbb{P}^1)$ for maximal parabolic subgroups $P$ and degrees $0\leq d\leq d_X$ is given there and it is stated (without rigorous proof) that $f_{P,d}$ has a dense open orbit in $\overline{M}_{0,3}(X,d)$ (cf. \cite[Proposition~3.1]{dmax}). Preliminary attempts to give a proof of this theorem for the degree $d=d_X$ were undertaken in \cite[Section~9]{thesis}. In this series of work, the present paper can be seen as a more rigorous and more final attempt to prove quasi-homogeneity for arbitrary parabolic subgroups and arbitrary minimal degrees. 

The very idea of the construction of the diagonal curve $f_{P,d}(\mathbb{P}^1)$ goes back to \cite{dmax}. However, in this work, we make clear in a general context to which set of strongly orthogonal roots, namely $\mathcal{B}_{R,e}$ where $e$ is the lifting of a minimal degree $d$, one has to associate the morphism $f_{P,d}$ to -- a question which is left open in \cite{dmax}.

In particular, the combinatorial aspects of minimal degrees which lead to the essential properties of generalized cascades of orthogonal roots and make our results eventually possible were already well prepared in \cite{minimaldegrees}. The reader can consider this paper, in particular Section~\ref{sec:preliminaries}, \ref{sec:gencascade}, as a follow-up which completes some aspects of the theory developed in \cite{minimaldegrees}.

%It must be clear from the context, or from \cite[Proof of Corollary~10.11]{minimaldegrees} that $\Pi_P=\{0\leq d\leq d_X\}$ for maximal parabolic subgroups $P$.

\subsection*{Organization}

As being said in the history section, a reader only interested in the combinatorial aspects of minimal degrees and generalized cascades of orthogonal roots may read Section~\ref{sec:preliminaries}, \ref{sec:gencascade} independently from the rest of the text. 

Moreover, most (or even all) of the results in Section~\ref{sec:positivity}, \ref{sec:lifting} are only interesting in the relative setting modulo $P$ whenever $P\neq B$. The reader can check for each of those results that the statements turn out to be trivial if $P=B$. 

Section~\ref{sec:lifting} deals with the theory of liftings. The step of passing from a degree $d\in\Pi_P$ to its lifting $e\in\Pi_B$ is superfluous if $P=B$, in the sense that the lifting of $e\in\Pi_B$ is $e$ itself. Therefore, the construction of $f_{B,e}$ for $e\in\Pi_B$ simplifies reasonably. Only the theory of generalized cascades of orthogonal roots is necessary to understand it. 

All in all, a reader only interested in Theorem~\ref{thm:main_intro} for degrees which satisfy Assumption~\ref{ass:main}\eqref{item:g/b}, i.e. in the theorem that $\overline{M}_{0,3}(G/B,e)$ is quasi-homogeneous under the action of $G$ for all minimal degrees $e\in\Pi_B$, can skip Section~\ref{sec:positivity}, \ref{sec:lifting} and go directly from Section~\ref{sec:gencascade} to Section~\ref{sec:dia}. For the case of a generalized complete flag variety $G/B$ and minimal degrees $e\in\Pi_B$, the proof of Theorem~\ref{thm:main_intro} simplifies a lot and relies only on our considerations in Section~\ref{sec:gencascade}.

\subsection*{Acknowledgment}

As being said in the history section, the author has adopted the problem of quasi-homogeneity from Nicolas Perrin and Pierre-Emmanuel Chaput \cite{dmax}. The author is grateful to both. In particular, Nicolas Perrin contributed in various ways to the genesis of this paper: In a preliminary draft version of this paper, Nicolas Perrin pointed out a crucial mistake in Lemma~\ref{lem:act} and helped the author to find the right set of tangent directions (cf. Notation~\ref{not:tangentdirection}). The author also would like to thank Nicolas Ressayre for an invitation to Lyon to give a talk on the subject of this paper. Finally, the author thanks Laura Visu-Petra for miscellaneous discussions. 

\section{Notation and conventions}

In this section, we set up basic notation and summarize well-known terminology concerning the theory of algebraic groups. The notation and the conventions are very similar to those in \cite[Section~1]{minimaldegrees}. We recall everything which is necessary to understand this paper. For more details, the reader may go to \cite{minimaldegrees}.

Let $R$ be the root system associated to $G$ and $T$. Let $R^+$ be the positive roots of $R$ associated to $B$. Let $\Delta$ be the set of simple roots associated to $R^+$. Let
$$
W=N_G(T)/T\text{ and }W_P=N_P(T)/T
$$
be the Weyl group of $G$ and $P$ respectively.
%The parabolic subgroup $P$ uniquely determines and is determined by its set of simple roots
Let $\Delta_P=\{\beta\in\Delta\mid s_\beta\in W_P\}$. 
%The group $W_P$ is a parabolic subgroup of $W$ in the sense that it is generated by the simple reflections $s_\beta$ for $\beta\in\Delta_P$. By the Bruhat decomposition (\cite[28.3, Theorem]{humphreys}), we have $P=BW_PB$. 
We set $R_P=R\cap\mathbb{Z}\Delta_P$ and $R_P^+=R_P\cap R^+$. The positive roots $R^+$ clearly induce a partial order \enquote{$\leq$} on $R$. In turn, this partial order induces via restriction a partial order on $R_P$ which is still denotes by \enquote{$\leq$} and coincides with the partial order induced by $R_P^+$.

\begin{notation}

We denote by $R^-$ the set of negative roots of $R$. In other words, we have $R^-=\{-\alpha\mid\alpha\in R^+\}=\{\alpha\in R\mid\alpha<0\}$. Similarly, we define $R_P^-$ by the requirement that $R_P^-=\{-\gamma\mid\gamma\in R_P^+\}=\{\gamma\in R_P\mid\gamma<0\}$.

\end{notation}

\begin{conv}

From now on, if we speak about a parabolic subgroup, we always mean what is usually called a standard parabolic subgroup (relative to the fixed $B$), i.e. a parabolic subgroup of $G$ containing $B$. In other words, by convention, all parabolic subgroups are standard.
Following this convention, parabolic subgroups of $G$ correspond one to one to subsets of $\Delta$  (cf. \cite[30.1]{humphreys}). In particular, the fixed parabolic subgroup $P$ corresponds to $\Delta_P$ as defined above.
For an arbitrary (standard) parabolic subgroup $Q$, we denote by $\Delta_Q$ the subset of $\Delta$ associated to $Q$ via this correspondence. Vice versa, for a given subset $S$ of $\Delta$, we often uniquely define a parabolic subgroup $Q$ by the requirement $\Delta_Q=S$. 

\end{conv}

Throughout the discussion, we fix a $W$-invariant scalar product $(-,-)$ on $\mathbb{R}\Delta$. This scalar product is unique up to non-zero scalar. Each root $\alpha\in R$ has a coroot $\alpha^\vee$ which is defined by $\alpha^\vee=\frac{2\alpha}{(\alpha,\alpha)}$. All coroots together form the dual root system $R^\vee=\{\alpha^\vee\mid\alpha\in R\}$. The set of simple coroots of $R^\vee$ is given by $\Delta^\vee=\{\beta^\vee\mid\beta\in\Delta\}$. For each $\beta\in\Delta$ we denote by $\omega_\beta\in\mathbb{R}\Delta$ the corresponding fundamental weight. It is defined by the equation $(\omega_\beta,\beta')=\delta_{\beta,\beta'}$ for all $\beta'\in\Delta$.

On the Weyl group, we have a natural length function. For $w\in W$, the length of $w$, denoted by $\ell(w)$, is defined to be the number of simple reflections in a reduced expression of $W$. It is well-known that this number does not depend on the choice of the reduced expression. Each coset $wW_P\in W/W_P$ has a unique minimal and maximal representative, i.e. contains a unique element of minimal and maximal length. 
%
%We denote by $W^P$ the set of all minimal representatives of cosets in $W/W_P$. 
%
The length function carries over from $W$ to $W/W_P$. The length of a coset $wW_P\in W/W_P$, denoted by $\ell(wW_P)$, is defined to be the length of the minimal representative in $wW_P$.

\begin{notation}

For each element $w\in W$, we denote by $I(w)=\{\alpha\in R^+\mid w(\alpha)<0\}$ the inversion set of $w$. With this notation, we have the identities
$
\ell(w)=\mathrm{card}(I(w))\text{ and }\ell(wW_P)=\mathrm{card}(I(w)\setminus R_P^+)
$
for all $w\in W$. A proof of these equalities can be found in \cite[5.6, Proposition~(b)]{humphreys3}.

\end{notation}

\begin{notation}

We denote by $w_o$ the longest element of $W$, i.e. the unique element of $W$ with maximal length. Similarly, we denote by $w_P$ the longest element of $W_P$, i.e. the unique element of $W_P$ with maximal length. Note that $w_o$ and $w_P$ are both involutions. 
%We denote by $w_X$ the minimal representative in $w_oW_P$. We have the relation $w_ow_P=w_X$ or equivalent $w_o=w_Xw_P$. For each $w\in W$, we write $w^*=w_ow$ for short. 
%We call $w^*$ the Poincar\'e dual of $w$.

\end{notation}

On the Weyl group, we have a natural partial order \enquote{$\preceq$} -- the so-called Bruhat. This partial order can be defined in terms of the Bruhat graph as in \cite[5.9]{humphreys3}. It has an equivalent characterization in terms of subexpressions as described in \cite[5.10]{humphreys3}. The geometric meaning of the Bruhat order is given by inclusions of Schubert varieties in $G/B$. We explain this geometric meaning in more detail once we have recalled the notion of Schubert varieties in general (cf. Remark~\ref{rem:geombruhat}).

\begin{conv}

All homology and cohomology groups in this paper are taken with integral coefficients. By convention, we write $H_*(X)=H_*(X,\mathbb{Z})$ and $H^*(X)=H^*(X,\mathbb{Z})$. For a closed irreducible subvariety $Z\subseteq X$, we denote by $[Z]\in H^{2\mathrm{codim}(Z)}(X)$ the cohomology class of $Z$. By abuse of notation, we also denote with the same symbol $[Z]\in H_{2\mathrm{dim}(Z)}(X)$ the homology class of $Z$. Both definitions are Poincar\'e dual to each other. 
%For a cohomology class $\sigma\in H^*(X)$, we denote by $\sigma^*\in H^*(X)$ the cohomology class which is dual to $\sigma$ with respect to the intersection pairing.

\end{conv}

Let $B^-=w_oBw_o$ be the Borel subgroup of $G$ opposite to $B$. Let $w\in W$. 
%We denote by $\Omega_w=BwP/P$ the Schubert cell associated to $w$. 
We denote by $X_w=\overline{BwP/P}$ the Schubert variety associated to $w$. We denote by $Y_w=\overline{B^-wP/P}$ the opposite Schubert variety associated to $w$. Note that $X_w$ and $Y_w$ depend only on $wW_P$. %, so that we can equally well parameterize those varieties by elements of $W/W_P$. 
We have the following equality for the dimension and codimension of Schubert and opposite Schubert varieties:
\begin{equation}
\label{eq:dimschubert}
\mathrm{dim}(X_w)=\mathrm{codim}(Y_w)=\ell(wW_P)\,.
\end{equation}

\begin{rem}
\label{rem:geombruhat}

Let $w\in W$. If we denote by $(G/B)_w$ the Schubert variety in $G/B$ associated to $w$, the geometric meaning of the Bruhat order is given by the equivalence:
$$
u\preceq v\Longleftrightarrow(G/B)_u\subseteq(G/B)_v\text{ where }u,v\in W\,.
$$

\end{rem}

Let $w\in W$. Using $X_w$ and $Y_w$ we can define Schubert cycles
$$
\sigma(w)=[X_w]\in H_{2\ell(wW_P)}(X)%\cong H^{2(\mathrm{dim}(X)-\ell(wW_P))}(X)
\text{ and }\sigma_w=[Y_w]\in H^{2\ell(wW_P)}(X)\,.
$$
From the Bruhat decomposition of $X$ (cf. \cite[28.3, Theorem]{humphreys})
%$
%X=\coprod_{w\in W^P}\Omega_w
%$
it follows easily that the cohomology of $X$ decomposes as direct sums
\begin{equation}
\label{eq:cohdecomp}
H^*(X)=\bigoplus_{w}\mathbb{Z}\sigma(w)=\bigoplus_{w}\mathbb{Z}\sigma_w
\end{equation}
where each of the direct sums in the equation ranges over all minimal representatives $w$ of the cosets in $W/W_P$. Poincar\'e duality transforms one basis of Schubert cycles into the other basis of Schubert cycles and vice versa.

Using Equation~\eqref{eq:dimschubert} and \eqref{eq:cohdecomp}, we see that we have the following decompositions
\begin{equation}
\label{eq:degdecomp}
H_2(X)=\bigoplus_{\beta\in\Delta\setminus\Delta_P}\mathbb{Z}\sigma(s_\beta)\text{ and }H^2(X)=\bigoplus_{\beta\in\Delta\setminus\Delta_P}\mathbb{Z}\sigma_{s_\beta}\,.
\end{equation}
In this work, we will be very much concerned with elements of $H_2(X)$ and $H^2(X)$. Therefore, it is useful to use identifications as in \cite[Section~2]{curvenbhd2}. For a simple root $\beta\in\Delta\setminus\Delta_P$, we will always identify the Schubert cycles $\sigma(s_\beta)$ with $\beta^\vee+\mathbb{Z}\Delta_P^\vee\in\mathbb{Z}\Delta^\vee/\mathbb{Z}\Delta_P^\vee$ and $\sigma_{s_\beta}$ with the fundamental weight $\omega_\beta$. Using these identification, we will simply write Equation~\eqref{eq:degdecomp} as
\begin{equation}
\label{eq:identification}
H_2(X)=\mathbb{Z}\Delta^\vee/\mathbb{Z}\Delta_P^\vee\text{ and }H^2(X)=\mathbb{Z}\{\omega_\beta\mid\beta\in\Delta\setminus\Delta_P\}\,.
\end{equation}
Under these identifications, the Poincar\'e pairing $H^2(X)\otimes H_2(X)\to\mathbb{Z}$ simply becomes the restriction of the $W$-invariant scalar product $(-,-)$ on $\mathbb{R}\Delta$. Note that $H_2(X)$ and $H^2(X)$ are naturally endowed with a partial order \enquote{$\leq$} which is given by comparing all the coefficients of the $\mathbb{Z}$-bases pointed out in Equation~\eqref{eq:degdecomp}.

\begin{conv}

If we speak about a degree, without further specification, then we always mean an effective class in $H_2(X)$. If we speak about a degree in $H_2(G/B)$, we mean an effective class in $H_2(G/B)$. In the later case, where the lattice might be different from $H_2(X)$, we explicitly mention it in our terminology. We reserve the term degree without specification for degrees in $H_2(X)$.

\end{conv}

\begin{notation}
\label{not:dalpha}

Let $\alpha\in R^+$. One degree associated to $\alpha$ will be ubiquitous in our discussion. By definition, the degree $d(\alpha)$ is given by the equation
$$
d(\alpha)=\alpha^\vee+\mathbb{Z}\Delta_P^\vee\in H_2(X)\,.
$$
Note that the degree $d(\alpha)$ depends not only on $\alpha$ but also on $P$ although $P$ is not explicitly mentioned in the notation $d(\alpha)$. No confusion will arise from this sloppiness since we refer always to one and the same parabolic subgroup $P$ which is fixed throughout the discussion. 
%We will never use the notation $d(\alpha)$ with respect to a different parabolic subgroup $Q$ but rather write the full expression $\alpha^\vee+\mathbb{Z}\Delta_Q^\vee$ if we need to refer to this degree depending on $Q$.
The geometric meaning of the degree $d(\alpha)$ will become clear later in the context of irreducible $T$-invariant curves (cf. Notation~\ref{not:Calpha}).

\end{notation}

\begin{notation}

We denote by $c_1(X)$ the first Chern class of the tangent bundle on $X$. The Chern class $c_1(X)$ is an element of $H^2(X)$ which will be often in use in this paper. For explicit computations, it is useful to have a description of $c_1(X)$ in terms of the root system. According to \cite[Lemma~3.5]{fulton} and the identifications made in Equation~\eqref{eq:identification}, we have $c_1(X)=\sum_{\gamma\in R^+\setminus R_P^+}\gamma$. In particular, we have $c_1(G/B)=\sum_{\gamma\in R^+}\gamma$.

\end{notation}

\section{Preliminaries}
\label{sec:preliminaries}

In this section, we set up the preliminaries which are used throughout the rest of the text. In particular, this section completes some aspects of the theory developed in \cite{minimaldegrees}. These aspects concern the properties of degrees with pairwise totally disjoint extended support (cf. Subsection~\ref{subsec:supports}). This situation was of no relevance for the analysis in \cite{minimaldegrees} but becomes increasingly important for the purpose of this paper.

It should be clear that many notions and ideas we use originally go back to \cite{curvenbhd2}. We rely in various ways on the theory of curve neighborhoods developed in \cite{curvenbhd2}. We review everything we need from this theory in Subsection~\ref{subsec:nbhd}.

%{\color{black} I want to make clear, once and for all, that many notions and ideas go back to \cite{curvenbhd2}.}

\subsection{Curve neighborhoods}
\label{subsec:nbhd}

In this subsection, we review everything we need from the theory of curve neighborhood due to \cite{curvenbhd2}. For examples, proofs and more informations, we encourage the reader to look at \cite{curvenbhd2}.

%{\color{black} For examples, we refer to \cite{curvenbhd2} - once and for all.}

\begin{defn}[{\cite[Section~4.2]{curvenbhd2}}]

Let $d$ be a degree. The maximal elements of the set $\{\alpha\in R^+\setminus R_P^+\mid d(\alpha)\leq d\}$ are called maximal roots of $d$. A sequence of roots $(\alpha_1,\ldots,\alpha_r)$ is called a greedy decomposition of $d$ if $\alpha_1$ is a maximal root of $d$ and $(\alpha_2,\ldots,\alpha_r)$ is a greedy decomposition of $d-d(\alpha_1)$. The empty sequence is the unique greedy decomposition of $0$.

\end{defn}

\begin{conv}

Let $d$ be a degree. Then the greedy decomposition of $d$ is unique up to reordering. This was proved in \cite[Section~4]{curvenbhd2}. This result plays an important role for the question of well-definedness of other objects described in terms of greedy decompositions. We will use it several times in this work. Although it is a non-trivial result due to \cite{curvenbhd2}, we will often use it without explicitly referring to the authors. The last sentence concerns in particular Section~\ref{sec:gencascade}.

\end{conv}

\begin{defn}[{\cite[Section~4.2]{curvenbhd2}}]
\label{def:pcosmall}

A root $\alpha\in R^+\setminus R_P^+$ is called $P$-cosmall if $\alpha$ is a maximal root of $d(\alpha)$.

\end{defn}

\begin{thm}
\label{thm:pcosmall}

Let $\alpha$ be a $P$-cosmall root. Then we have $(\alpha,\gamma)=0$ for all $\gamma\in R_P^+\setminus I(s_\alpha)$.

\end{thm}

\begin{proof}

We first prove a seemingly weaker statement.

\begin{proof}[Claim: We have $s_\alpha(R_P^+\setminus I(s_\alpha))=R_P^+\setminus I(s_\alpha)$]\renewcommand{\qedsymbol}{$\triangle$}

Indeed, by \cite[Theorem~6.1(c)]{curvenbhd2}, we know that $s_\alpha$ maps $R_P^+$ to the complement of $R^+\setminus R_P^+$ in $R$, i.e. to $R^-\cup R_P$. On the other hand $R^+\setminus I(s_\alpha)$ is mapped to $R^+$ under $s_\alpha$. Altogether, it follows that $s_\alpha(R_P^+\setminus I(s_\alpha))\subseteq(R^-\cup R_P)\cap R^+=R_P^+$. If a root $\gamma$ satisfies $s_\alpha(\gamma)\in I(s_\alpha)$, then we necessarily have $\gamma\in R^-$ (by applying $s_\alpha$ to $s_\alpha(\gamma)$). Thus, we even have $s_\alpha(R_P^+\setminus I(s_\alpha))\subseteq R_P^+\setminus I(s_\alpha)$. This suffices to prove the claim.
\end{proof}

\begin{proof}[Claim: We have $\sum_{\gamma\in R_P^+\setminus I(s_\alpha)}(\gamma,\alpha^\vee)=0$]\renewcommand{\qedsymbol}{$\triangle$}

Indeed, since $(-,-)$ is $W$-invariant, we know that $(\gamma,\alpha^\vee)=-(s_\alpha(\gamma),\alpha^\vee)$. The previous claim therefore yields that the sum the claim speaks about is equal to its negative. Hence, we find the desired vanishing.\footnote{It is actually easy to work out a different proof of this claim without using the previous claim. One can simply apply Lemma~\ref{lem:prep4} to the degree $d=d(\alpha)$}
\end{proof}

Since $\alpha\in R^+\setminus R_P^+$ by definition, all $\gamma\in R_P^+\setminus I(s_\alpha)$ satisfy $(\gamma,\alpha^\vee)\leq 0$ (otherwise we have $s_\alpha(\gamma)<0$). The previous claim shows that none of these inequalities can be strict. In other words, the theorem follows.
\end{proof}

\begin{defn}[{\cite[Section~3]{curvenbhd2}}]

Let $u,v\in W$ and $\beta\in\Delta$. Then we define the Hecke product of $u$ and $s_\beta$ by
$$
u\cdot s_\beta=\begin{cases}
us_\beta & \text{if }us_\beta\succ u\\
u & \text{if }us_\beta\prec u\,.
\end{cases}
$$
Let $v=s_{\beta_1}\cdots s_{\beta_l}$ be any reduced expression for $v$. Then we define the Hecke product of $u$ and $v$ by
$$
u\cdot v=u\cdot s_{\beta_1}\cdot\ldots\cdot s_{\beta_l}\,.
$$
That the expression $u\cdot v$ is well-defined (independent of the choice of the reduced expression for $v$) is proved in detail in \cite[Section~3]{curvenbhd2}.

\end{defn}

The reader finds a list of the most important properties of the Hecke product in \cite[Proposition~3.1]{curvenbhd2}. Among other properties, this list contains the associativity of the Hecke product (\cite[Proposition~3.1(a)]{minimaldegrees}).
%We will use these properties only sporadically. 
Mostly, we need the Hecke product to define for any degree $d$ a Weyl group element $z_d^P$ which captures the geometric properties of $d$.

\begin{defn}[{\cite[Section~4.2]{curvenbhd2}}]
\label{def:zdp}

Let $d$ be a degree. Let $(\alpha_1,\ldots,\alpha_r)$ be a greedy decomposition of $d$. Then we define an element $z_d^P\in W$ by the following equation
$$
z_d^Pw_P=s_{\alpha_1}\cdot\ldots\cdot s_{\alpha_r}\cdot w_P\,.
$$
%By Proposition~\ref{prop:hecke}(\ref{item:wpodot})
It is easy to see that $z_d^P$ is the minimal representative in $z_d^PW_P$ (cf. \cite[Proposition~2.4(8)]{minimaldegrees}). Well-definedness questions of the element $z_d^P$ (independence of the choice of the greedy decomposition of $d$) are discussed in detail in \cite[Section~4, in particular Definition~4.6]{curvenbhd2}.

\end{defn}

\begin{rem}

Most or all intuitions concerning the element $z_d^P$ for a degree $d$ come from its geometric meaning which is illuminated in \cite[Theorem~5.1]{curvenbhd2}. This theorem says among other things that the degree $d$ curve neighborhood of the zero dimensional Schubert variety $X_1$ is itself a Schubert variety parametrized by the element $z_d^P$. We will use \cite[Theorem~5.1]{curvenbhd2} in this weak form precisely once in this paper, namely in the proof of Theorem~\ref{thm:main}.

\end{rem}

\subsection{Supports}
\label{subsec:supports}

This subsection is about various kinds of notions of supports. Most notably, we recapitulate the notion of the extended support of a degree which was first introduced in \cite[Subsection~3.1]{minimaldegrees}. It turns out that the extended support is the right way to extend the naive support with the proper amount of simple roots in $\Delta_P$ to get a useful notion. In this way, we can formulate a disjointness assumption (namely that the extended supports of two degrees are totally disjoint\footnote{We will recall the meaning of \enquote{totally disjoint} in Definition~\ref{def:totallydisjoint}}) which enables us to prove certain addition theorems in Subsection~\ref{subsec:addition} (see for example Theorem~\ref{thm:additionminimaldegrees}).  

\todo[inline,color=green]{Mention that for a connected degree $d$, we have $\widetilde{\Delta}(d)=\Delta(\alpha(d))$ -- maybe as a fact together with trivial properties from February 3. Maybe I don't give it a number because I don't intend to refer to it. I refer directly to \cite{minimaldegrees}.}

\todo[inline,color=green]{For supports and Hecke products, I want to refer to the literature where the reader finds extensive computation / calculation rules\ldots For supports, I already did this. For Hecke products, I will do this in the section on curve neighborhoods. Also, more information on local notions in \cite{minimaldegrees} -- say this in the relevant subsection.}

\begin{notation}
\label{not:support}

Let $\alpha\in R^+$. Then we denote by $\Delta(\alpha)$  the support of $\alpha$, i.e. the set of simple roots $\beta\in\Delta$ such that $\beta\leq\alpha$. 

\end{notation}

\begin{notation}

Let $\alpha\in R^+$. Then we denote by $\Delta_\alpha^\circ$ the set of all simple roots which are orthogonal to $\alpha$; in formulas $\Delta_\alpha^\circ=\{\beta\in\Delta\mid(\alpha,\beta)=0\}$.

\end{notation}

\begin{defn}[{\cite[Subsection~3.1]{minimaldegrees}}]

Let $d$ be a degree. We define the naive support of $d$ to be the set $\Delta(d)$ of all simple roots $\beta\in\Delta\setminus\Delta_P$ such that $(\omega_\beta,d)>0$. We define the extended support of $d$ to be the set $\widetilde{\Delta}$ defined as the union $\widetilde{\Delta}(d)=\bigcup_{i=1}^r\Delta(\alpha_i)$ where $(\alpha_1,\ldots,\alpha_r)$ is a greedy decomposition of $d$. The extended support is clearly well-defined since the greedy decomposition is unique up to reordering. 

\end{defn}

\begin{rem}

Let $d$ be a degree. We have the following trivial relations between the naive and the extended support:
$$
\Delta(d)=\widetilde{\Delta}(d)\setminus\Delta_P\subseteq\widetilde{\Delta}(d)\subseteq\Delta(d)\cup\Delta_P=\widetilde{\Delta}(d)\cup\Delta_P\,.
$$
In particular, for a degree $e\in H_2(G/B)$, we have $\widetilde{\Delta}(e)=\Delta(e)$.

\end{rem}

\begin{defn}[{\cite[Definition~3.15]{minimaldegrees}}]

We say that a degree $d$ is a connected degree if $\widetilde\Delta(d)$ is a connected subset of the Dynkin diagram. We say that a degree $d$ is a disconnected degree if $\widetilde\Delta(d)$ is a disconnected subset of the Dynkin diagram.

\end{defn}

\begin{notation}
\label{not:alphad}

Let $d\neq 0$ be a connected degree. Then the first entry of a greedy decomposition of $d$ is uniquely determined by $d$ -- does not depend on the choice of the greedy decomposition of $d$ (cf. \cite[Proposition~3.16]{minimaldegrees}). We denote by $\alpha(d)$ this unique first entry of a greedy decomposition of $d$. With this notation, we have $\widetilde{\Delta}(d)=\Delta(\alpha(d))$ by \cite[loc. cit.]{minimaldegrees}.

\end{notation}

\begin{conv}

As we consider the empty set as a connected subset of the Dynkin diagram, the degree $d=0$ is considered as a connected degree with empty greedy decomposition. In particular, there does not exist a unique first entry $\alpha(d)$ of the degree $d=0$. We mentioned this exception once in Notation~\ref{not:alphad}. To avoid trivial considerations, we will from now on tacitly assume that a connected degree $d$ is nonzero whenever we speak about $\alpha(d)$. The reader can convince himself that the case $d=0$ can be treated in a trivial way in all proofs we do: All statements about the connected degree $d=0$ are obvious right in the beginning, although we will not explicitly say this each time.

\end{conv}

\begin{defn}

Let $w\in W$. We define the support of $w$ to be the set $\Delta(w)$ of all simple roots $\beta\in\Delta$ such that $s_\beta\preceq w$.

\end{defn}

% The definition of the support of a Weyl group element somehow already appears in \cite{humphreys3}, but I cannot find the exact reference now, maybe because I have an online version of the book available different from the printed version I used previously. 

The reader finds a list of the most important properties of the support of a Weyl group element in \cite[Proposition~3.17]{minimaldegrees}. For later use, we add one further property which is subject to Lemma~\ref{lem:support}.

\begin{lem}
\label{lem:support}

Let $u,v\in W$ such that $\Delta(u)$ and $\Delta(v)$ are disjoint. Then we have $u\cdot v=uv$.

\end{lem}

\begin{proof}

Let $Q$ be the parabolic subgroup of $G$ such that $\Delta_Q=\Delta(u)$. Let $Q'$ be the parabolic subgroup of $G$ such that $\Delta_{Q'}=\Delta(v)$. By definition, we clearly have $u\in W_Q$ and $v\in W_{Q'}$. We also have $v^{-1}\in W_{Q'}$ since $\Delta(v^{-1})=\Delta(v)$. By \cite[5.5, Theorem~(b)]{humphreys3}, it follows that $I(u)\subseteq R_Q^+$ and $I(v^{-1})\subseteq R_{Q'}^+$. The assumption $\Delta(u)\cap\Delta(v)=\emptyset$ implies that $R_{Q}^+\cap R_{Q'}^+=\emptyset$ and consequently that $I(u)\cap I(v^{-1})=\emptyset$. The statement now follows from \cite[Proposition~3.2: $(c)\Leftrightarrow(e)$]{curvenbhd2}.
\end{proof}

\subsection{Local notions}

Local notions were studied intensively in \cite[Section~6]{minimaldegrees}. In particular, it was shown in \cite[Theorem~6.10]{minimaldegrees} that minimal degrees in quantum products behave well under \enquote{localization}. In this work, we will only speak about locally high roots and about the root subsystem $R(\varphi)$ of $R$ associated to a positive root $\varphi$. We introduce these notions in this subsection.

\begin{notation}

Let $\varphi$ be a positive root. We denote by $R(\varphi)$ the root subsystem of $R$ generated by $\Delta(\varphi)$. Since $\Delta(\varphi)$ is a connected subset of the Dynkin diagram, the root system $R(\varphi)$ is always irreducible.

\end{notation}

\begin{defn}[{\cite[Section~1]{kostant}}]

Let $\varphi\in R^+$. We say $\varphi$ is locally high if $\varphi$ is the highest root of $R(\varphi)$.

\end{defn}

\begin{defn}[{\cite[VI, 1, 3]{bourbaki_roots}}]

We say that two roots $\alpha$ and $\alpha'$ are strongly orthogonal if and only if $\alpha\pm\alpha'\notin R\cup\{0\}$.

\end{defn}

\begin{defn}[{\cite[Section~1]{kostant}}]
\label{def:totallydisjoint}

Two subsets of roots $S$ and $S'$ are called totally disjoint if every element of $S$ is strongly orthogonal to every element of $S'$.

\end{defn}

\begin{fact}
\label{fact:stronglyorthogonal}

Let $S\subseteq\Delta$. Let $\varphi_1,\ldots,\varphi_k$ be locally high roots such that 
$
\Delta(\varphi_1),\ldots,\Delta(\varphi_k)
$
are the distinct connected components of $S$. Then $R(\varphi_1),\ldots,R(\varphi_k)$ are pairwise totally disjoint.

\end{fact}

\begin{proof}

It is very easy to supply a proof of Fact~\ref{fact:stronglyorthogonal}. We leave the details to the reader.
\end{proof}

\begin{fact}
\label{fact:lochighBcosmall}

Every locally high root is $B$-cosmall.
%In particular, simple roots are $B$-cosmall, as it was already pointed out in \cite{curvenbhd2}. The argument that $\alpha$ is simply long does not work out.

\end{fact}

\begin{proof}

Let $\varphi$ be a locally high root. Let $\alpha\in R^+$ be a root such that $\varphi\leq\alpha$ and $\alpha^\vee\leq\varphi^\vee$. In order to prove that $\varphi$ is $B$-cosmall, we have to show that $\alpha=\varphi$. The inequality $\alpha^\vee\leq\varphi^\vee$ implies that $\Delta(\alpha)\subseteq\Delta(\varphi)$ and thus $\alpha\in R(\varphi)$. By definition, the root $\varphi$ is the highest root of $R(\varphi)$. Consequently, we necessarily have $\alpha\leq\varphi$. In total, this implies that $\alpha=\varphi$ -- as required.
\end{proof}

\subsection{Connected components of a degree}

{\color{black} As we already introduced a notion of connected (and disconnected) degrees, it is plausible also to introduce a notion of connected components of a degree. We do so in this subsection. The notion of connected components of a degree is useful because many characteristics of a degree, such as its set of maximal roots (cf. Theorem~\ref{thm:maxroots}) or the number of its greedy decompositions (cf. Theorem~\ref{thm:nd}), are already determined by its connected components. 
%It will be used in many of our proofs and considerations.
}
%(cf. Theorem on $N_d$).}

\todo[inline,color=green]{It is probably sufficient to introduce the notion \enquote{connected degree} for degrees in $H_2(G/B)$ as I only need it there (maybe mention that there is a \enquote{right} generalization). Material on greedy decomposition, cosmall is needed before this.}

\begin{defn}
\label{def:connectedcomponents}

Let $d$ be a degree. Let $\varphi_1,\ldots,\varphi_k$ be locally high roots such that 
$$
\Delta(\varphi_1),\ldots,\Delta(\varphi_k)
$$
are the distinct connected components of $\widetilde{\Delta}(d)$. Let $(\alpha_1,\ldots,\alpha_r)$ be a greedy decomposition of $d$. Then we define the connected components of $d$ to be the degrees
$$
d_i=\sum_{1\leq j\leq r\colon\Delta(\alpha_j)\subseteq\Delta(\varphi_i)}d(\alpha_j)\text{ where }1\leq i\leq k\,.
$$
Since the greedy decomposition of $d$ is unique up to reordering, 
%(cf. \cite[Section~4.2]{curvenbhd2}), 
the connected components $d_1,\ldots,d_k$ of $d$ are clearly well-defined -- do not depend on the choice of the greedy decomposition but only on $d$.

\end{defn}

\begin{lem}
\label{lem:connectedcomponents1}

Let $d$ be a degree. Let $d_1,\ldots,d_k$ be the connected components of $d$. Then 
$$
\widetilde{\Delta}(d_1),\ldots,\widetilde{\Delta}(d_k)
$$ 
are the distinct connected components of $\widetilde{\Delta}(d)$, in particular each $d_i$ is a connected degree. Moreover, we have $d=\sum_{i=1}^k d_i$.

\end{lem}

\begin{proof}

Let $\varphi_1,\ldots,\varphi_k$ be locally high roots such that $\Delta(\varphi),\ldots,\Delta(\varphi_k)$ are the distinct connected components of $\widetilde{\Delta}(d)$. Let $(\alpha_1,\ldots,\alpha_r)$ be a greedy decomposition of $d$. It is clear that for every $1\leq j\leq r$ there exists a unique $1\leq i\leq k$ such that $\Delta(\alpha_j)\subseteq\Delta(\varphi_i)$. Hence, we find by %Definition~\ref{def:connectedcomponents} 
definition that $d=\sum_{i=1}^k d_i$. By \cite[Proposition~3.10(7)]{minimaldegrees} and by definition of $\widetilde{\Delta}(d_i)$, we have
\begin{equation}
\label{eq:supportinclusion}
\widetilde{\Delta}(d_i)=\bigcup_{1\leq j\leq r\colon\Delta(\alpha_j)\subseteq\Delta(\varphi_i)}\Delta(\alpha_j)\subseteq\Delta(\varphi_i)\text{ for all }1\leq i\leq k
\end{equation}
and consequently
%I use: (a) the definition of $\varphi_1,\ldots,\varphi_k$, (b) for every $j$ there exists (a unique) $i$ such that $\Delta(\alpha_j)\subseteq\Delta(\varphi_i)$.
\begin{equation}
\label{eq:supportunion}
\widetilde{\Delta}(d)=\coprod_{i=1}^k\widetilde{\Delta}(d_i)\,.
\end{equation}
If one of the inclusions in Equation~\eqref{eq:supportinclusion} is strict, it follows from Equation~\eqref{eq:supportunion} that
$$
\widetilde{\Delta}(d)\subsetneq\coprod_{i=1}^k\Delta(\varphi_i)=\widetilde{\Delta}(d)\,.
$$
Therefore, we conclude that $\widetilde{\Delta}(d_i)=\Delta(\varphi_i)$ for all $1\leq i\leq k$, in other words that 
$$
\widetilde{\Delta}(d_1),\ldots,\widetilde{\Delta}(d_k)
$$
are the distinct connected components of $\widetilde{\Delta}(d)$. Since each $\widetilde{\Delta}(d_i)$ is in particular connected, it follows that each $d_i$ is a connected degree.
%I use the definition of connected degrees.
\end{proof}

\todo[inline,color=green]{I need to introduce the index either in notation and conventions or already in the introduction or both. Note, I freely use the notation $c_1$ already in the section on chain cascades. I think I don't need to mention the combinatorial description of $c_1$ -- only later when I actually use it.

I mention the combinatorial description of $c_1$ already when I introduce this notation in notation and conventions (in the introduction $c_1$ does not occur and is therefore only there necessary). There I cite \cite[Lemma~3.5]{fulton} for this description, everywhere else I cite \cite[Equation~(3)]{curvenbhd2} for this description (for some reason). Actually, both references appear and that should be enough.}

\todo[inline,color=green]{connected degrees, connected components of degrees, minimal degree $d_X$, minimal degrees in $\Pi_B$ and $\Pi_P$ (we speak about a minimal degree in $\Pi_\bullet$ and always specify the precise set relative to $B$ or $P$ to avoid confusion), the $\theta$-sequence is totally ordered (and rightfully called like this), if $P$ is maximal, I want to introduce a convention about the use of $\prod$ and $\bigodot$, only if the members commute (e.g. because the roots are orthogonal), usage of $\Delta$ with respect to roots and degrees, inversion sets, $w_o$, $\Delta^\circ$, $\Delta_\alpha^\circ$, pt=class of a point=$\sigma_{w_o}$, $w_P$}

\todo[inline,color=green]{trivial statements like: each connected component of a degree is connected, mention formula for connected components as needed in the proof of Theorem~\ref{thm:gencascade}\eqref{item:totallyordered} (plus some formula in the vain of: union of supports of connected components is the original support $\star\colon$ I can cite the use of these formulas twice in the proof of Theorem~\ref{thm:gencascade}\eqref{item:totallyordered})}

\todo[inline,color=green]{first entry of a greedy decomposition of a connected degree $d$: $\alpha(d)$, recall again the notion strongly orthogonal (and, if needed, totally disjoint), mention the uniqueness of the greedy decomposition up to reordering directly after the definition -- I somehow assume that this is obvious whenever I compute $\mathcal{B}_{R,e}$, parabolic=standard parabolic: many conventions can be adopted (also the degree convention, cf. Lemma~\ref{lem:connectedcomponents1})}

\subsection{Minimal degrees in quantum products}

In this subsection we introduce the class of minimal degrees. We organize these degrees in a set $\Pi_P$ where the $P$ indicates the parabolic subgroup relative to which they are computed. Minimal degrees are central because they feature many important properties (most prominently orthogonality relations, cf. \cite[Section~8]{minimaldegrees}) which eventually lead to the quasi-homogeneity result which is subject to this paper.
%It will be clear from the introduction what quasi-homogeneity results are meant.
We will elaborate on the consequences of the properties of a minimal degree $e\in\Pi_B$ for the greedy decomposition of $e$ and the element $z_e^B$ in Section~\ref{sec:gencascade}.

Minimal degrees naturally arise in the context of quantum cohomology as minimal degrees of quantum product (cf. \cite{minimaldegrees}). In our exposition, we choose a completely combinatorial definition of minimal degrees which makes only use of the theory of curve neighborhoods surveyed in Subsection~\ref{subsec:nbhd}.
%this section before. 
We do so to keep the prerequisites as low as possible. Moreover, for all of our proofs the combinatorial approach is more suitable. The relation to quantum cohomology is mentioned only in a few remarks intended to the reader familiar with this theory and not needed to understand the purpose of this paper.

\todo[inline,color=green]{I want to make a remark in the introduction, directly referring to this subsection, so that a reader not familiar with quantum cohomology can still understand.}

\todo[inline,color=green]{There are many examples of minimal degrees: cosmall, exceptional, simple, sub-sequences of greedy decompositions of elements in $\Pi_P$, in particular sub-sequences of greedy decomposition of $d_X$ -- quite much, in view of the fact that $d_X$ can be explicitly computed as in \cite{minimaldegrees}. The addition theorem gives a further way to construct minimal degrees. Explain at least the two equivalent definitions of minimal degrees -- they are needed both.}

\begin{defn}
\label{def:minimaldegrees}

Let $d$ be a degree. We say that $d$ is a minimal degree if $d$ is a minimal element of the set
$$
\{d'\text{ a degree such that }z_d^P\preceq z_{d'}^P\}\,.
$$

\end{defn}

\begin{notation}

We denote by $\Pi_P$ the set of all minimal degrees. In particular, the set of all minimal degrees in $H_2(G/B)$ is denoted by $\Pi_B$.

\end{notation}

\begin{rem}[For the reader familiar with quantum cohomology]
\label{rem:quantum}

Let $(QH^*(X),\star)$ be the (small) quantum cohomology ring attached to $X$ as defined in \cite[Section~10]{fultonpan} . In terms of quantum cohomology, the set $\Pi_P$ can be described as follows:
\begin{equation}
\label{eq:quantum}
\Pi_P=\{d\text{ is a minimal degree in }\sigma_u\star\sigma_v\text{ for some }u,v\in W\}\,.
\end{equation}
The definition of a minimal degree in the quantum product of two Schubert cycles is given in \cite[Definition~5.14]{minimaldegrees}. The inclusion \enquote{$\subseteq$} in Equation~\eqref{eq:quantum} follows since a degree $d\in\Pi_P$ is a minimal degree in $\sigma_{z_d^P}\star\mathrm{pt}$ (cf. \cite[Definition~4.1, Theorem~5.15]{minimaldegrees}). The inclusion \enquote{$\supseteq$} in Equation~\eqref{eq:quantum} follows from \cite[Definition~4.1, Theorem~5.10, Theorem~5.15]{minimaldegrees}.

\end{rem}

{\color{black}\begin{notation}
\label{not:dx}

By \cite[Definition~4.1, Theorem~4.7]{minimaldegrees}, there exists a unique minimal element of the set
$$
\{d\text{ a degree such that }w_oW_P=z_d^PW_P\}\,.
$$
This unique minimal element is denoted by $d_X$. By definition, we clearly have $d_X\in\Pi_P$. In particular, we have $d_{G/B}\in\Pi_B$.

\end{notation}}

\begin{rem}

The minimal degree $d_X$ is the main object of study of \cite{minimaldegrees} and therefore very well understood. In particular, it can be explicitly computed (cf. \cite[Corollary~7.12]{minimaldegrees}).

\end{rem}

\begin{ex}
\label{ex:minimaldegrees}

We summarize examples of minimal degrees and ways to produce new minimal degrees out of existing ones. 

\begin{enumerate}

\item

Let $\alpha$ be a $P$-cosmall root. Then we have $d(\alpha)\in\Pi_P$ (\cite[Proposition~4.4(6)]{minimaldegrees}).

\item

Let $\beta$ be a simple root. Then we have $d(\beta)\in\Pi_P$ (\cite[Proposition~4.4(8)]{minimaldegrees}).

\item

Suppose that $R$ is simply laced. Then we have $d(\alpha)\in\Pi_P$ for all $\alpha\in R^+$ (\cite[Theorem~4.15]{minimaldegrees})

\item

Let $d\in\Pi_P$. Let $(\alpha_1,\ldots,\alpha_r)$ be a subsequence of a greedy decomposition of $d$. Then we have $\sum_{i=1}^r d(\alpha_i)\in\Pi_P$. This follows by repeated application of \cite[Proposition~4.4(9)]{minimaldegrees}.

\item

Let $d_1,\ldots,d_k\in\Pi_P$ be minimal degrees such that $\widetilde{\Delta}(d_1),\ldots,\widetilde{\Delta}(d_k)$ are pairwise totally disjoint. Then we have $\sum_{i=1}^k d_i\in\Pi_P$. We will prove this statement later on (cf. Theorem~\ref{thm:additionminimaldegrees}).

\item
\label{item:subsetsofcascade}

Let $(\alpha_1,\ldots,\alpha_r)$ be a subsequence of a greedy decomposition of $d_{G/B}$. Then we have $\sum_{i=1}^r d(\alpha_i)\in\Pi_P$. This follows from \cite[Theorem~7.14, Remark~8.4]{minimaldegrees} and Remark~\ref{rem:general}.

\end{enumerate}

\end{ex}

\begin{comment}

\begin{rem}

In view of the fact that greedy decompositions of $d_{G/B}$ are well-known and can be explicitly computed (they are specific orderings of the ordinary cascade of orthogonal roots, cf. Remark~\ref{rem:general}), Example~\ref{ex:minimaldegrees}\eqref{item:subsetsofcascade} gives a bunch of minimal degrees whose greedy decomposition has length strictly larger that one.

\end{rem}

\end{comment}

\subsection{Addition theorems}
\label{subsec:addition}

In this subsection we present various kinds of addition theorems. By an addition theorem, we mean a theorem which expresses properties of the sum $d_1+d_2$ of two degrees (or more generally of $k$ degrees) in terms of properties of the individual degrees $d_1$ and $d_2$. In most cases, we do so by assuming from the beginning that $\widetilde{\Delta}(d_1)$ and $\widetilde{\Delta}(d_2)$ are totally disjoint. In this way, we are for example able to describe all greedy decomposition of $d_1+d_2$ in terms of greedy decompositions of $d_1$ and $d_2$ whenever the extended supports are totally disjoint (cf. Theorem~\ref{thm:additiongreedy}, \ref{thm:additiongreedyconverse}). In particular, this gives us a way to reduce many problems, e.g. testing the minimality of a degree (cf. Theorem~\ref{thm:additionminimaldegrees} and Corollary~\ref{cor:additionminimaldegrees}), from an arbitrary degree $d$ to a connected degree by passing to the connected components of $d$. 

\todo[inline,color=green]{It turns out that I cited \cite[Proposition~3.10(7), 4.4(9) -- in precisely this block style]{minimaldegrees} unnecessarily often, e.g. if I want to show that the connected components of a degree in $\Pi_P$ lie in $\Pi_P$. I will introduce in this section a more general theorem (for arbitrary $P$ not only $B$) -- as a consequence of the addition theorem for minimal degrees. Afterwards, I should refer to this theorem -- and not to \cite{minimaldegrees} anymore.

I cannot find any unnecessary citations, except maybe in the proof of Fact~\ref{fact:altdef}. But I leave it like this since I think it is better. In the case modulo $B$ everything is \enquote{trivial} anyway.}

\begin{thm}
\label{thm:inclusionextendedsupport}

Let $d$ and $d'$ be two degrees such that $d\leq d'$. Then we have $\widetilde{\Delta}(d)\subseteq\widetilde{\Delta}(d')$ (and clearly also $\Delta(d)\subseteq\Delta(d')$).

\end{thm}

\begin{proof}

We first reduce the theorem to the case of a connected degree $d$. Assume that the assertion is true for all connected degrees $d$. Let $d_1,\ldots,d_k$ be the connected components of $d$. By Lemma~\ref{lem:connectedcomponents1} we have $d_i\leq d'$ for all $1\leq i\leq k$. The assumption implies $\widetilde{\Delta}(d_i)\subseteq\widetilde{\Delta}(d')$. Again, by Lemma~\ref{lem:connectedcomponents1}, we then find that $\widetilde{\Delta}(d)\subseteq\widetilde{\Delta}(d')$ (cf. Equation~\eqref{eq:supportunion}).

Without loss of generality, we may assume that $d$ is a connected degree. Let $\alpha=\alpha(d)$. Since $d(\alpha)\leq d\leq d'$, we can find a maximal root $\alpha'$ of $d'$ such that $\alpha\leq\alpha'$. By definition, $\alpha'$ occurs in a greedy decomposition of $d'$. Thus, we have $\Delta(\alpha')\subseteq\widetilde{\Delta}(d')$. By \cite[Proposition~3.16]{minimaldegrees}, we have $\widetilde{\Delta}(d)=\Delta(\alpha)$. Therefore the desired inclusion follows from the trivial inclusion $\Delta(\alpha)\subseteq\Delta(\alpha')$.
%Strictly speaking, it is not necessary to use connected components for this proof. But using them, it seems more elegant to me.
\end{proof}

\begin{thm}[Addition theorem for extended supports]
\label{thm:additionextendedsupport}

Let $d$ and $d'$ be two degrees. Then we have $\widetilde{\Delta}(d+d')=\widetilde{\Delta}(d)\cup\widetilde{\Delta}(d')$.

\end{thm}

\begin{rem}

The corresponding addition theorem for naive supports is absolutely trivial and also follows immediately from Theorem~\ref{thm:additionextendedsupport} by subtracting $\Delta_P$. Therefore, we can think of Theorem~\ref{thm:additionextendedsupport} as a more refined statement. 

\end{rem}

\begin{proof}[Proof of Theorem~\ref{thm:additionextendedsupport}]

Let $d$ and $d'$ be as in the statement. Let $\widetilde{\Delta}=\widetilde{\Delta}(d)\cup\widetilde{\Delta}(d')$ for short. By Theorem~\ref{thm:inclusionextendedsupport} we have
$$
\widetilde{\Delta}\subseteq\widetilde{\Delta}(d+d')\subseteq\widetilde{\Delta}\cup\Delta_P\,.
$$
Let $S=\widetilde{\Delta}(d+d')\setminus\widetilde{\Delta}$. By the previous inclusions, we clearly have $S\subseteq\Delta_P$ and $\widetilde{\Delta}(d+d')=S\amalg\widetilde{\Delta}$.

\begin{proof}[Claim: $S$ and $\widetilde{\Delta}$ are totally disjoint]\renewcommand{\qedsymbol}{$\triangle$}

Let $\beta\in S$ and $\beta'\in\widetilde{\Delta}$. We have to show that $\beta$ and $\beta'$ are strongly orthogonal. Since clearly $\beta-\beta'\notin R$, it suffices to show that $(\beta,\beta')=0$ (\cite[Lemma~7.3]{minimaldegrees}). Suppose for a contradiction that $(\beta,\beta')\neq 0$. 
%Let $\delta=s_{\beta}(\beta')$. Since $\delta$ is a root and $(\beta,\beta')\neq 0$, we must have $(\beta,\beta')<0$ and thus $\delta>\beta'$.
%One can probably also argue with root strings to show that $(\beta,\beta')<0$ -- but that is an overkill as it is seen so easily directly. 
Since $\beta\neq\beta'$, we clearly have $(\beta,\beta')<0$. (All non-diagonal entries of a Cartan matrix are non-positive.) 

By replacing $d$ with $d'$ if necessary, we may assume that $\beta'\in\widetilde{\Delta}(d)$. This means that there exists a root $\alpha$ which occurs in a greedy decomposition of $d$ such that $\beta'\in\Delta(\alpha)$. 
%By definition we have $\alpha'\in R^+\setminus R_P^+$ and $d(\alpha')\leq d$. Therefore we can find a maximal root $\alpha$ of $d$ such that $\beta'\leq\alpha'\leq\alpha$. Since $\alpha$ is a maximal root of $d$, it is in particular $P$-cosmall.
Since $\alpha$ occurs in a greedy decomposition of $d$, it is in particular $P$-cosmall. Since $\beta\in\Delta_P$, it follows from \cite[Corollary~3.19]{minimaldegrees} that $(\alpha,\beta)\geq 0$. 

Let $\alpha=\sum_{\mu\in\Delta(\alpha)}n_\mu\mu$ be the expression of $\alpha$ as a linear combination of simple roots. Since $\alpha$ is a positive root (we even have $\alpha\in R^+\setminus R_P^+$), we know that $n_\mu>0$ for all $\mu\in\Delta(\alpha)$. By definition of $\alpha$, we have $\beta'\in\Delta(\alpha)$. Hence, we can write 
$$
(\alpha,\beta)=\sum_{\mu\in\Delta(\alpha)\setminus\{\beta'\}}n_\mu(\mu,\beta)+n_{\beta'}(\beta,\beta')\,.
$$
We already figured out that $(\beta,\beta')<0$ and $(\alpha,\beta)\geq 0$. In view of these inequalities and the previous displayed equation, we find that
$$
\sum_{\mu\in\Delta(\alpha)\setminus\{\beta'\}}n_\mu(\mu,\beta)\geq-n_{\beta'}(\beta,\beta')>0\,.
$$
Therefore, there must exist a $\mu\in\Delta(\alpha)\setminus\{\beta'\}$ such that $(\mu,\beta)>0$. Since all non-diagonal entries of a Cartan matrix are non-positive, it follows that $\mu=\beta$ and thus $\beta\in\Delta(\alpha)\subseteq\widetilde{\Delta}(d)\subseteq\widetilde{\Delta}$. This is a contradiction since by definition $\beta\in S$ and $\beta\notin\widetilde{\Delta}$. This proves the claim.
\end{proof}

Let $\alpha$ be a root which occurs in a greedy decomposition of $d+d'$. By definition, we have $\Delta(\alpha)\subseteq\widetilde{\Delta}(d+d')$ and thus
$$
\Delta(\alpha)=(\Delta(\alpha)\cap S)\amalg(\Delta(\alpha)\cap\widetilde{\Delta})\,.
$$
Since $\Delta(\alpha)$ is necessarily connected, the previous claim shows that either $\Delta(\alpha)\subseteq S$ or $\Delta(\alpha)\subseteq\widetilde{\Delta}$. If $\Delta(\alpha)\subseteq S$, then we have $\alpha\in R_P^+$ since $S\subseteq\Delta_P$ -- a contradiction since $\alpha\in R^+\setminus R_P^+$ by definition. Therefore, we conclude that $\Delta(\alpha)\subseteq\widetilde{\Delta}$. Since $\alpha$ was an arbitrary entry in a greedy decomposition of $d+d'$, the definition of the extended support shows that $\widetilde{\Delta}(d+d')\subseteq\widetilde{\Delta}$. This means that $S=\emptyset$ and $\widetilde{\Delta}(d+d')=\widetilde{\Delta}$ -- as claimed.
\end{proof}

\begin{thm}[Addition theorem for greedy decompositions]
\label{thm:additiongreedy}

Let $d_1,\ldots,d_k$ be degrees such that $\widetilde{\Delta}(d_1),\ldots,\widetilde{\Delta}(d_k)$ are pairwise totally disjoint. Let $(\alpha_1^i,\ldots,\alpha_{r_i}^i)$ be a greedy decomposition of $d_i$ for all $1\leq i\leq k$. Let $r=\sum_{i=1}^k r_i$ for short. Let $(\alpha_1,\ldots,\alpha_r)$ be a sequence of roots such that $(\alpha_1^i,\ldots,\alpha_{r_i}^i)$ is a subsequence of $(\alpha_1,\ldots,\alpha_r)$ for all $1\leq i\leq k$. Then the sequence $(\alpha_1,\ldots,\alpha_r)$ is a greedy decomposition of $\sum_{i=1}^k d_i$.

\end{thm}

\begin{proof}

We first assume that $k=2$ and prove the theorem for that case. We do so by induction on $r$. The case where $r_1=0$ or $r_2=0$ is obvious. Assume that $r_1>0$ and $r_2>0$. This means that $r\geq 2$. Assume further that the statement is known for sequences of length strictly less than $r$. By renaming the indices if necessary ($i\in\{1,2\}$), we may assume that $\alpha_1=\alpha_1^1$. Let us write $\alpha=\alpha_1=\alpha_1^1$ and $d=d_1+d_2$ for short.

\begin{proof}[Claim: $\alpha$ is a maximal root of $d$]\renewcommand{\qedsymbol}{$\triangle$}

Let $\alpha'$ be a maximal root of $d$ such that $\alpha\leq\alpha'$. We can clearly choose such a root $\alpha'$ since $d(\alpha)\leq d_1\leq d$ and $\alpha\in R^+\setminus R_P^+$ by definition. By Theorem~\ref{thm:additionextendedsupport}, we have 
$$
\Delta(\alpha')\subseteq\widetilde{\Delta}(d)=\widetilde{\Delta}(d_1)\amalg\widetilde{\Delta}(d_2)\,.
$$
Since $\widetilde{\Delta}(d_1)$ and $\widetilde{\Delta}(d_2)$ are totally disjoint and since $\Delta(\alpha')$ is connected, it follows that ether $\Delta(\alpha')\subseteq\widetilde{\Delta}(d_1)$ or $\Delta(\alpha')\subseteq\widetilde{\Delta}(d_2)$. Since $\alpha\leq\alpha'$, we have $\Delta(\alpha)\subseteq\Delta(\alpha')$. Thus the second case (i.e. $\Delta(\alpha')\subseteq\widetilde{\Delta}(d_2)$) implies that $\Delta(\alpha)\subseteq\widetilde{\Delta}(d_1)\cap\widetilde{\Delta}(d_2)$ which contradicts the assumption that $\widetilde{\Delta}(d_1)$ and $\widetilde{\Delta}(d_2)$ are totally disjoint. Therefore, we conclude that $\Delta(\alpha')\subseteq\widetilde{\Delta}(d_1)$. In view of this inclusion, the inequality $d(\alpha')\leq d$ implies that $d(\alpha')\leq d_1$. By definition, $\alpha$ is a maximal root of $d_1$. This means that we must have $\alpha=\alpha'$ (since $\alpha\leq\alpha'$ and $d(\alpha')\leq d_1$). By the choice of $\alpha'$, we have now proved that $\alpha$ is a maximal root of $d$ -- as claimed.
\end{proof}

In view of the claim, $(\alpha_1,\ldots,\alpha_r)$ is a greedy decomposition of $d$ if and only if $(\alpha_2,\ldots,\alpha_r)$ is a greedy decomposition of $d-d(\alpha)=(d_1-d(\alpha))+d_2$. We now apply the induction hypothesis to the sequence $(\alpha_2,\ldots,\alpha_r)$ to see that the latter statement is true. (This is possible since $(\alpha_2^1,\ldots,\alpha_{r_1}^1)$ is a greedy decomposition of $d_1-d(\alpha)$, $(\alpha_1^2,\ldots,\alpha_{r_2}^2)$ is a greedy decomposition of $d_2$, $(\alpha_2^1,\ldots,\alpha_{r_1}^1)$ and $(\alpha_1^2,\ldots,\alpha_{r_2}^2)$ are subsequences of $(\alpha_2,\ldots,\alpha_r)$, and since $\widetilde{\Delta}(d_1-d(\alpha))\subseteq\widetilde{\Delta}(d_1)$ and $\widetilde{\Delta}(d_2)$ are totally disjoint.)

We now prove the statement for arbitrary $k$ by induction. The case $k=1$ is obvious. The case $k=2$ was treated above. Assume that $k>2$ and that the statement is known for all positive integers strictly less than $k$.

By definition and assumption, it is clear that for all $1\leq j\leq r$ there exists a unique $1\leq i\leq k$ such that $\Delta(\alpha_j)\subseteq\widetilde{\Delta}(d_i)$. Therefore, we can form a maximal subsequence $(\mu_1,\ldots,\mu_{r'})$ of $(\alpha_1,\ldots,\alpha_r)$ of length $r'=r_1+r_2$ such that 
\begin{equation}
\label{eq:ih}
\Delta(\mu_j)\subseteq\widetilde{\Delta}(d_1+d_2)=\widetilde{\Delta}(d_1)\amalg\widetilde{\Delta}(d_2)\text{ for all }1\leq j\leq r'\,.
\end{equation}
Here, the last equality follows from Theorem~\ref{thm:additionextendedsupport} and the totally disjointness assumption. Moreover, by definition, it is clear that $(\alpha_1^1,\ldots,\alpha_{r_1}^1)$ and $(\alpha_1^2,\ldots,\alpha_{r_2}^2)$ are subsequences of $(\mu_1,\ldots,\mu_{r'})$. The induction hypothesis applied to the case $k=2$ implies that $(\mu_1,\ldots,\mu_{r'})$ is a greedy decomposition of $d_1+d_2$. Finally, the induction hypothesis applied to the $k-1$ degrees $d_1+d_2,d_3,\ldots,d_k$ and the greedy decompositions $(\mu_1,\ldots,\mu_{r'})$ of $d_1+d_2$ and $(\alpha_1^i,\ldots,\alpha_{r_i}^i)$ of $d_i$ for all $3\leq i\leq k$ (clearly, $(\mu_1,\ldots,\mu_{r'})$ is a subsequence of $(\alpha_1,\ldots,\alpha_r)$ and $\widetilde{\Delta}(d_1+d_2),\widetilde{\Delta}(d_3),\ldots,\widetilde{\Delta}(d_k)$ are pairwise totally disjoint, cf. Equation~\eqref{eq:ih} and the initial totally disjointness assumption) implies that $(\alpha_1,\ldots,\alpha_r)$ is a greedy decomposition of $\sum_{i=1}^k d_i$. This is all we wanted to prove.
\end{proof}

\begin{comment}

\begin{thm}[Addition theorem for greedy decompositions]

Let $d$ and $d'$ be two degrees such that $\widetilde{\Delta}(d)$ and $\widetilde{\Delta}(d')$ are totally disjoint. Let $(\alpha_1,\ldots,\alpha_r)$ be a greedy decomposition of $d$. Let $(\alpha_1',\ldots,\alpha_s')$ be a greedy decomposition of $d'$. Then $(\alpha_1,\ldots,\alpha_r,\alpha_1',\ldots,\alpha_s')$ is a greedy decomposition of $d+d'$.
%in particular: $\widetilde{\Delta}(d+d')=\widetilde{\Delta}(d)\amalg\widetilde{\Delta}(d')$.

\end{thm}

\end{comment}

\begin{thm}[Converse of Theorem~\ref{thm:additiongreedy}]
\label{thm:additiongreedyconverse}

Let $d_1,\ldots,d_k$ be degrees such that $$\widetilde{\Delta}(d_1),\ldots,\widetilde{\Delta}(d_k)$$ are pairwise totally disjoint. Let $(\alpha_1,\ldots,\alpha_r)$ be a greedy decomposition of $\sum_{i=1}^k d_i$. Then there exist unique greedy decompositions $(\alpha_1^i,\ldots,\alpha_{r_i}^i)$ of $d_i$ for all $1\leq i\leq k$ such that $r=\sum_{i=1}^k r_i$ and such that $(\alpha_1^i,\ldots,\alpha_{r_i}^i)$ is a subsequence of $(\alpha_1,\ldots,\alpha_r)$ for all $1\leq i\leq k$. In other words, every greedy decomposition of $\sum_{i=1}^k d_i$ can be constructed in the way described in Theorem~\ref{thm:additiongreedy} and uniquely determines the greedy decompositions of $d_i$ for all $1\leq i\leq k$ out of which it is constructed.\footnote{The other way round, it should be clear that $(\alpha_1,\ldots,\alpha_r)$ is \emph{not} uniquely determined by $(\alpha_1^i,\ldots,\alpha_{r_i}^i)$ for all $1\leq i\leq k$.}

\end{thm}

\begin{proof}

Let $d=\sum_{i=1}^k d_i$ for short. By Theorem~\ref{thm:additionextendedsupport} and assumption, it is clear that we have
$$
\widetilde{\Delta}(d)=\coprod_{i=1}^k\widetilde{\Delta}(d_i)\,.
$$
Again by the totally disjointness assumption, it is therefore clear that for all $1\leq j\leq r$ there exists a unique $1\leq i\leq k$ such that $\Delta(\alpha_j)\subseteq\widetilde{\Delta}(d_i)$. It follows that
$$
d_i=\sum_{1\leq j\leq r\colon\Delta(\alpha_j)\subseteq\widetilde{\Delta}(d_i)}d(\alpha_j)\,.
$$
For all $1\leq i\leq k$, we are now forced to define $(\alpha_1^i,\ldots,\alpha_{r_i}^i)$ as the maximal subsequence of $(\alpha_1,\ldots,\alpha_r)$ such that $\Delta(\alpha_j^i)\subseteq\widetilde{\Delta}(d_i)$ for all $1\leq j\leq r_i$. This already proves the uniqueness part of the statement. By \cite[Proposition~3.10(7)]{minimaldegrees} it also follows that $(\alpha_1^i,\ldots,\alpha_{r_i}^i)$ is indeed a greedy decomposition of $d_i$ for all $1\leq i\leq k$. By definition, we also have $r=\sum_{i=1}^k r_i$.
%since $r_i=\card\{1\leq j\leq r\mid\Delta(\alpha_j)\subseteq\widetilde{\Delta}(d_i)\}$ etc.
\end{proof}

\begin{thm}
\label{thm:maxroots}

Let $d$ be a degree. Let $d_1,\ldots,d_k$ be the connected components of $d$. Then there exist precisely $k$ maximal roots of $d$, namely $\alpha(d_1),\ldots,\alpha(d_k)$.

\end{thm}

\begin{proof}

It is clear that $\alpha(d_1),\ldots,\alpha(d_k)$ are pairwise distinct, since their supports are even pairwise totally disjoint (cf. Lemma~\ref{lem:connectedcomponents1}). Therefore, it suffices to show the equality of sets
$$
\{\alpha\text{ a maximal root of }d\}=\{\alpha(d_1),\ldots,\alpha(d_k)\}\,.
$$
We first prove the inclusion \enquote{$\supseteq$}. Let $1\leq i\leq k$. Each greedy decomposition of $d_i$ has as unique first entry $\alpha(d_i)$. Consequently, Theorem~\ref{thm:additiongreedy} applied to $d_1,\ldots,d_k$ shows that there exists a greedy decomposition of $d$ which has as first entry $\alpha(d_i)$. (The assumption of Theorem~\ref{thm:additiongreedy} is satisfied in view of Lemma~\ref{lem:connectedcomponents1}.) This means that $\alpha(d_i)$ is a maximal root of $d$. Next, we prove the inclusion \enquote{$\subseteq$}. Let $\alpha$ be a maximal root of $d$. This means that there exists a greedy decomposition of $d$ which has as first entry $\alpha$. By Theorem~\ref{thm:additiongreedyconverse} applied to $d_1,\ldots,d_k$ there exists an $1\leq i\leq k$ and a greedy decomposition of $d_i$ such that the first entry of this greedy decomposition is $\alpha$. (Again, the assumption of Theorem~\ref{thm:additiongreedyconverse} is satisfied in view of Lemma~\ref{lem:connectedcomponents1}.) In other words, we have $\alpha=\alpha(d_i)$ for some $1\leq i\leq k$.
\end{proof}

\begin{cor}
\label{cor:maxroots}

Let $d$ be a degree. All maximal roots of $d$ occur in every greedy decomposition of $d$.

\end{cor}

\begin{lem}
\label{lem:maxroots}

Let $d$ be a degree. Let $d_1,\ldots,d_k$ be the connected components of $d$. Then there exist roots $\alpha_{k+1},\ldots,\alpha_r\in R^+\setminus R_P^+$ ($k\leq r$) such that $(\alpha(d_1),\ldots,\alpha(d_k),\alpha_{k+1},\ldots,\alpha_r)$ is a greedy decomposition of $d$.

\end{lem}

\begin{proof}[Proof of Corollary~\ref{cor:maxroots} and Lemma~\ref{lem:maxroots}]

In order to prove Corollary~\ref{cor:maxroots}, it clearly suffices to prove Lemma~\ref{lem:maxroots}. This is because of Theorem~\ref{thm:maxroots} and the uniqueness of the greedy decomposition up to reordering. But Lemma~\ref{lem:maxroots} follows directly from Theorem~\ref{thm:additiongreedy} applied to $d_1,\ldots,d_k$. The assumption of Theorem~\ref{thm:additiongreedy} is saitsfied in view of Lemma~\ref{lem:connectedcomponents1}.
\end{proof}

\begin{defn}[{\cite[Section~4.2]{curvenbhd2}}]

Let $d$ be a degree. Let $(\alpha_1,\ldots,\alpha_r)$ be a greedy decomposition $d$. Then we define
$$
\tilde{z}_d^P=s_{\alpha_1}\cdot\ldots\cdot s_{\alpha_r}\,.
$$
The element $\tilde{z}_d^P$ is well-defined (does only depend on $d$ and not on the choice of the greedy decomposition of $d$) for the same reason as $z_d^P$ is well-defined. This was proved in \cite[Section~4]{curvenbhd2}.

\end{defn}

\begin{rem}
\label{rem:tildezinvolution}

Let $d$ be a degree. Let $(\alpha_1,\ldots,\alpha_r)$ be a greedy decomposition of $d$. Let $e=\sum_{i=1}^k\alpha_i^\vee\in H_2(G/B)$\footnote{With the terminology of \cite[Subsection~4.1]{minimaldegrees}, the degree $e$ is the induction of $d$.}. Then we have $\tilde{z}_d^P=z_e^B$. In particular, it follows from \cite[Corollary~4.9]{curvenbhd2} that $\tilde{z}_d^P$ is an involution. It can also be seen more directly that $\tilde{z}_d^P$ is an involution by using \cite[Proposition~3.1(b)]{curvenbhd2} and \cite[Proposition~3.10(2)]{minimaldegrees}.

\end{rem}

\begin{thm}[Addition theorem for $\tilde{z}_d^P$]
\label{thm:additiontildez}

Let $d_1,\ldots,d_k$ be degrees such that $\widetilde{\Delta}(d_1),\ldots,\widetilde{\Delta}(d_k)$ are pairwise totally disjoint. Let $d=\sum_{i=1}^k d_i$ for short. Then we have
\begin{equation}
\label{eq:tildez}
\tilde{z}_d^P=\tilde{z}_{d_1}^P\cdot\ldots\cdot\tilde{z}_{d_k}^P=\tilde{z}_{d_1}^P\cdots\tilde{z}_{d_k}^P\,.
\end{equation}
This means in particular that the Hecke and the ordinary product in Equation~\eqref{eq:tildez} is independent of the ordering of the factors. 
%In other words, we have
%$$
%\tilde{z}_{d_i}^P\cdot\tilde{z}_{d_j}^P=\tilde{z}_{d_j}^P\cdot\tilde{z}_{d_i}^P=\tilde{z}_{d_i}^P\tilde{z}_{d_j}^P=\tilde{z}_{d_j}^P\tilde{z}_{d_i}^P\text{ for all }1\leq i,j\leq k\,.
%$$

\end{thm}

\begin{proof}

Let $(\alpha_1^i,\ldots,\alpha_{r_i}^i)$ be a greedy decomposition of $d_i$ for all $1\leq i\leq k$. By Theorem~\ref{thm:additiongreedy}, the sequence $(\alpha_1^1,\ldots,\alpha_{r_1}^1,\ldots,\alpha_1^k,\ldots,\alpha_{r_k}^k)$ is a greedy decomposition of $d$. Hence, the first equality in Equation~\eqref{eq:tildez} follows directly from the definition of $\tilde{z}_d^P$. We prove the second equality of Equation~\eqref{eq:tildez} by induction on $k$. The case $k=1$ is immediate. Assume that $k>1$ and that the second equality of Equation~\eqref{eq:tildez} is known for all values strictly smaller than $k$. By induction hypothesis applied to $d_2,\ldots,d_k$, it suffices to show that $\tilde{z}_{d_1}^P\cdot\tilde{z}_{d'}^P=\tilde{z}_{d_1}^P\tilde{z}_{d'}^P$ where $d'=\sum_{i=2}^k d_i$. By Lemma~\ref{lem:support}, it therefore suffices to show that $\Delta(\tilde{z}_{d_1}^P)\cap\Delta(z_{d'}^P)=\emptyset$. But by \cite[Proposition~3.17]{minimaldegrees} and Theorem~\ref{thm:additionextendedsupport}, we have $\Delta(\tilde{z}_{d_1}^P)=\widetilde{\Delta}(d_1)$ and $\Delta(\tilde{z}_{d'}^P)=\coprod_{i=2}^k\widetilde{\Delta}(d_i)$. Hence, the results follows by assumption.
\end{proof}

\begin{thm}[Addition theorem for minimal degrees]
\label{thm:additionminimaldegrees}

Let $d_1,\ldots,d_k$ be degrees such that $\widetilde{\Delta}(d_1),\ldots,\widetilde{\Delta}(d_k)$ are pairwise totally disjoint. Let $d=\sum_{i=1}^k d_i$ for short. Then we have
$$
d_1,\ldots,d_k\in\Pi_P\text{ if and only if }d\in\Pi_P\,.
$$

\end{thm}

\begin{proof}

The implication from right to left follows from Theorem~\ref{thm:additiongreedyconverse} and \cite[Proposition~4.4(9)]{minimaldegrees}. Assume that $d_1,\ldots,d_k\in\Pi_P$. A simple induction on $k$ shows that we may assume that $k=2$. Let $d'\in\Pi_P$ be a minimal degree such that $d'\leq d$ and such that $z_d^P=z_{d'}^P$. We have
$$
\Delta(d')\subseteq\Delta(d)=\Delta(d_1)\amalg\Delta(d_2)
$$
by the totally disjointness assumption. Therefore, we can write $d'=d_1'+d_2'$ where $d_1'$ and $d_2'$ are two degrees such that $d_1'\leq d_1$ and $d_2'\leq d_2$. By Theorem~\ref{thm:inclusionextendedsupport} we then have $\widetilde{\Delta}(d_1')\subseteq\widetilde{\Delta}(d_1)$ and $\widetilde{\Delta}(d_2')\subseteq\widetilde{\Delta}(d_2)$. The totally disjointness assumption implies that $\widetilde{\Delta}(d_1')$ and $\widetilde{\Delta}(d_2')$ are totally disjoint. If we now apply Theorem~\ref{thm:additiontildez} to the degrees $d_1,d_2$ and $d_1',d_2'$ we find that
\begin{equation}
\label{eq:tildez2}
\tilde{z}_d^P=\tilde{z}_{d_1}^P\tilde{z}_{d_2}^P\text{ and }\tilde{z}_{d'}^P=\tilde{z}_{d_1'}^P\tilde{z}_{d_2'}^P
\end{equation}
By definition, it is clear that $z_d^PW_P=\tilde{z}_d^PW_P$ and $z_{d'}^PW_P=\tilde{z}_{d'}^PW_P$. The fact that $z_d^P=z_{d'}^P$ and Equation~\eqref{eq:tildez2} therefore imply that 
$$
\tilde{z}_{d_1}^P\tilde{z}_{d_2}^PW_P=\tilde{z}_{d_1'}^P\tilde{z}_{d_2'}^PW_P\,.
$$
The last equation means that there exists a $w\in W_P$ such that
$$
\tilde{z}_{d_1}^P\tilde{z}_{d_2}^Pw=\tilde{z}_{d_1'}^P\tilde{z}_{d_2'}^P\,.
$$
In view of Remark~\ref{rem:tildezinvolution}, the last equation becomes after rearranging
$$
\tilde{z}_{d_2}^Pw\tilde{z}_{d_2'}^P=\tilde{z}_{d_1}^P\tilde{z}_{d_1'}^P\,.
$$
By changing the order of the product in Equation~\eqref{eq:tildez2} we also find the analogous equation
$$
\tilde{z}_{d_1}^Pw\tilde{z}_{d_1'}^P=\tilde{z}_{d_2}^P\tilde{z}_{d_2'}^P\,.
$$
Using the last two displayed equations and the computation rules for the support of a Weyl group element (cf. \cite[Proposition~3.17]{minimaldegrees}) and the Hecke product (cf. \cite[Proposition~3.1]{curvenbhd2}), we see that
\begin{align*}
\Delta(\tilde{z}_{d_1}^P\tilde{z}_{d_1'}^P)&\subseteq\widetilde{\Delta}(d_1)\cap(\widetilde{\Delta}(d_2)\cup\Delta_P)\subseteq\Delta_P\,,\\
\Delta(\tilde{z}_{d_2}^P\tilde{z}_{d_2'}^P)&\subseteq\widetilde{\Delta}(d_2)\cap(\widetilde{\Delta}(d_1)\cup\Delta_P)\subseteq\Delta_P\,.
\end{align*}
But this means that $\tilde{z}_{d_1}^P\tilde{z}_{d_1'}^P,\tilde{z}_{d_2}^P\tilde{z}_{d_2'}^P\in W_P$. Remark~\ref{rem:tildezinvolution} shows again that we can rewrite this containment as
$\tilde{z}_{d_1}^PW_P=\tilde{z}_{d_1'}^PW_P$ and $\tilde{z}_{d_2}^PW_P=\tilde{z}_{d_2'}^PW_P$. As $d_1,d_2\in\Pi_P$, $d_1'\leq d_1$ and $d_2'\leq d_2$, these two equations and the very definition of minimal degrees imply that $d_1=d_1'$ and $d_2=d_2'$. Eventually, this means $d=d'\in\Pi_P$ -- as claimed.
\end{proof}

\begin{cor}
\label{cor:additionminimaldegrees}

Let $d$ be a degree. Let $d_1,\ldots,d_k$ be the connected components of $d$. Then we have
$$
d_1,\ldots,d_k\in\Pi_P\text{ if and only if }d\in\Pi_P\,.
$$

\end{cor}

\begin{proof}

This follows directly from Lemma~\ref{lem:connectedcomponents1} and Theorem~\ref{thm:additionminimaldegrees}.
\end{proof}

\subsection{On the number of greedy decompositions of a degree}

This subsection is not logically needed elsewhere in the text. The impatient reader can skip it. As an immediate consequence of Theorem~\ref{thm:additiongreedy}, \ref{thm:additiongreedyconverse}, we develop a formula for the number of greedy decompositions of a degree (cf. Theorem~\ref{thm:nd}). 
%Later on in Appendix~\ref{appendix:graph}, we want to elaborate on Equation~\eqref{eq:nd} by visualizing its combinatorial meaning with the help of a graph associated to a degree. Here, we only deduce it as a direct outcome of the work already done before. For a sample computation, which makes clearer how to actually use Theorem~\ref{thm:nd}, we also refer to Appendix~\ref{appendix:graph} (cf. Example~\ref{ex:computationndtyped}).
For the reader familiar with Kostant's cascade of orthogonal roots \cite[Section~1]{kostant}, we give a sample computation of the number of greedy decompositions of $d_{G/B}$ (cf. Example~\ref{ex:np}) which makes clear how to use Theorem~\ref{thm:nd}.

\begin{notation}

Let $d$ be a degree. We denote by $r(d)$ the length of a greedy decomposition. By the uniqueness of the greedy decomposition up to reordering, $r(d)$ is clearly well-defined -- does not depend on the choice of the greedy decomposition but only on $d$. Furthermore, we denote by $N_d$ the number of greedy decompositions of $d$.
%, i.e. the cardinality of the set
%$$
%\{(\alpha_1,\ldots,\alpha_r)\mid(\alpha_1,\ldots,\alpha_r)\text{ is a greedy decomposition of }d\}
%$$

\end{notation}

\begin{thm}
\label{thm:nd}

Let $d$ be a degree. Let 
$$
d_1,\ldots,d_k\text{ be the connected components of}
\begin{cases}
d & \text{if $d$ is a disconnected degree,}\\
d-d(\alpha(d)) & \text{if $d$ is a connected degree.}
\end{cases}
$$
Then we have the following recursive formula for the number of greedy decompositions of $d$ which expresses $N_d$ in terms of $N_{d'}$ with $d'<d$:
\begin{equation}
\label{eq:nd}
N_d=\frac{(r(d_1)+\cdots+r(d_k))!}{r(d_1)!\cdots r(d_k)!}\cdot N_{d_1}\cdots N_{d_k}\,.
\end{equation}

\end{thm}

\todo[inline,color=green]{After the definition of connected / disconnected degree, make a remark / convention that $d=0$, i.e. $\widetilde{\Delta}(d)=\emptyset$ is, as the definition says, actually connected. But anyway, we will often tacitly assume that a connected degree $d$ is automatically $\neq 0$, so that $\alpha(d)$ is well-defined. That matters actually whenever we use $\alpha(d)$ since I never worry about this case. The case $d=0$ can always be treated in a trivial way (in all proofs we do). Mention it once that $d\neq 0$ in the definition of $\alpha(d)$, and then the convention follows\ldots}

\begin{rem}
\label{rem:nd}

Equally well, one can express the meaning of the recursive formula in Theorem~\ref{thm:nd} by saying that $N_d=N_{d-d(\alpha(d))}$ for a connected degree $d$ and that $N_d$ can be written in terms of $N_{d_1},\ldots,N_{d_k}$ as in Equation~\eqref{eq:nd} for the connected components $d_1,\ldots,d_k$ of $d$.

\end{rem}

\begin{proof}[Proof of Theorem~\ref{thm:nd}]

We prove the equivalent version of Equation~\eqref{eq:nd} as described in Remark~\ref{rem:nd}. Assume first that $d$ is a connected degree. We have an obvious bijection
$$
\{\text{greedy decompositions of }d\}\cong\{\text{greedy decompositions of }d-d(\alpha(d))\}
$$
where a greedy decomposition $(\alpha_1,\ldots,\alpha_r)$ of $d$ is sent to $(\alpha_2,\ldots,\alpha_r)$ and, the other way round, a greedy decomposition $(\alpha_2,\ldots,\alpha_r)$ of $d-d(\alpha(d))$ is sent to $(\alpha(d),\alpha_2,\ldots,\alpha_r)$. This assignment is clearly well-defined in view of \cite[Proposition~3.10(7)]{minimaldegrees} and leads to the equality $N_d=N_{d-d(\alpha(d))}$ by taking cardinalities. 

Next, let $d_1,\ldots,d_k$ be the connected components of $d$. Let $r=r(d)$ for short. By Theorem~\ref{thm:additiongreedy}, \ref{thm:additiongreedyconverse} we have the following partition of the set of greedy decompositions of $d$ into disjoint sets:
\begin{multline*}
\{\text{greedy decompositions of }d\}=\\
\coprod_{\substack{(\alpha_1^i,\ldots,\alpha_{r_i}^i)\text{ greedy of }d_i\\\forall 1\leq i\leq k}}\{(\alpha_1,\ldots,\alpha_r)\mid(\alpha_1^i,\ldots,\alpha_{r_i}^i)\text{ is a subsequence of }(\alpha_1,\ldots,\alpha_r)\text{ for all }i\}\,.
\end{multline*}
The disjoint union in this formula runs over $N_{d_1}\cdots N_{d_k}$ sets of equal cardinality and each of these sets has cardinality equal to the number of ways how to combine $k$ disjoint sequences of length $r_1,\ldots, r_k$ respectively to a unique sequence of length $r=\sum_{i=1}^k r_i$. But this number is precisely 
$
\frac{r!}{r_1!\cdots r_k!}
$
-- a so-called multinomial coefficient. This directly leads to Equation~\eqref{eq:nd} by taking cardinalities in the previous displayed equation.
\end{proof}

\begin{ex}[For the reader familiar with Kostant's cascade of orthogonal roots {\cite[Section~1]{kostant}}]
\label{ex:np}

Let $R$ be of type $\mathsf{D}_p$ where $p\geq 3$. Let $r_p=r(d_{G/B})$ and $N_p=N_{d_{G/B}}$. In this example, we compute $r_p$ and $N_p$ for all $p\geq 3$. By \cite[Corollary~7.12, Footnote~8]{minimaldegrees}, we know that any greedy decomposition of $d_{G/B}$ is an ordering of Kostant's cascade of orthogonal roots (cf. Remark~\ref{rem:general}), in particular the length $r_p$ of a greedy decomposition of $d_{G/B}$ equals the cardinality
%length (i.e. the cardinality) 
of Kostant's cascade of orthogonal roots. Therefore, a simple computation shows that we have
$$
r_p=\begin{cases}
p-1 &\text{if $p$ is odd}\\
p &\text{if $p$ is even.}
\end{cases}
$$
Using the previous formula for $r_p$ and Theorem~\ref{thm:nd}, a simple induction for the odd and the even case shows
$$
N_p=\begin{cases}
(p-2)!! &\text{if $p$ is odd}\\
2(p-1)!! &\text{if $p$ is even}
\end{cases}
$$
Indeed, we have the initial values $N_3=1$, $N_4=3!=6$ and the recursive relation
$$
N_p=(r_{p-2}+1)N_{p-2}\text{ for }p\geq 5
$$
which results from Theorem~\ref{thm:nd}.

\end{ex}

\section{Generalized cascades of orthogonal roots}
\label{sec:gencascade}

\todo[inline,color=green]{Key feature of the argumentation are orthogonality relations. We encourage the reader to read Section~8 first. These generalizations are fundamental for the arguments leading eventually to quasi-homogeneity -- and the main results of this paper.

Orthogonality in generalized chain cascades is crucial for the construction of the diagonal curve.}

\todo[inline,color=green]{For the line of arguments leading to quasi-homogeneity for generalized complete flag varieties only one length equality is actually needed. Mention in the introduction that this corollary is the most important thing for the partial result in $G/B$.

I mention enough in the organization section, even if I don't mention Corollary~\ref{cor:gencascade:main2} explicitly.}

\todo[inline,color=green]{Concerning the last sentence of the next paragraph: It will be clear from the introduction that the diagonal curve is the candidate for a curve which has a dense open orbit. I define the term diagonal curve in Construction~\ref{construction:dia}.}

In this section, we develop the theory of generalized cascades of orthogonal roots. This theory is a direct generalization of Kostant's cascade of orthogonal roots \cite[Section~1]{kostant} in the sense that Kostant's cascade is associated to the specific minimal degree $d_{G/B}\in\Pi_B$ while the general construction works for an arbitrary minimal degree $e\in\Pi_B$ (for more details see Remark~\ref{rem:general}). Basically all properties of Kostant's cascade as they were investigated in \cite[loc. cit.]{kostant} carry over in a reasonable way to generalized cascades as the theorems in this section show. But the proofs are less evident. As a key feature of minimal degrees, we need the so-called orthogonality relations in greedy decompositions of minimal degrees which were first proved in \cite[Theorem~8.1]{minimaldegrees}. With the help of this property of minimal degrees, it is easy to see that two distinct elements of a generalized cascade of orthogonal roots are strongly orthogonal (Theorem~\ref{thm:gencascade}\eqref{item:stronglyorthogonal}). This kind of orthogonality is essential to make the diagonal curve (cf. Definition~\ref{def:dia}) well-defined and eventually to prove our main result on quasi-homogeneity. 

\begin{defn}
\label{def:gencascade}

Let $e\in\Pi_B$. Let $\varphi\in R^+$. Then we define 
$$
\mathcal{B}_{R,e}=\{\alpha\in R^+\mid\alpha\text{ occurs in a greedy decomposition of }e\}\,.
$$
We call the set $\mathcal{B}_{R,e}$ of positive roots a generalized cascade of orthogonal roots. Furthermore, we define
$$
C_{R,e}(\varphi)=\{\alpha\in\mathcal{B}_{R,e}\mid\alpha\geq\varphi\}\,.
$$
We call the subset $C_{R,e}(\varphi)$ of $\mathcal{B}_{R,e}$ a generalized chain cascade. We define the length of a generalized chain cascade to be the cardinality of a generalized chain cascade.

\end{defn}

\begin{rem}
\label{rem:general}

We will see in Theorem~\ref{thm:gencascade}\eqref{item:stronglyorthogonal} that the name \enquote{generalized cascade of orthogonal roots} is justified by showing that two distinct elements of any generalized cascade of orthogonal roots are indeed orthogonal (even strongly orthogonal). For now, we only want to mention that Definition~\ref{def:gencascade} is a direct generalization of Kostant's cascade of orthogonal roots and chain cascades \cite[Section~1]{kostant} from one specific minimal degree in $\Pi_B$, namely $d_{G/B}$, to arbitrary minimal degrees in $\Pi_B$. Indeed, by \cite[Corollary~7.12, Footnote~8]{minimaldegrees}, we have
\begin{align*}
\mathcal{B}_{R,d_{G/B}}&=\text{ordinary cascade of orthogonal roots}\,,\\
C_{R,d_{G/B}}(\varphi)&=\text{ordinary chain cascades where }\varphi\in R^+.
\end{align*}
While this relation to the original concepts is important to understand their generalizations intuitively, it is not formally needed elsewhere in the text.

\end{rem}

\begin{fact}
\label{fact:altdef}

Let $e\in\Pi_B$. Let $e_1,\ldots,e_k$ be the connected components of $e$. Then we have
$$
\mathcal{B}_{R,e}=\coprod_{i=1}^k\mathcal{B}_{R,e_i}\,.
$$
Moreover, if $e$ is a connected degree, then we have
$$
\mathcal{B}_{R,e}=\{\alpha(e)\}\amalg\mathcal{B}_{R,e-\alpha(e)^\vee}\,.
$$
The two previous formulas uniquely determine $\mathcal{B}_{R,e}$ and can serve as an inductive definition of generalized cascades of orthogonal roots alternative to Definition~\ref{def:gencascade}.

\end{fact}

\begin{rem}

The reader may wish to compare the two formulas in Fact~\ref{fact:altdef} which can serve as an alternative definition of generalized cascades of orthogonal roots with the definition of ordinary cascades of orthogonal roots which we chose in \cite[Definition~7.6]{minimaldegrees} (see also \cite[Remark~7.7]{minimaldegrees}). The two definitions are completely analogous. One may also expect analogous properties of ordinary and generalized cascades of orthogonal roots to hold. Indeed, we will see throughout this section that this is true (cf. Theorem~\ref{thm:gencascade}, \ref{thm:gencascade:main}, \ref{thm:gencascade:main2}).

\end{rem}

\begin{comment}

\begin{proof}[Proof of Fact~\ref{fact:altdef}]

The first formula in Fact~\ref{fact:altdef} follows directly from Lemma~\ref{lem:connectedcomponents1} and \cite[Proposition~3.10(7), 4.4(9)]{minimaldegrees}. Suppose that $e$ is a connected degree. Let $\alpha=\alpha(e)$ for short. The equation
$$
\mathcal{B}_{R,e}=\{\alpha\}\cup\mathcal{B}_{R,e-\alpha^\vee}
$$
follows similar as the first formula of Fact~\ref{fact:altdef} by \cite[loc. cit.]{minimaldegrees}. The disjointness of the previous union follows by \cite[Remark~8.4]{minimaldegrees}.
\end{proof}

\end{comment}

\begin{proof}[Proof of Fact~\ref{fact:altdef}]

By Corollary~\ref{cor:additionminimaldegrees}, we have $e_1,\ldots,e_k\in\Pi_B$. Hence, the sets $\mathcal{B}_{R,e_i}$ are well-defined for all $1\leq i\leq k$. It follows immediately from Lemma~\ref{lem:connectedcomponents1} that the union $\bigcup_{i=1}^k\mathcal{B}_{R,e_i}$ is disjoint. The equality $\mathcal{B}_{R,e}=\coprod_{i=1}^k\mathcal{B}_{R,e_i}$ is a trivial consequence of Theorem~\ref{thm:additiongreedy} and Theorem~\ref{thm:additiongreedyconverse}.
Suppose that $e$ is a connected degree. Let $\alpha=\alpha(e)$ for short. The equation
$$
\mathcal{B}_{R,e}=\{\alpha\}\cup\mathcal{B}_{R,e-\alpha^\vee}
$$
follows similar as in the proof of Theorem~\ref{thm:nd} by \cite[Proposition~3.10(7), 4.4(9)]{minimaldegrees}. The disjointness of the previous union follows by \cite[Remark~8.4]{minimaldegrees}.
\end{proof}

\todo[inline,color=green]{I want to add a fact here which compares the definition with the original one: decomposition into connected components, inductive definition.}

\begin{thm}
\label{thm:gencascade}

Let $e\in\Pi_B$. Let $\varphi\in R^+$.

\begin{enumerate}

\item 
\label{item:totallyordered}

The generalized chain cascade $C_{R,e}(\varphi)$ is totally ordered.

%Any generalized chain cascade $C_{R,e}(\varphi)$ is totally ordered.

\item
\label{item:gencascadebcosmall}

All elements of the generalized cascade of orthogonal roots $\mathcal{B}_{R,e}$ are $B$-cosmall.

\item
\label{item:stronglyorthogonal}

Two distinct elements of the generalized cascade of orthogonal roots $\mathcal{B}_{R,e}$ are strongly orthogonal.

\item
\label{item:totallydisjoint1}

Let $\alpha,\alpha'\in\mathcal{B}_{R,e}$ such that $C_{R,e}(\alpha)\cap C_{R,e}(\alpha')=\emptyset$. Then the sets $R(\alpha)$ and $R(\alpha')$ are totally disjoint.

\item
\label{item:totallydisjoint2}

Let $\alpha,\alpha'\in\mathcal{B}_{R,e}$. Assume that $\alpha$ and $\alpha'$ do not belong to a common generalized chain cascade $C_{R,e}(\varphi)$, i.e. assume that there exists no $\varphi\in R^+$ such that $\alpha,\alpha'\in C_{R,e}(\varphi)$. Then the sets $R(\alpha)$ and $R(\alpha')$ are totally disjoint.

\end{enumerate}

\end{thm}

\begin{rem}
\label{rem:general2}

The properties of generalized cascades of orthogonal roots which are subject to Item~\eqref{item:totallyordered}, \eqref{item:stronglyorthogonal}, \eqref{item:totallydisjoint2} are direct generalizations of properties of the ordinary cascade of orthogonal roots. Indeed, one respectively recovers the statements \cite[Remark~1.3, Lemma~1.6, Proposition~1.7]{kostant} (see also \cite[Proposition~7.8]{minimaldegrees} for a summary) by setting $e=d_{G/B}$ (cf. Remark~\ref{rem:general}). 
%The strength of the results lies in the fact that they not only hold for one specific minimal degree in $\Pi_B$, namely $d_{G/B}$, but for all. Consequently, the generalizations are worth a formal proof, while the statements of Kostant are very easy to see. 
In view of Fact~\ref{fact:lochighBcosmall}, Item~\eqref{item:gencascadebcosmall} can be understood as a weaker version of \cite[Proposition~1.4]{kostant} which holds for all $e\in\Pi_B$. Note that there is no hope to expect locally high roots in $\mathcal{B}_{R,e}$ for all $e\neq d_{G/B}$, since this already fails in the simplest examples (e.g. if $R$ is simply laced and not of type $\mathsf{A}$).

%One may consider $\alpha\in R^+$ minimal such that $\Delta(\alpha)=\Delta(\theta_1)$ and $e=\alpha^\vee$. Since $\alpha$ is $B$-cosmall ($R$ simply laced), we have $\mathcal{B}_{R,e}=\{\alpha\}$, but $\alpha$ is not locally high.

%cf. \cite[Theorem~4.15]{minimaldegrees}). 
%Item~\eqref{item:totallydisjoint1} is only an auxiliary statement to prove Item~\eqref{item:totallydisjoint2}.

\todo[inline,color=green]{{\color{black} Here I want to refer to the previous paper and Kostant for the concordance of the result (also to the previous remark after Definition~\ref{def:gencascade}). Also, I want to mention that there is no chance to get locally high roots in generalized cascades, but only $B$-cosmall roots. The first statement now requires a proof\ldots Moreover, (4) is just an auxiliary statement to prove (5).

Even the statement $z_d^B=\prod_{\alpha\in\mathcal{B}_{R,e}}s_\alpha$ can be seen as a generalization of a result of Kostant. We will have to remark this\ldots As in Remark~\ref{rem:general2}, we see that\ldots}}

\end{rem}

\begin{proof}[Proof of Item~\eqref{item:totallyordered}]

We prove this statement by induction on the length of generalized chain cascades. For any generalized chain cascade of length one, the statement is trivially satisfied. Let $n>1$ be an integer and suppose that any generalized chain cascade of length less than $n$ is totally ordered. Let $C_{R,e}(\varphi)$ be an arbitrary generalized chain cascade of length $n$ for some $e\in\Pi_B$ and some $\varphi\in R^+$. Let
$$
\hat{e}=\sum_{\alpha\in C_{R,e}(\varphi)}\alpha^\vee\,.
$$
By \cite[Proposition~4.4(9)]{minimaldegrees}, we know that $\hat{e}\in\Pi_B$. Moreover, by \cite[Proposition~3.10(7)]{minimaldegrees}, we know that 
\begin{equation}
\label{eq:bequalsc}
\mathcal{B}_{R,\hat{e}}=C_{R,e}(\varphi)\,.
\end{equation}

\begin{proof}[Claim: The degree $\hat{e}$ is connected]\renewcommand{\qedsymbol}{$\triangle$}

Indeed, let $\hat{e}_1,\ldots,\hat{e}_k$ be the connected components of $\hat{e}$. By Equation~\eqref{eq:bequalsc}, we have $\varphi\leq\alpha$ for all $\alpha\in\mathcal{B}_{R,\hat{e}}$ and thus $\Delta(\varphi)\subseteq\Delta(\alpha)$ for all $\alpha\in\mathcal{B}_{R,\hat{e}}$. By Lemma~\ref{lem:connectedcomponents1}, 
%and the definition of connected components of a degree
we have
%Since
$$
\hat{e}_i=\sum_{\alpha\in\mathcal{B}_{R,\hat{e}}\colon\Delta(\alpha)\subseteq\Delta(\hat{e}_i)}\alpha^\vee\,.
$$
It follows that $\Delta(\varphi)\subseteq\Delta(\hat{e}_i)$ for all $1\leq i\leq k$. Again, by Lemma~\ref{lem:connectedcomponents1}, we must have $k=1$ and that $\hat{e}$ is connected.
\end{proof}

%\noindent
Since $\hat{e}$ is connected, we can define $\alpha_1=\alpha(\hat{e})$ to be the unique first entry of a greedy decomposition of $\hat{e}$. 
\todo[disable]{Needs a reference to the introduction, where I want to recall basic facts on connected degrees. It's a definition, therefore I ignore this reference.}
By \cite[Proposition~3.16]{minimaldegrees}, we have $\alpha_1\geq\alpha$ for all $\alpha\in\mathcal{B}_{R,\hat{e}}$. By Equation~\eqref{eq:bequalsc}, this means that $\alpha_1$ is the unique maximal element of $C_{R,e}(\varphi)$

\begin{proof}[Claim: We have $C_{R,\hat{e}-\alpha_1^\vee}(\varphi)=C_{R,e}(\varphi)\setminus\{\alpha_1\}$]\renewcommand{\qedsymbol}{$\triangle$}

First note that $\hat{e}-\alpha_1^\vee\in\Pi_B$ by \cite[Proposition~4.4(9)]{minimaldegrees}. Hence, it makes sense to speak about $C_{R,\hat{e}-\alpha_1^\vee}(\varphi)$ and $\mathcal{B}_{R,\hat{e}-\alpha_1^\vee}$.
%
%Indeed, 
By Equation~\eqref{eq:bequalsc}, it suffices to prove that $C_{R,\hat{e}-\alpha_1^\vee}(\varphi)=\mathcal{B}_{R,\hat{e}}\setminus\{\alpha_1\}$. By \cite[Proposition~3.10(7), Remark~8.4]{minimaldegrees}, it is clear that 
\begin{equation}
\label{eq:bequalsb}
\mathcal{B}_{R,\hat{e}}\setminus\{\alpha_1\}=\mathcal{B}_{R,\hat{e}-\alpha_1^\vee}\,.
\end{equation}
Thus, it suffices to show that $\alpha\geq\varphi$ for all $\alpha\in\mathcal{B}_{R,\hat{e}-\alpha_1^\vee}$. But this later statements is clear in view of Equation~\eqref{eq:bequalsc}, \eqref{eq:bequalsb}, since we have an inclusion $\mathcal{B}_{R,\hat{e}-\alpha_1^\vee}\subseteq C_{R,e}(\varphi)$.
\end{proof}

The previous claim shows that $C_{R,\hat{e}-\alpha_1^\vee}(\varphi)$ is a generalized chain cascade of length $n-1$. Thus the induction hypothesis implies that there exists a total ordering
$$
C_{R,\hat{e}-\alpha_1^\vee}(\varphi)=\{\alpha_2\geq\cdots\geq\alpha_n\}\,.
$$
Again the previous claim and the fact that $\alpha_1$ is the unique maximal element of $C_{R,e}(\varphi)$ show that we get a total ordering
\begin{equation*}
C_{R,e}(\varphi)=\{\alpha_1\geq\alpha_2\geq\cdots\geq\alpha_n\}\,.\qedhere
\end{equation*}
\end{proof}

\begin{proof}[Proof of Item~\eqref{item:gencascadebcosmall}]

Let $\alpha\in\mathcal{B}_{R,e}$. By definition, $\alpha$ occurs in a greedy decomposition of $e$. This means, there exists a degree $e'\in H_2(G/B)$ such that $\alpha^\vee\leq e'\leq e$ and such that $\alpha$ is a maximal root of $e'$. This implies in particular that $\alpha$ is $B$-cosmall.
\end{proof}

\begin{proof}[Proof of Item~\eqref{item:stronglyorthogonal}]

The statement is a consequence of the assumption $e\in\Pi_B$ which leads to a phenomenon which we call \enquote{orthogonality relations in greedy decompositions} (cf. \cite[Section~8]{minimaldegrees}. By \cite[Corollary~8.3]{minimaldegrees}, we know that two different entries of a greedy decomposition of $e\in\Pi_B$ are strongly orthogonal, in particular we know that two distinct elements of $\mathcal{B}_{R,e}$ are strongly orthogonal.
\end{proof}

\begin{proof}[Proof of Item~\eqref{item:totallydisjoint1}]

Let $\alpha,\alpha'\in\mathcal{B}_{R,e}$ such that $C_{R,e}(\alpha)\cap C_{R,e}(\alpha')=\emptyset$. If $e$ is a connected degree, then we have $\alpha,\alpha'\leq\alpha(e)$ (\cite[Proposition~3.16]{minimaldegrees}) and thus $\alpha(e)\in C_{R,e}(\alpha)\cap C_{R,e}(\alpha')$ -- contrary to our assumption. Therefore, we conclude that $e$ is a disconnected degree. Let $e_1,\ldots,e_k$ be the connected components of $e$. By Lemma~\ref{lem:connectedcomponents1}, we necessarily have $k>1$. By definition, we have $\Delta(\alpha),\Delta(\alpha')\subseteq\Delta(e)$. By Lemma~\ref{lem:connectedcomponents1}, we conclude that $\Delta(\alpha)\subseteq\Delta(e_i)$ and $\Delta(\alpha')\subseteq\Delta(e_j)$ for some $1\leq i,j\leq k$. By Theorem~\ref{thm:additiongreedyconverse}, we know that $\alpha$ occurs in a greedy decomposition of $e_i$ and that $\alpha'$ occurs in a greedy decomposition of $e_j$. Hence, \cite[Proposition~3.16]{minimaldegrees} says that $\alpha\leq\alpha(e_i)$ and that $\alpha'\leq\alpha(e_j)$. Moreover, it follows from Lemma~\ref{lem:maxroots} that $\alpha(e_i),\alpha(e_j)\in\mathcal{B}_{R,e}$. If $i=j$, then we must have $\alpha(e_i)=\alpha(e_j)\in C_{R,e}(\alpha)\cap C_{R,e}(\alpha')$ -- contrary to the initial assumption. Therefore, we conclude that $i\neq j$.

%By \cite[Proposition~3.10(7)]{minimaldegrees}, we know that there exists a greedy decomposition of $e_i$ and $e_j$ which is a sub-sequence of a greedy decomposition of $e$. Consequently, we have $\alpha(e_i),\alpha(e_j)\in\mathcal{B}_{R,e}$. 
%(cf. Lemma~\ref{lem:connectedcomponents1}). 
%If $i=j$, then $\alpha,\alpha'\leq\alpha(e_i)=\alpha(e_j)$ by \cite[Proposition~3.16]{minimaldegrees} and thus $\alpha(e_i)=\alpha(e_j)\in C_{R,e}(\alpha)\cap C_{R,e}(\alpha')$ -- contrary to the initial assumption. Therefore, we conclude that $i\neq j$. 

Let $\varphi_i,\varphi_j\in R^+$ be locally high roots such that $\Delta(e_i)=\Delta(\varphi_i)$ and $\Delta(e_j)=\Delta(\varphi_j)$. By Fact~\ref{fact:stronglyorthogonal} (applied to $S=\Delta(e)$), we conclude that $R(\varphi_i)$ and $R(\varphi_j)$ are totally disjoint. Since we clearly have $R(\alpha)\subseteq R(\varphi_i)$ and $R(\alpha')\subseteq R(\varphi_j)$, it follows that $R(\alpha)$ and $R(\alpha')$ are also totally disjoint -- as claimed.
\end{proof}

\begin{proof}[Proof of Item~\eqref{item:totallydisjoint2}]

Let $\alpha,\alpha'\in\mathcal{B}_{R,e}$. Assume that there exists no $\varphi\in R^+$ such that $\alpha,\alpha'\in C_{R,e}(\varphi)$. If $C_{R,e}(\alpha)\cap C_{R,e}(\alpha')=\emptyset$, then, by Item~\eqref{item:totallydisjoint1}, the set $R(\alpha)$ and $R(\alpha')$ are totally disjoint. Therefore, we may assume that $C_{R,e}(\alpha)\cap C_{R,e}(\alpha')\neq\emptyset$. Let $C=C_{R,e}(\alpha)\cap C_{R,e}(\alpha')$ for short. By Item~\eqref{item:totallyordered}, there exists a unique minimal element in $C$, which we denote by $\varphi$. It is easy to see that we have $C=C_{R,e}(\varphi)$.

\begin{proof}[Claim: We have $\alpha<\varphi$ and $\alpha'<\varphi$]\renewcommand{\qedsymbol}{$\triangle$}

Indeed, since $\varphi\in C_{R,e}(\alpha)$, we have $\alpha\leq\varphi$. On the other hand, if $\alpha=\varphi$, then $\alpha\in C\subseteq C_{R,e}(\alpha')$ -- contrary to our initial assumption. Therefore, we must have $\alpha\neq\varphi$ and thus $\alpha<\varphi$. Analogously, we see that $\alpha'<\varphi$.
\end{proof}

Let
$$
\hat{e}=\sum_{\mu\in\mathcal{B}_{R,e}\colon\mu<\varphi}\mu^\vee\,.
$$
By \cite[Proposition~4.4(9)]{minimaldegrees}, we know that $\hat{e}\in\Pi_B$. Moreover, by \cite[Proposition~3.10(7)]{minimaldegrees}, we know that
\begin{equation}
\label{eq:descriptionofb}
\mathcal{B}_{R,\hat{e}}=\{\mu\in\mathcal{B}_{R,e}\mid\mu<\varphi\}\,.
\end{equation}
The previous claim shows in particular that $\alpha,\alpha'\in\mathcal{B}_{R,\hat{e}}$.

\begin{proof}[Claim: We have $C_{R,\hat{e}}(\alpha)\cap C_{R,\hat{e}}(\alpha')=\emptyset$]\renewcommand{\qedsymbol}{$\triangle$}

Suppose for a contradiction that 
$$
C_{R,\hat{e}}(\alpha)\cap C_{R,\hat{e}}(\alpha')\neq\emptyset\,.
$$
By Equation~\eqref{eq:descriptionofb}, then there exists a $\varphi'\in\mathcal{B}_{R,e}$ such that $\varphi'<\varphi$ and such that $\varphi'\geq\alpha$ and $\varphi'\geq\alpha'$. This means that $\varphi'$ is an element of $C$ which is strictly smaller than $\varphi$ -- contrary to the choice of $\varphi$ as the unique minimal element of $C$.
\end{proof}

By the previous claim, we can apply Item~\eqref{item:totallydisjoint1} to $\alpha,\alpha'\in\mathcal{B}_{R,\hat{e}}$ and get that $R(\alpha)$ and $R(\alpha')$ are totally disjoint -- as desired.
\end{proof}

\begin{thm}
\label{thm:gencascade:main}

Let $e\in\Pi_B$. Then we have the following formula:
$$
z_e^B=\prod_{\alpha\in\mathcal{B}_{R,e}}s_\alpha\,.
$$

\end{thm}

\begin{rem}

Note that the product Theorem~\ref{thm:gencascade:main} speaks about is well-defined in view of Theorem~\ref{thm:gencascade}\eqref{item:stronglyorthogonal}. For the very same reason, all other products we will refer to are also well-defined.

\end{rem}

\begin{rem}
\label{rem:general3}

As in Remark~\ref{rem:general2}, we see that Theorem~\ref{thm:gencascade:main} is an important generalization of \cite[Proposition~1.10]{kostant}. Indeed, we recover the formula
$$
w_o=\prod_{\alpha\in\mathcal{B}_{R,d_{G/B}}}s_\alpha
$$
where $\mathcal{B}_{R,d_{G/B}}$ is the ordinary cascade of orthogonal roots by setting $e=d_{G/B}$ (cf. Remark~\ref{rem:general}). %Theorem~\ref{thm:gencascade:main} and its corollaries have many application throughout the text and play a key role in our argumentation.

\end{rem}

\begin{proof}[Proof of Theorem~\ref{thm:gencascade:main}]

Let $(\alpha_1,\ldots,\alpha_r)$ be a greedy decomposition of $e$. For brevity, let $\alpha=\alpha_1$. By \cite[Proposition~4.4(9)]{minimaldegrees}, we know that $e-\alpha^\vee\in\Pi_B$. By \cite[Proposition~3.10(7)]{minimaldegrees}, 
%and the uniqueness of the greedy decomposition up to reordering, 
we know that $\mathcal{B}_{R,e-\alpha^\vee}=\{\alpha_2,\ldots,\alpha_r\}$. Moreover, since $e\in\Pi_B$, we know that there are no repeated entries in a greedy decomposition of $e$ (\cite[Remark~8.4]{minimaldegrees}). Hence, it follows that
\begin{equation}
\label{eq:bequalsb2}
\mathcal{B}_{R,e-\alpha^\vee}=\mathcal{B}_{R,e}\setminus\{\alpha\}\,.
\end{equation}

We now perform an induction on the length of the greedy decomposition of $e$. The statement of Theorem~\ref{thm:gencascade:main} is obvious whenever the greedy decomposition of $e$ has length zero, i.e. if $e=0$ and $\mathcal{B}_{R,e}=\emptyset$. Assume that $r>0$ and that the statement is known for all minimal degrees in $\Pi_B$ whose greedy decomposition has length less than $r$. By \cite[Proposition~3.10(7)]{minimaldegrees}, a greedy decomposition of $e-\alpha^\vee$ (e.g. $(\alpha_2,\ldots,\alpha_r)$) has length $r-1$. Hence, the induction hypothesis applies to $e-\alpha^\vee$. In view of Equation~\eqref{eq:bequalsb2}, we find that
\begin{equation}
\label{eq:indhyp}
z_{e-\alpha^\vee}^B=\prod_{\mu\in\mathcal{B}_{R,e}\setminus\{\alpha\}}s_\mu\,.
\end{equation}
To prove Theorem~\ref{thm:gencascade:main}, it therefore suffices to show that
\begin{equation}
\label{eq:zeequalsze}
z_e^B=s_\alpha z_{e-\alpha^\vee}^B\,.
\end{equation}
By \cite[Proposition~3.2: $(c)\Leftrightarrow(e)$]{curvenbhd2} and \cite[Corollary~4.9]{curvenbhd2} (see also \cite[Proposition~3.10(6)]{minimaldegrees} for a more refined statement), Equation~\eqref{eq:zeequalsze} is equivalent to the statement
\begin{equation}
\label{eq:inversioncap}
I(s_\alpha)\cap I(z_{e-\alpha^\vee}^B)=\emptyset\,.
\end{equation}

Let $Q$ be the parabolic subgroup of $G$ such that $\Delta_Q=\Delta_\alpha^\circ$.

\begin{proof}[Claim: We have $I(z_{e-\alpha^\vee}^B)\subseteq R_Q^+$]\renewcommand{\qedsymbol}{$\triangle$}

Indeed, to prove the claim, it suffices to show that $z_{e-\alpha^\vee}^B\in W_Q$ (\cite[5.5, Theorem~(b)]{humphreys3}). By Equation~\eqref{eq:bequalsb2}, \eqref{eq:indhyp}, it suffices to show that 
$$
s_\mu\in W_Q\Leftrightarrow\mu\in R_Q^+\Leftrightarrow\Delta(\mu)\subseteq\Delta_Q
$$ 
for all $\mu\in\mathcal{B}_{R,e-\alpha^\vee}$. But by \cite[Theorem~8.1]{minimaldegrees}, we have 
$$
\Delta(\mu)\subseteq\Delta(e-\alpha^\vee)\subseteq\Delta_\alpha^\circ=\Delta_Q
$$
for all $\mu\in\mathcal{B}_{R,e-\alpha^\vee}$. This completes the proof of the claim.
\end{proof}

We are now able to prove Equation~\eqref{eq:inversioncap} which will complete the proof of Theorem~\ref{thm:gencascade:main}. Let $\mu\in I(z_{e-\alpha^\vee}^B)$. By the previous claim, we have $\mu\in R_Q^+$. By definition of $Q$, this means that $\alpha$ is orthogonal to $\mu$. Thus, we have $s_\alpha(\mu)=\mu>0$ and $\mu\notin I(s_\alpha)$. This means that Equation~\eqref{eq:inversioncap} is indeed true.
\end{proof}

\begin{thm}
\label{thm:gencascade:main2}

Let $e\in\Pi_B$. Then we have the following formulas:

\begin{align*}
I(z_e^B)&=\coprod_{\alpha\in\mathcal{B}_{R,e}}I(s_\alpha)\,,\\
I(z_e^B)\setminus R_P^+&=\coprod_{\alpha\in\mathcal{B}_{R,e}\setminus R_P^+}\left(I(s_\alpha)\setminus R_P^+\right)\,.
\end{align*}

\end{thm}

\begin{rem}

As in Remark~\ref{rem:general2}, we see that the first formula in Theorem~\ref{thm:gencascade:main2} is a generalization of \cite[Equation~(1.5)]{kostant}. Indeed, we recover the formula
$$
R^+=\coprod_{\alpha\in\mathcal{B}_{R,d_{G/B}}}I(s_\alpha)
$$
where $\mathcal{B}_{R,d_{G/B}}$ is the ordinary cascade of orthogonal roots by setting $e=d_{G/B}$ (cf. Remark~\ref{rem:general}).

\end{rem}

\begin{proof}[Proof of Theorem~\ref{thm:gencascade:main2}]

Let us first note that the second formula follows directly from the first by subtracting $R_P^+$. Indeed, we have $I(s_\alpha)\subseteq R_P^+$ for all $\alpha\in\mathcal{B}_{R,e}\cap R_P^+$ by \cite[5.5, Theorem~(b)]{humphreys3}. We now prove the first formula.

Let $(\alpha_1,\ldots,\alpha_r)$ be a greedy decomposition of $e$. For brevity, let $\alpha=\alpha_1$. The situation is now precisely the same as in the proof of Theorem~\ref{thm:gencascade:main} and we can freely use the formulas we worked out there. By Equation~\eqref{eq:zeequalsze}, \eqref{eq:inversioncap} and \cite[Proposition~3.2: $(d)\Leftrightarrow(e)$]{curvenbhd2}, 
%I suppress that I again use that both $s_\alpha$ and $z_{e-\alpha^\vee}^B$ are involutions -- this is only confusing now
we have
\begin{equation}
\label{eq:inclusion}
I(s_\alpha)\amalg I(z_{e-\alpha^\vee}^B)\subseteq I(z_e^B)\,.
\end{equation}

We now perform an induction on the length of the greedy decomposition of $e$ -- the first formula of Theorem~\ref{thm:gencascade:main2} being obvious whenever the greedy decomposition of $e$ has length zero.
%i.e. if $e=0$ and $\mathcal{B}_{R,e}=\emptyset$ as above.
By the induction hypothesis applied to $e-\alpha^\vee$
%the length of a greedy decomposition of $e-\alpha^\vee$ is $r-1$ as above
and Equation~\eqref{eq:bequalsb2}, we find that
$$
I(z_{e-\alpha^\vee}^B)=\coprod_{\mu\in\mathcal{B}_{R,e}\setminus\{\alpha\}}I(s_\mu)\,.
$$
Combined with the last equation, Inclusion~\eqref{eq:inclusion} gives the inclusion
$$
\coprod_{\mu\in\mathcal{B}_{R,e}}I(s_\mu)\subseteq I(z_e^B)\,.
$$
To prove that this inclusion is actually an equality (and hence Theorem~\ref{thm:gencascade:main2}), it suffices to show that the left and right set have the same cardinality. But this follows from Theorem~\ref{thm:gencascade:main} and the triangle inequality:
\begin{equation*}
\ell(z_e^B)\leq\sum_{\mu\in\mathcal{B}_{R,e}}\ell(s_\mu)\,.\qedhere
\end{equation*}
\end{proof}

\begin{cor}[Length additivity in generalized cascades of orthogonal roots]
\label{cor:gencascade:main}

Let $e\in\Pi_B$. Then we have the following formulas:
\begin{align*}
\ell(z_e^B)&=\sum_{\alpha\in\mathcal{B}_{R,e}}\ell(s_\alpha)\,,\\
%&=(c_1(G/B),e)-\mathrm{card}(\mathcal{B}_{R,e})\\
\ell(z_e^BW_P)&=\sum_{\alpha\in\mathcal{B}_{R,e}\setminus R_P^+}\ell(s_\alpha W_P)\,.
\end{align*}

\end{cor}

\begin{proof}

This corollary follows directly from Theorem~\ref{thm:gencascade:main2} by taking the cardinality of the involved sets.
\end{proof}

\begin{rem}

Several preliminary attempts towards a suitable length additivity theorem were already done in \cite[Lemma~8.9]{thesis}. Corollary~\ref{cor:gencascade:main} can be seen as final generalization of these.

\end{rem}

\todo[inline,color=green]{I want to add a reference to \cite{thesis} once I can explain in which sense I proved $P$-length additivity there. That requires the theory of lifting first: lifting of $d_X$ to $d_{G/B}$ stays the main example. Maybe even Kostant has a statement which is similar\ldots and which I can refer to here or after Theorem~\ref{thm:gencascade:main2}\ldots not to the thesis\ldots because, there, I presented bad style with unnecessarily involved proofs.} 

\todo[inline,color=green]{What I did in the thesis was as follows: Let $F$ be a subset of $\mathcal{B}_{R,d_{G/B}}\setminus R_P^+$. Then we have
$$
\ell\left(\left(\prod_{\alpha\in F}s_\alpha\right)W_P\right)=\sum_{\alpha\in F}\ell(s_\alpha W_P)\,.
$$
I did this formula for the case of a maximal parabolic subgroup $P$ with an unnecessary involved proof. Here, it is not important to mention this.}

\begin{cor}
\label{cor:gencascade:main2}

Let $e\in\Pi_B$. Then we have the following formula:
$$
\ell(z_e^B)=(c_1(G/B),e)-\mathrm{card}(\mathcal{B}_{R,e})\,.
$$

\end{cor}

\begin{proof}

By Theorem~\ref{thm:gencascade}\eqref{item:gencascadebcosmall}, we know that all $\alpha\in\mathcal{B}_{R,e}$ are $B$-cosmall. Hence, it follows from \cite[Theorem~6.1: $(a)\Leftrightarrow(b)$]{curvenbhd2} that $\ell(s_\alpha)=(c_1(G/B),\alpha^\vee)-1$ for all $\alpha\in\mathcal{B}_{R,e}$. Since there are no repeated entries in a greedy decomposition of $e$ (\cite[Remark~8.4]{minimaldegrees}), we clearly have $e=\sum_{\alpha\in\mathcal{B}_{R,e}}\alpha^\vee$. The corollary follows from these facts and Corollary~\ref{cor:gencascade:main}.
\end{proof}

\section{Positivity in generalized cascades of orthogonal roots}
\label{sec:positivity}

In this section, we work out a positivity statement in generalized cascades of orthogonal roots (Theorem~\ref{thm:positivity}). This theorem closes the general properties of generalized chain cascades which we started to investigate in Section~\ref{sec:gencascade}. We apply this theorem to degrees $e\in\Pi_B$ where $z_e^B$ is the maximal representative in $z_e^BW_P$ (cf. Theorem~\ref{thm:positivity2}).
%
%the situation of elements $z_e^B$ where $e\in\Pi_B$ which are maximal representatives in $z_e^BW_P$ (cf. Theorem~). 
%
In Section~\ref{sec:lifting}, we set up a framework where such degrees naturally occur. Namely, they occur as the lifting of a degree $d\in\Pi_P$ (cf. Fact~\ref{fact:lifting}\eqref{item:maxrep}).  
%
%As an example of those elements, the reader can bear in mind $z_e^B$ where $e$ is the lifting of a degree $d\in\Pi_P$
%
Finally, we give an application specific to the combinatorics in type $\mathsf{A}$ which might be useful in different contexts (Theorem~\ref{thm:typea}).
This section is not strictly necessary to understand the main result on quasi-homogeneity. The impatient reader can skip it.

\todo[inline,color=green]{{\color{black} Relevant for the relative setting, plausible because maximal representatives occur as lifts as discussed in the next section, preliminary result to achieve one of our main examples: in type $\mathsf{A}$ minimal degrees are admissible, first main theorem (the actual positivity statement) closes the general properties of generalized chain cascades -- the other results are more specific.

This section is (after the correction) not necessary to understand the proof of Theorem~\ref{thm:main}. The impatient reader can skip it.}}

\begin{thm}[Positivity in generalized cascades of orthogonal roots]
\label{thm:positivity}

Let $e\in\Pi_B$. Let $\hat{e}=\sum_{\alpha\in\mathcal{B}_{R,e}\cap R_P^+}\alpha^\vee$. Let $\gamma\in I(z_{\hat{e}}^B)$. Then we have
$$
(\alpha,\gamma)\geq 0\text{ for all }\alpha\in\mathcal{B}_{R,e}\setminus R_P^+\,.
$$

\end{thm}

\begin{rem}

Note that the choice of $e\in\Pi_B$ in Theorem~\ref{thm:positivity} does not depend on $P$. Thus, for a fixed $e\in\Pi_B$, we get a bunch of positivity results by varying $P$.

\end{rem}

\begin{proof}[Proof of Theorem~\ref{thm:positivity}]

Let $e$, $\hat{e}$, $\gamma$ be as in the statement. Suppose for a contradiction, there exists an $\alpha\in\mathcal{B}_{R,e}\setminus R_P^+$ such that $(\alpha,\gamma)<0$. By \cite[Proposition~3.10(7), 4.4(9)]{minimaldegrees}, we know that
$
\mathcal{B}_{R,\hat{e}}=\mathcal{B}_{R,e}\cap R_P^+
$.
Thus, the assumption $\gamma\in I(z_{\hat{e}}^B)$ and Theorem~\ref{thm:gencascade:main2} imply that there exists an $\alpha'\in\mathcal{B}_{R,e}\cap R_P^+$ such that $\gamma\in I(s_{\alpha'})$. Since $\alpha'$ is member of a generalized cascade of orthogonal roots, we know by Theorem~\ref{thm:gencascade}\eqref{item:gencascadebcosmall} that $\alpha'$ is $B$-cosmall. It is clear that $\alpha'\neq\gamma$, since otherwise we had $(\alpha,\gamma)=0$ by Theorem~\ref{thm:gencascade}\eqref{item:stronglyorthogonal} (note that $\alpha\neq\alpha'$ since $\alpha\in\mathcal{B}_{R,e}\setminus R_P^+$ and $\alpha'\in\mathcal{B}_{R,e}\cap R_P^+$) -- contradicting our initial assumption that $(\alpha,\gamma)<0$. Therefore, it follows by \cite[Theorem~6.1: $(a)\Leftrightarrow(c)$]{curvenbhd2} that $(\gamma,\alpha'^\vee)=1$.

We now consider the root $\delta$ defined by $\delta=s_\alpha s_{\alpha'}(\gamma)$. Since $\alpha$ and $\alpha'$ are orthogonal (cf. Theorem~\ref{thm:gencascade}\eqref{item:stronglyorthogonal}) we compute
\begin{equation}
\label{eq:smalldelta}
\delta=\gamma-\alpha'-(\gamma,\alpha^\vee)\alpha\,.
\end{equation}
Since $\alpha\in\mathcal{B}_{R,e}\setminus R_P^+$, there exists a $\beta\in\Delta\setminus\Delta_P$ such that $(\omega_\beta,\alpha^\vee)>0$.
Since $\alpha'\in R_P^+$, we have $I(s_{\alpha'})\subseteq R_P^+$ (cf. \cite[5.5, Theorem~(b)]{humphreys3}), in particular $\gamma\in R_P^+$. Thus, we have $s_{\alpha'}(\gamma)=\gamma-\alpha'\in R_P^-$ and consequently $(\omega_\beta,s_{\alpha'}(\gamma)^\vee)=0$. Altogether, it follows from the assumption $(\alpha,\gamma)<0$ and Equation~\eqref{eq:smalldelta} that $(\omega_\beta,\delta^\vee)>0$ and thus that $\delta$ is a positive root. We can now reformulate Equation~\eqref{eq:smalldelta} as
\begin{equation}
\label{eq:smalldelta2}
-(\gamma,\alpha^\vee)\alpha>\alpha'-\gamma>0\,.
\end{equation}

\begin{proof}[First case: The roots $\alpha$ and $\alpha'$ are comparable]\renewcommand{\qedsymbol}{$\triangle$}

Suppose that $\alpha$ and $\alpha'$ are comparable. Let $\alpha^*$ be the maximum and $\alpha_*$ be the minimum of the totally ordered set $\{\alpha,\alpha'\}$, i.e. we have $\{\alpha_*,\alpha^*\}=\{\alpha,\alpha'\}$. Let 
$$
\tilde{e}=\sum_{\mu\in\mathcal{B}_{R,e}\colon\mu\leq\alpha^*}\mu^\vee\,.
$$
By \cite[Proposition~4.4(9)]{minimaldegrees}, we have $\tilde{e}\in\Pi_B$. By \cite[Proposition~3.10(7)]{minimaldegrees}, it follows that $\Delta(\tilde{e})=\Delta(\alpha^*)$. Thus, $\tilde{e}$ is a connected degree with $\alpha(\tilde{e})=\alpha^*$. By \cite[Proposition~3.10(7), 3.16, Theorem~8.1]{minimaldegrees}, it follows that
\begin{equation}
\label{eq:orthogonal}
\Delta(\alpha_*)\subseteq\Delta(\tilde{e}-\alpha^{*\vee})\subseteq\Delta_{\alpha^*}^\circ\,.\qedhere
\end{equation}
\end{proof}

\begin{proof}[First sub-case: $\alpha<\alpha'$, i.e. $\alpha_*=\alpha$ and $\alpha^*=\alpha'$]\renewcommand{\qedsymbol}{$\triangle$}

By Inequality~\eqref{eq:smalldelta2}, it follows that $\Delta(\alpha'-\gamma)\subseteq\Delta(\alpha)$. By Inclusion~\eqref{eq:orthogonal}, we then have $\Delta(\alpha'-\gamma)\subseteq\Delta_{\alpha'}^\circ$. This means that $(\alpha'-\gamma,\alpha'^\vee)=0$ and thus $(\gamma,\alpha'^\vee)=2$. But we already saw that $(\gamma,\alpha'^\vee)=1$ -- a contradiction.
\end{proof}

\begin{proof}[Second sub-case: $\alpha>\alpha'$, i.e. $\alpha_*=\alpha'$ and $\alpha^*=\alpha$]\renewcommand{\qedsymbol}{$\triangle$}

%In this case, we have
%$$
%-\alpha'-(\gamma,\alpha^\vee)\alpha\geq\alpha-\alpha'>0
%$$

By Inequality~\eqref{eq:smalldelta2}, we have
\begin{equation}
\label{eq:alpha'}
\alpha'>\alpha'-\gamma>0
\end{equation}
and thus $\Delta(\alpha'-\gamma)\subseteq\Delta(\alpha')$. Inclusion~\eqref{eq:orthogonal} implies that $\Delta(\alpha'-\gamma)\subseteq\Delta_\alpha^\circ$. This means that $(\alpha,\alpha'-\gamma)=0$ and thus $(\alpha,\gamma)=0$ (cf. Theorem~\ref{thm:gencascade}\eqref{item:stronglyorthogonal}). But by assumption we have $(\alpha,\gamma)<0$ -- a contradiction.
\end{proof}

\begin{proof}[Second case: The roots $\alpha$ and $\alpha'$ are incomparable]\renewcommand{\qedsymbol}{$\triangle$}

Suppose that $\alpha$ and $\alpha'$ are incomparable. Then there exists no $\varphi\in R^+$ such that $\alpha,\alpha'\in C_{R,e}(\varphi)$ (cf. Theorem~\ref{thm:gencascade}\eqref{item:totallyordered}). Theorem~\ref{thm:gencascade}\eqref{item:totallydisjoint2} implies that $R(\alpha)$ and $R(\alpha')$ are totally disjoint. On the other hand, Inequality~\eqref{eq:smalldelta2}, \eqref{eq:alpha'} show that the positive root $\alpha'-\gamma$ belongs both to $R(\alpha)$ and $R(\alpha')$ -- a contradiction.
\end{proof}

All in all, this shows that our initial assumption that there exists an $\alpha\in\mathcal{B}_{R,e}\setminus R_P^+$ such that $(\alpha,\gamma)<0$ must be false. In other words, the statement of the theorem is true.
\end{proof}

\begin{lem}
\label{lem:positivity}

Let $e\in\Pi_B$. Let $\gamma\in R_P^+$. Then there exists at most one $\alpha\in\mathcal{B}_{R,e}\setminus R_P^+$ such that $(\alpha,\gamma)>0$.

\end{lem}

\begin{rem}

Lemma~\ref{lem:positivity} says in particular that all or all except one inequality in Theorem~\ref{thm:positivity} are actually equalities. This follows from the fact that $I(z_{\hat{e}}^B)\subseteq R_P^+$ where $\hat{e}$ depends on $e$ and $P$ as in Theorem~\ref{thm:positivity}. We saw this fact in the proof of Theorem~\ref{thm:positivity}. It is also easy to deduce more directly without invoking Theorem~\ref{thm:gencascade:main2}.

\end{rem}

\begin{proof}[Proof of Lemma~\ref{lem:positivity}]

Let $e$ and $\gamma$ be as in the statement. Suppose for a contradiction there exist two distinct roots $\alpha,\alpha'\in\mathcal{B}_{R,e}\setminus R_P^+$ such that $(\alpha,\gamma)>0$ and $(\alpha',\gamma)>0$. In view of the assumption, it is easy to see that $\gamma\in I(s_\alpha)\cap I(s_{\alpha'})$. (Indeed, $s_{\alpha}(\gamma)$ and $s_{\alpha'}(\gamma)$ must contain a simple root in $\Delta\setminus\Delta_P$ with negative coefficient in their expression as linear combination of simple roots.) This contradicts the disjointness result in Theorem~\ref{thm:gencascade:main2}.
\end{proof}

\begin{thm}
\label{thm:positivity2}

Let $e\in\Pi_B$. Suppose that $z_e^B$ is the maximal representative in $z_e^BW_P$. 
%Then we have the inclusion
%$$
%\{\gamma\in R_P^+\mid\exists\alpha\in\mathcal{B}_{R,e}\setminus R_P^+\colon(\alpha,\gamma)<0\}\subseteq
%\{\gamma\in R_P^+\mid\exists\alpha\in\mathcal{B}_{R,e}\setminus R_P^+\colon(\alpha,\gamma)>0\}
%$$
Let $\gamma\in R_P^+$. Suppose there exists an $\alpha\in\mathcal{B}_{R,e}\setminus R_P^+$ such that $(\alpha,\gamma)<0$. Then there exists an $\alpha'\in\mathcal{B}_{R,e}\setminus R_P^+$ such that $(\alpha',\gamma)>0$.

\end{thm}

\begin{rem}
\label{rem:unique} 

Lemma~\ref{lem:positivity} shows that the $\alpha'$ in Theorem~\ref{thm:positivity2} is unique.

\end{rem}

\begin{proof}[Proof of Theorem~\ref{thm:positivity2}]

Let $e$ be as in the statement. Let 
$$
\hat{e}=\sum_{\alpha\in\mathcal{B}_{R,e}\cap R_P^+}\alpha^\vee\text{ and }
\tilde{e}=\sum_{\alpha\in\mathcal{B}_{R,e}\setminus R_P^+}\alpha^\vee\,.
$$
Theorem~\ref{thm:positivity} implies the following inclusion:
\begin{equation}
\label{eq:consofthm:positivity}
\{\gamma\in R_P^+\mid\exists\alpha\in\mathcal{B}_{R,e}\setminus R_P^+\colon(\alpha,\gamma)<0\}\subseteq R_P^+\setminus I(z_{\hat{e}}^B)\,.
\end{equation}
By assumption $z_e^B$ is the maximal representative in $z_e^BW_P$. Therefore we have $R_P^+\subseteq I(z_e^B)$. Theorem~\ref{thm:gencascade:main2} and \cite[Proposition~3.10(7), 4.4(9)]{minimaldegrees} show that $I(z_e^B)=I(z_{\hat{e}}^B)\amalg I(z_{\tilde{e}}^B)$. Both facts together yield that $R_P^+\setminus I(z_{\hat{e}}^B)=I(z_{\tilde{e}}^B)\cap R_P^+$. Finally, it is easy to see that
\begin{equation}
\label{eq:consequence2}
I(z_{\tilde{e}}^B)\cap R_P^+=\{\gamma\in R_P^+\mid\exists\alpha\in\mathcal{B}_{R,e}\setminus R_P^+\colon(\alpha,\gamma)>0\}\,.
\end{equation}
(This equation is a direct consequence of Theorem~\ref{thm:gencascade:main2} applied to $\tilde{e}$. One only has to note that we have $\mathcal{B}_{R,\tilde{e}}=\mathcal{B}_{R,e}\setminus R_P^+$ by \cite[loc. cit.]{minimaldegrees}. For the proof of the inclusion \enquote{$\supseteq$} one may use similar arguments as in the proof of Lemma~\ref{lem:positivity}. Note also, as a consequence of Lemma~\ref{lem:positivity}, we can equally well write $\exists!\alpha$ instead of $\exists\alpha$ on the right side of Equation~\eqref{eq:consequence2}.) Inclusion~\eqref{eq:consofthm:positivity} and Equation~\eqref{eq:consequence2} together yield the inclusion
$$
\{\gamma\in R_P^+\mid\exists\alpha\in\mathcal{B}_{R,e}\setminus R_P^+\colon(\alpha,\gamma)<0\}\subseteq\{\gamma\in R_P^+\mid\exists\alpha\in\mathcal{B}_{R,e}\setminus R_P^+\colon(\alpha,\gamma)>0\}\,.
$$
This inclusion shows all we claimed in the statement.
\end{proof}

\begin{thm}
\label{thm:typea}

Assume that $R$ is of type $\mathsf{A}$. Let $e\in\Pi_B$. Suppose that $z_e^B$ is the maximal representative in $z_e^BW_P$. Let $\gamma\in R_P^+$. Then there exists at most one $\alpha\in\mathcal{B}_{R,e}\setminus R_P^+$ such that $(\alpha,\gamma)<0$.

\end{thm}

{\color{black}\begin{rem}

Note that Theorem~\ref{thm:typea} is specific to type $\mathsf{A}$. In general, there may exist several $\alpha$'s as in the statement of Theorem~\ref{thm:typea}. The reader can find examples for this behavior in type $\mathsf{D}_4$ or type $\mathsf{E}_6$ for the degree $e=d_{G/B}\in\Pi_B$.

%We give two examples for this behavior.

\end{rem}}

\begin{comment}

\begin{ex}

Let $R$ be of type $\mathsf{D}_4$ or type $\mathsf{E}_6$. Let $P$ be the parabolic subgroup of $G$ such that $\Delta_P$ is as in the second row of Table~\ref{table:example}. Let $e=d_{G/B}\in\Pi_B$. Clearly, $z_e^B=w_o$ is the maximal representative in $z_e^BW_P$. Let $\gamma\in R_P^+$ be as in the third row of Table~\ref{table:example}. The fourth row of Table~\ref{table:example} lists all $\alpha\in\mathcal{B}_{R,e}\setminus R_P^+$ such that $(\alpha,\gamma)<0$. For this particular choice of $\gamma$, we see that there are three such $\alpha$'s, in particular more than one. This shows that Theorem~\ref{thm:typea} fails in type $\mathsf{D}_4$ and type $\mathsf{E}_6$.

\end{ex}

\end{comment}

\begin{comment}

\begin{table}

\begin{tabular}{llll}

Type & $\mathsf{D}_4$ & $\mathsf{E}_6$ & $\mathsf{G}_2$ \\\hline\hline
$\Delta_P$ & $\beta_2$ & $\Delta\setminus\{\beta_4\}$ & $\beta_2$ \\
$\gamma\in R_P^+$ & $\beta_2$ & $\beta_2$ & $\beta_2$ \\
$\alpha\in\mathcal{B}_{R,d_{G/B}}\setminus R_P^+\colon (\alpha,\gamma)<0$ & $\beta_1$, $\beta_3$, $\beta_4$ & $101111$, $001110$, $\beta_4$ & $\beta_1$

\end{tabular}

\caption{Counterexample to Theorem~\ref{thm:typea} for $R$ not of type $\mathsf{A}$. We denote by $\Delta=\{\beta_1,\beta_2,\beta_3,\ldots\}$ the set of simple roots with the numbering as in \cite[Plate I-IX]{bourbaki_roots}. We denote by $abc\cdots$ the root $a\beta_1+b\beta_2+c\beta_3+\cdots$.}

\label{table:example}

\end{table}

\end{comment}

\begin{proof}[Proof of Theorem~\ref{thm:typea}]

Assume that $R$ is of type $\mathsf{A}_n$ for some $n\geq 1$. Let $\Delta=\{\beta_1,\ldots,\beta_n\}$ be the set of simple roots with the numbering as in \cite[Plate~I]{bourbaki_roots}.

\begin{proof}[First case: Assume that $\gamma\in\Delta_P$]\renewcommand{\qedsymbol}{$\triangle$}

Let $\gamma=\beta_i$ for some $1\leq i\leq n$. Assume for a contradiction that there exist two distinct roots $\alpha,\alpha'\in\mathcal{B}_{R,e}\setminus R_P^+$ such that $(\alpha,\gamma)<0$ and $(\alpha',\gamma)<0$. Write 
\begin{align*}
\alpha &=\beta_{i_*(\alpha)}+\cdots+\beta_{i^*(\alpha)}\\
\alpha'&=\beta_{i_*(\alpha')}+\cdots+\beta_{i^*(\alpha')}
\end{align*}
for some integers $1\leq i_*(\alpha)\leq i^*(\alpha)\leq n$ and $1\leq i_*(\alpha')\leq i^*(\alpha')\leq n$. The assumption $(\alpha,\gamma)<0$ and $(\alpha',\gamma)<0$ implies that we have either
\begin{align*}
i_*(\alpha)=i_*(\alpha')=i+1&\text{ or }i^*(\alpha)=i^*(\alpha')=i-1\text{ or }\\
i_*(\alpha)=i+1\text{, }i^*(\alpha')=i-1&\text{ or }i^*(\alpha)=i-1\text{, }i_*(\alpha')=i+1\,.
\end{align*}
In the first two cases, we have $(\alpha,\alpha')>0$ - contrary to Theorem~\ref{thm:gencascade}\eqref{item:stronglyorthogonal} according to which we have $(\alpha,\alpha')=0$. The last two cases are symmetric. By replacing $\alpha$ and $\alpha'$ if necessary, we may assume that we are in the fourth case. 

By Theorem~\ref{thm:positivity2} there exists an $\alpha''\in\mathcal{B}_{R,e}\setminus R_P^+$ such that $(\alpha'',\gamma)>0$. The root $\alpha''$ is clearly distinct from $\alpha$ and $\alpha'$. Write
$$
\alpha''=\beta_{i_*(\alpha'')}+\cdots+\beta_{i^*(\alpha'')}
$$
for some integers $1\leq i_*(\alpha'')\leq i^*(\alpha'')\leq n$. The inequality $(\alpha'',\gamma)>0$ implies that we have either $i_*(\alpha'')=i$ or $i^*(\alpha'')=i$. In the first case we have $(\alpha,\alpha'')<0$. In the second case we have $(\alpha',\alpha'')<0$. Both conclusions are contrary to Theorem~\ref{thm:gencascade}\eqref{item:stronglyorthogonal} according to which we have $(\alpha,\alpha'')=(\alpha',\alpha'')=0$.
This shows that the conclusion of Theorem~\ref{thm:typea} holds for all $\gamma\in\Delta_P$.
\end{proof}

\begin{proof}[Second case: Assume that $\gamma\in R_P^+$ is arbitrary]\renewcommand{\qedsymbol}{$\triangle$}

By the first case, Lemma~\ref{lem:positivity} and Theorem~\ref{thm:positivity2}, we have
$
\sum_{\alpha\in\mathcal{B}_{R,e}\setminus R_P^+}(\beta,\alpha^\vee)\geq 0
%\begin{cases}
%0&\text{if there exists an $\alpha\in\mathcal{B}_{R,e}\setminus R_P^+$ such that $(\alpha,\gamma)<0$}\\
%1&\text{otherwise}
%\end{cases}
$
for all $\beta\in\Delta_P$. By summing over all $\beta\in\Delta(\gamma)$, we also obtain 
\begin{equation}
\label{eq:positive}
\sum_{\alpha\in\mathcal{B}_{R,e}\setminus R_P^+}(\gamma,\alpha^\vee)\geq 0\,.
\end{equation}
Let $l$ be the number of roots $\alpha\in\mathcal{B}_{R,e}\setminus R_P^+$ such that $(\alpha,\gamma)<0$. If $l=0$, there is nothing to prove. Assume that $l\geq 1$. We have to show that we even have $l=1$. But by Lemma~\ref{lem:positivity} and Theorem~\ref{thm:positivity2} we clearly have 
$$
\sum_{\alpha\in\mathcal{B}_{R,e}\setminus R_P^+}(\gamma,\alpha^\vee)=-l+1\,.
$$
Together with Inequality~\eqref{eq:positive} this gives $l\leq 1$ and thus $l=1$ -- as desired.
\end{proof}

The second case shows that the conclusion of Theorem~\ref{thm:typea} holds for all $\gamma\in R_P^+$.
\end{proof}

\begin{cor}
\label{cor:typea}

Assume that $R$ is of type $\mathsf{A}$. Let $e\in\Pi_B$. Suppose that $z_e^B$ is the maximal representative in $z_e^BW_P$. Let $\gamma\in R_P^+$. Then we have
$$
\sum_{\alpha\in\mathcal{B}_{R,e}\setminus R_P^+}(\gamma,\alpha^\vee)\in\{0,1\}\,.
$$

\end{cor}

\begin{proof}

This follows directly from Lemma~\ref{lem:positivity} and Theorem~\ref{thm:positivity2}, \ref{thm:typea}.
\end{proof}

\section{The lifting of a minimal degree}
\label{sec:lifting}

In this section, we introduce the notion of the lifting of a minimal degree. This notion basically prepares the construction of the diagonal curve in the relative setting modulo $P$ by associating to a degree $d\in\Pi_P$ the generalized cascade of orthogonal roots $\mathcal{B}_{R,e}$ where $e\in\Pi_B$ is the lifting of $d$. In case of a generalized complete flag variety $G/B$ (i.e. $P=B$) the step of passing from a minimal degree 
%$d$
to its lifting 
%$e$ 
is superfluous. It is only necessary in a more general context.

\begin{comment}

We investigate the basic properties of liftings in order to prove, with the help of the results established in Section~\ref{sec:positivity}, that three classes of examples of degrees satisfy a certain positivity property which is the defining property of the so-called $P$-admissible degrees (cf. Proposition~\ref{prop:padmissible}). The $P$-admissible degrees are a subclass of minimal degrees in $\Pi_P$ for which we are able to obtain a quasi-homogeneity result in the relative setting modulo $P$. All minimal degrees in $\Pi_B$ are $B$-admissible. In this way, we are able to extend the situation from $G/B$ to $G/P$.  

\end{comment}

{\color{black} The nontrivial input in this section is certainly \cite[Corollary~3]{postnikov} which states that there exists a unique minimal degree in the quantum product of two Schubert classes in $H^*(G/B)$. This result guarantees in particular the uniqueness of the lifting of a minimal degree and also has other consequences concerning the uniqueness of minimal degrees in the relative setting modulo $P$ which we investigate in the beginning of this section.}

\begin{notation}

Let $d$ be a degree. Let $Q$ be a parabolic subgroup of $G$ containing $P$. We denote by $d_Q$ the image of $d$ under the natural map $H_2(X)\to H_2(G/Q)$, i.e. we have $d_Q=d+\mathbb{Z}\Delta_Q^\vee$. We will mostly apply this notation for the relative situation $B\subseteq P$, i.e. if $e\in H_2(G/B)$ is a degree, we write $e_P=e+\mathbb{Z}\Delta_P^\vee$.\footnote{With the terminology of \cite[Subsection~4.1]{minimaldegrees}, the degree $e_P$ is the restriction of $e$.}

\end{notation}

\begin{defn}
\label{def:lifting}

Let $d\in\Pi_P$. By \cite[Corollary~3]{postnikov} and \cite[Definition~4.1, Theorem~5.15]{minimaldegrees} there exists a unique minimal element of the set
$$
\{e\in H_2(G/B)\text{ a degree such that }z_d^Pw_P\preceq z_e^B\}\,.
$$
This unique minimal element is called the lifting of $d$.

\end{defn}

{\color{black}\begin{ex}

The degree $d_{G/B}$ is the lifting of $d_X$. This follows directly from the definition of $d_X$ and $d_{G/B}$ in Notation~\ref{not:dx}

\end{ex}}

\begin{ex}

%The lifting of a degree $e\in H_2(G/B)$ is the degree itself.

The lifting of a degree $e\in\Pi_B$ is the degree $e$ itself.

\end{ex}

\begin{fact}
\label{fact:lifting}

Let $d\in\Pi_P$. Let $e$ be the lifting of $d$. Then the following is true.

\begin{enumerate}

\item 
\label{item:defoflifting}

We have $e\in\Pi_B$ and $z_d^Pw_P=z_e^B$.

\item
\label{item:maxrep}

The element $z_e^B$ is the maximal representative in $z_e^BW_P$.

\item
\label{item:ePequalsd}

We have $e_P=d$.

\end{enumerate}

\end{fact}

\begin{proof}

Item~\eqref{item:maxrep} is clear from the equation $z_d^Pw_P=z_e^B$ in Item~\eqref{item:defoflifting}. All other claims follow from \cite[Corollary~3]{postnikov} and \cite[Definition~4.1, Theorem~4.13]{minimaldegrees} applied to the relative situation $B\subseteq P$.
\end{proof}

\begin{thm}[Consequence of {\cite[Corollary~3]{postnikov}}]
\label{thm:uniqueness}

Let $d,d'\in\Pi_P$ such that $z_d^P\preceq z_{d'}^P$. Then we have $d\leq d'$.

\end{thm}

\begin{proof}

Let $d$ and $d'$ be as in the statement. Let $e$ be the lifting of $d$ and let $e'$ be the lifting of $d$. As $z_d^P\preceq z_{d'}^P$ by definition, Fact~\ref{fact:lifting}\eqref{item:defoflifting} implies that $z_e^B\preceq z_{e'}^B$. Since $e\in\Pi_B$ by Fact~\ref{fact:lifting}\eqref{item:defoflifting}, \cite[Corollary~3]{postnikov} and \cite[Proposition~4.4(3)]{minimaldegrees} imply that $e\leq e'$. Finally, Fact~\ref{fact:lifting}\eqref{item:ePequalsd} shows that $d\leq d'$ -- as required.
\end{proof}

\begin{cor}

Let $d,d'\in\Pi_P$ such that $z_d^P=z_{d'}^P$. Then we have $d=d'$.

\end{cor}

\begin{proof}

This is an immediate corollary of Theorem~\ref{thm:uniqueness}.
\end{proof}

\begin{cor}
\label{cor:uniqueness}

Let $d\in\Pi_P$. Then $d$ is the unique minimal element of the set
$$
\{d'\text{ a degree such that }z_d^P\preceq z_{d'}^P\}\,.
$$

\end{cor}

{\color{black}\begin{rem}

That $d\in\Pi_P$ is a minimal element of the prescribed set is content of 
%the definition of minimal degrees we chose in the section on minimal degrees.
Definition~\ref{def:minimaldegrees}. Corollary~\ref{cor:uniqueness} makes a statement about the uniqueness which is nontrivial.

\end{rem}

\begin{rem}

Corollary~\ref{cor:uniqueness} shows in particular that $d_X$ is the unique minimal element of the set
$$
\{d\text{ a degree such that }w_oW_P=z_d^PW_P\}\,.
$$
This result was found independently from \cite[Corollary~3]{postnikov} in \cite[Definition~4.1, Theorem~4.7]{minimaldegrees} (cf. Notation~\ref{not:dx}).

\end{rem}

%\begin{rem}

%Mention that both the uniqueness of $d_X$ as well as $d\leq d_X$ for all $d\in\Pi_P$ can now be deduced from the previous theorem and its corollary. But only because we used results from the literature which where not known to the author before. Even if these statements are trivial now, the techniques developed in \cite{minimaldegrees} in order to prove them (such as orthogonality relations) are not trivial at all and needed all over in this work. They have their justification.

%\end{rem}

\begin{rem}[For the reader familiar with quantum cohomology]
\label{rem:uniqueness}

%Explain what the previous corollary means in terms of quantum cohomology with a reference to the section on minimal degrees where I define $\Pi_P$ and also explain what it means geometrically. (To the intention of the reader familiar with quantum cohomology.)

By Remark~\ref{rem:quantum}, we know that a degree $d\in\Pi_P$ is a minimal degree in $\sigma_{z_d^P}\star\mathrm{pt}$. In view of \cite[Definition~4.1, Theorem~5.15]{minimaldegrees}, Corollary~\ref{cor:uniqueness} says that a degree $d\in\Pi_P$ is even the unique minimal degree in $\sigma_{z_d^P}\star\mathrm{pt}$. In particular, the degree $d_X$ is the unique minimal degree in $\mathrm{pt}\star\mathrm{pt}$.

\end{rem}}

\begin{proof}[Proof of Corollary~\ref{cor:uniqueness}]

Let $d\in\Pi_P$. Let $d''$ be a minimal element of the set 
$$
\{d'\text{ a degree such that }z_d^P\preceq z_{d'}^P\}\,.
$$
We have to show that $d=d''$. By \cite[Definition~4.1, Proposition~4.4(10)]{minimaldegrees}, we know that $d''\in\Pi_P$. By definition of $d''$, we also have $z_d^P\preceq z_{d''}^P$. Therefore, Theorem~\ref{thm:uniqueness} implies that $d\leq d''$. By the minimality of $d''$, this implies that $d=d''$ -- as claimed.
\end{proof}

\begin{cor}
\label{cor:dsmallerdx}

For all $d\in\Pi_P$ we have $d\leq d_X$.

\end{cor}

\begin{rem}

Corollary~\ref{cor:dsmallerdx} was found independently from \cite[Corollary~3]{postnikov} in \cite[Definition~4.1, Theorem~10.7]{minimaldegrees}

\end{rem}

\begin{proof}[Proof of Corollary~\ref{cor:dsmallerdx}]

Let $d\in\Pi_P$. By definition, we have $w_oW_P=z_{d_X}^PW_P$ and thus $z_d^P\preceq z_{d_X}^P$. As $d,d_X\in\Pi_P$, Theorem~\ref{thm:uniqueness} now implies that $d\leq d_X$.
\end{proof}

\begin{thm}
\label{thm:splitting1}

Let $d\in\Pi_P$. Let $e$ be the lifting of $d$. Let $\hat{e}=\sum_{\alpha\in\mathcal{B}_{R,e}\setminus R_P^+}\alpha^\vee$. A root $\alpha$ is a maximal root of $d$ if and only if $\alpha$ is a maximal root of $\hat{e}$.

\end{thm}

\begin{proof}

Let $\alpha$ be a maximal root of $d$. First, we show that $\alpha$ is also a maximal root of $\hat{e}$. By definition, $\alpha$ is $P$-cosmall and also $B$-cosmall. Hence, we have $z_{d(\alpha)}^Pw_P=s_\alpha\cdot w_P$ and $z_{\alpha^\vee}^B=s_\alpha$. From \cite[Proposition~3.1(c)]{curvenbhd2}, it follows that $z_{\alpha^\vee}^B\preceq z_{d(\alpha)}^Pw_P$. By definition, we have $d(\alpha)\leq d$ and hence $z_{d(\alpha)}^Pw_P\preceq z_d^Pw_P=z_e^B$ by \cite[Corollary~4.12(b)]{curvenbhd2} and Fact~\ref{fact:lifting}\eqref{item:defoflifting}. Altogether, it follows that $z_{\alpha^\vee}^B\preceq z_e^B$. By \cite[Proposition~4.4(6)]{minimaldegrees}, we have $\alpha^\vee\in\Pi_B$ and by Fact~\ref{fact:lifting}\eqref{item:defoflifting}, we have $e\in\Pi_B$. Hence, Theorem~\ref{thm:uniqueness} and the previous relation in the Bruhat order show that $\alpha^\vee\leq e$.

\begin{proof}[Claim: $\alpha$ is a maximal root of $e$]\renewcommand{\qedsymbol}{$\triangle$}

Indeed, let $\alpha'$ be a maximal root of $e$ such that $\alpha\leq\alpha'$. (As $\alpha^\vee\leq e$, we can clearly choose such an $\alpha'$.) By definition, we have $\alpha'^\vee\leq e$ and thus by Fact~\ref{fact:lifting}\eqref{item:ePequalsd} that $d(\alpha')\leq d$. Since $\alpha\leq\alpha'$ and $\alpha\in R^+\setminus R_P^+$ by definition, it is also clear that $\alpha'\in R^+\setminus R_P^+$. Since $\alpha$ is a maximal root of $d$, it follows that $\alpha=\alpha'$. In other words, $\alpha$ is a maximal root of $e$.
\end{proof}

As $\alpha$ is a maximal root of $e$, there exists a greedy decomposition of $e$ such that $\alpha$ occurs as its first entry. This means in particular that $\alpha\in\mathcal{B}_{R,e}$ and since $\alpha\in R^+\setminus R_P^+$ also that $\alpha\in\mathcal{B}_{R,e}\setminus R_P^+$. By \cite[Proposition~3.10(7), 4.4(9)]{minimaldegrees}, we have $\hat{e}\in\Pi_B$ and $\mathcal{B}_{R,\hat{e}}=\mathcal{B}_{R,e}\setminus R_P^+$. Thus, we have $\alpha\in\mathcal{B}_{R,\hat{e}}$ and consequently $\alpha^\vee\leq\hat{e}\leq e$. Since $\alpha$ is a maximal root of $e$ by the claim, the last inequality implies that $\alpha$ is also a maximal root of $\hat{e}$. This proves the implication from left to right.

To prove the other implication, let $\alpha$ be a maximal root of $\hat{e}$. We have to show that $\alpha$ is also a maximal root of $d$. By definition, there exists a greedy decomposition of $\hat{e}$ such that $\alpha$ occurs as its first entry. This means that $\alpha\in\mathcal{B}_{R,\hat{e}}=\mathcal{B}_{R,e}\setminus R_P^+$. (The last equality was already justified in the paragraph before.) In particular, we have $\alpha\in R^+\setminus R_P^+$. Moreover, as $\alpha^\vee\leq\hat{e}$ by definition, it follows from Fact~\ref{fact:lifting}\eqref{item:ePequalsd} that $d(\alpha)\leq\hat{e}_P=e_P=d$. (To see that $\hat{e}_P=e_P$, just note that $e=\sum_{\alpha\in\mathcal{B}_{R,e}}\alpha^\vee$ in view of \cite[Remark~8.4]{minimaldegrees}.) These facts show that we can choose a maximal root $\alpha'$ of $d$ such that $\alpha\leq\alpha'$. By the first implication already shown, $\alpha'$ is also a maximal root of $\hat{e}$. Since $\alpha$ is also a maximal root of $\hat{e}$ by our initial choice, it follows that $\alpha=\alpha'$. In other words, $\alpha$ is a maximal root of $d$. This proves the implication from right to left.
\end{proof}

\begin{cor}

Let $d\in\Pi_P$. Let $e$ be the lifting of $d$. Let $\hat{e}=\sum_{\alpha\in\mathcal{B}_{R,e}\setminus R_P^+}\alpha^\vee$. Then we have $\widetilde{\Delta}(d)=\Delta(\hat{e})$.

\end{cor}

\begin{proof}

Let $d_1,\ldots,d_k$ be the connected components of $d$. Let $\hat{e}_1,\ldots,\hat{e}_{k'}$ be the connected components of $\hat{e}$. By Theorem~\ref{thm:maxroots}, \ref{thm:splitting1}, it follows that $k=k'$ and
$$
\{\alpha(d_1),\ldots,\alpha(d_k)\}=\{\alpha(\hat{e}_1),\ldots,\alpha(\hat{e}_k)\}\,.
$$
By Notation~\ref{not:alphad} and Lemma~\ref{lem:connectedcomponents1}, we also have
$$
\widetilde{\Delta}(d)=\coprod_{i=1}^k\Delta(\alpha(d_i))\text{ and }\Delta(\hat{e})=\coprod_{i=1}^k\Delta(\alpha(\hat{e}_i))\,.
$$
The result follows from these facts.
\end{proof}

\begin{thm}
\label{thm:splitting2}

Let $\alpha_1,\ldots,\alpha_k$ be $P$-cosmall roots such that $\Delta(\alpha_1),\ldots,\Delta(\alpha_k)$ are pairwise totally disjoint. Let $d=\sum_{i=1}^k d(\alpha_i)$ for short. Then we have $d\in\Pi_P$. Let $e$ be the lifting of $d$. Then we have $\mathcal{B}_{R,e}\setminus R_P^+=\{\alpha_1,\ldots,\alpha_k\}$.

\end{thm}

\begin{proof}

For all $1\leq i\leq k$, let $d_i=d(\alpha_i)$ for short. We have $\widetilde{\Delta}(d_i)=\Delta(\alpha_i)$ for all $1\leq i\leq k$ since $\alpha_i$ is $P$-cosmall. By Theorem~\ref{thm:additiongreedy} applied to $d_1,\ldots,d_k$, we know that $(\alpha_1,\ldots,\alpha_k)$ is a greedy decomposition of $d$. In particular, we have $\widetilde{\Delta}(d)=\coprod_{i=1}^k\Delta(\alpha_i)$ and that $\Delta(\alpha_1),\ldots,\Delta(\alpha_k)$ are the distinct connected components of $\widetilde{\Delta}(d)$. By definition, we see that $d_1,\ldots,d_k$ are the connected components of $d$.

To prove that $d\in\Pi_P$, it suffices to prove that $d_1,\ldots,d_k\in\Pi_P$ (Corollary~\ref{cor:additionminimaldegrees}). But the later statement was proved in \cite[Proposition~4.4(6)]{minimaldegrees}. This proves the first claim of the theorem. The lifting $e$ of $d$ is now well-defined. 

In the first paragraph, we saw that $d_1,\ldots,d_k$ are the connected components of $d$. Hence, Theorem~\ref{thm:maxroots} implies that there are precisely $k$ maximal roots of $d$, namely $\alpha_1,\ldots,\alpha_k$. Let $\hat{e}=\sum_{\alpha\in\mathcal{B}_{R,e}\setminus R_P^+}\alpha^\vee$. Let $\hat{e}_1,\ldots,\hat{e}_{k'}$  be the connected components of $\hat{e}$. By Theorem~\ref{thm:maxroots}, \ref{thm:splitting1}, it follows that $k=k'$ and that we can rename the connected components of $\hat{e}$ in such a way that $\alpha(\hat{e}_i)=\alpha_i$ for all $1\leq i\leq k$. By Lemma~\ref{lem:maxroots} applied to $\hat{e}$, there exist roots $\alpha_{k+1},\ldots,\alpha_r$ ($k\leq r$) such that $(\alpha_1,\ldots,\alpha_k,\alpha_{k+1},\ldots,\alpha_r)$ is a greedy decomposition of $\hat{e}$.

\begin{proof}[Claim: We have $\alpha_{k+1},\ldots,\alpha_r\in R_P^+$]\renewcommand{\qedsymbol}{$\triangle$}

To prove the claim, it clearly suffices to show that $\hat{e}-\sum_{i=1}^k\alpha_i^\vee\in\mathbb{Z}\Delta_P^\vee$ or equivalent to show that $\hat{e}_P=d$. But we already saw this equality in the proof of Theorem~\ref{thm:splitting1}
\end{proof}

Also, in the proof of Theorem~\ref{thm:splitting1}, we saw that $\mathcal{B}_{R,\hat{e}}=\mathcal{B}_{R,e}\setminus R_P^+$, so that all members of every greedy decomposition of $\hat{e}$ are elements of $R^+\setminus R_P^+$. This fact is only compatible with the claim if $k=r$. In other words, we have shown that $(\alpha_1,\ldots,\alpha_k)$ is a greedy decomposition of $\hat{e}$. We can reformulate this result by saying that $\mathcal{B}_{R,e}\setminus R_P^+=\mathcal{B}_{R,\hat{e}}=\{\alpha_1,\ldots,\alpha_k\}$. This is all we wanted to prove.
\end{proof}

\begin{notation}

Let $d\in\Pi_P$. Let $e$ be the lifting of $d$. Then we denote by $\Sigma_{P,d}$ the sum defined by
$$
\Sigma_{P,d}=\sum_{\substack{\alpha\in\mathcal{B}_{R,e}\setminus R_P^+\\\gamma\in R_P^+}}(\gamma,\alpha^\vee)\,.
$$

\end{notation}

\begin{defn}
\label{def:Padmissible}

Let $d\in\Pi_P$. We say that $d$ is a $P$-admissible degree if $\Sigma_{P,d}\geq 0$.

\end{defn}

\begin{prop}
\label{prop:padmissible}

\leavevmode

\begin{enumerate}

\item 
\label{item:badmissible}

All minimal degrees in $\Pi_B$ are $B$-admissible.

\item
\label{item:pcosmalladmissible}

Let $\alpha_1,\ldots,\alpha_k$ be $P$-cosmall roots such that $\Delta(\alpha_1),\ldots,\Delta(\alpha_k)$ are pairwise totally disjoint. Then $\sum_{i=1}^k d(\alpha_i)$ is a $P$-admissible degree.

\item
\label{item:typeAadmissible}

Assume that $R$ is of type $\mathsf{A}$. Then all minimal degrees in $\Pi_P$ are $P$-admissible.

\end{enumerate}

\end{prop}

\begin{proof}[Proof of Item~\eqref{item:badmissible}] 

This is immediately clear from the definition. For all $e\in\Pi_B$, we have $\Sigma_{B,e}=0$ since $R_B^+=\emptyset$.
\end{proof}

\begin{proof}[Proof of Item~\eqref{item:pcosmalladmissible}]

Let $\alpha_1,\ldots,\alpha_k$ be as in the statement. Let $d=\sum_{i=1}^k d(\alpha_i)$ for short. By Theorem~\ref{thm:splitting2}, we have $d\in\Pi_P$. To see that $d$ is a $P$-admissible degree, we have to show that $\Sigma_{P,d}\geq 0$. In order to prove that $\Sigma_{P,d}\geq 0$, it suffices to prove that $\sum_{\gamma\in R_P^+}(\gamma,\alpha_i^\vee)\geq 0$ for all $1\leq i\leq k$ (cf. Theorem~\ref{thm:splitting2}). In other words, Theorem~\ref{thm:splitting2} helps us to reduce the assertion $\Sigma_{P,d}\geq 0$ to the case $k=1$. We assume from now on that $k=1$ and we write $\alpha=\alpha_1$ and $d=d(\alpha)$. 

\begin{proof}[Claim: We have $\Sigma_{P,d}=(c_1(G/B),\alpha^\vee)-(c_1(X),d)$]\renewcommand{\qedsymbol}{$\triangle$}

By \cite[Equation~(3)]{curvenbhd2}, we have 
$$
(c_1(G/B),\alpha^\vee)=\sum_{\gamma\in R^+}(\gamma,\alpha^\vee)\text{ and }(c_1(X),d)=\sum_{\gamma\in R^+\setminus R_P^+}(\gamma,d)\,.
$$
As $\sum_{\gamma\in R^+\setminus R_P^+}\gamma\in H^2(X)=\mathbb{Z}\{\omega_\beta\mid\beta\in\Delta\setminus\Delta_P\}$ we also find that
$$
\sum_{\gamma\in R^+\setminus R_P^+}(\gamma,d)=\sum_{\gamma\in R^+\setminus R_P^+}(\gamma,\alpha^\vee)\,.
$$
The claim now follows by subtracting the two displayed expressions for $(c_1(G/B),\alpha^\vee)$ and $(c_1(X),d)$. 
\end{proof}

By assumption, $\alpha$ is $P$-cosmall and hence also $B$-cosmall. By \cite[Theorem~6.1: $(a)\Leftrightarrow(b)$]{curvenbhd2} we therefore have
$$
\ell(s_\alpha)=(c_1(G/B),\alpha^\vee)-1\text{ and }\ell(s_\alpha W_P)=(c_1(X),d)-1\,.
$$
In view of the previous claim, these equations give $\Sigma_{P,d}=\ell(s_\alpha)-\ell(s_\alpha W_P)$. But this quantity is by definition certainly nonnegative. Hence, we find $\Sigma_{P,d}\geq 0$ -- as required.
\end{proof}

\begin{proof}[Proof of Item~\eqref{item:typeAadmissible}]

By Fact~\ref{fact:lifting}\eqref{item:maxrep}, the claim follows from the stronger statement that for all  $e\in\Pi_B$ such that $z_e^B$ is the maximal representative in $z_e^BW_P$ the inequality
$$
\sum_{\substack{\alpha\in\mathcal{B}_{R,e}\setminus R_P^+\\\gamma\in R_P^+}}(\gamma,\alpha^\vee)\geq 0
$$
holds. In turn, this statement follows from the more stronger statement that for all $e\in\Pi_B$ such that $z_e^B$ is the maximal representative in $z_e^BW_P$ and for all $\gamma\in R_P^+$ the inequality
$$
\sum_{\alpha\in\mathcal{B}_{R,e}\setminus R_P^+}(\gamma,\alpha^\vee)\geq 0
$$
holds. But this is clear in view of the even more precise statement in Corollary~\ref{cor:typea}.
\end{proof}

\begin{comment}

\begin{ex}
\label{ex:notPadmissible}

We give three examples of degrees which are not $P$-admissible. Let $R$ be of type $\mathsf{D}_4$, $\mathsf{E}_6$ or type $\mathsf{G}_2$. Let $P$ be the parabolic subgroup of $G$ such that $\Delta_P$ is as in the second row of Table~\ref{table:example}. Let $d=d_X\in\Pi_P$. The lifting of $d$ is $e=d_{G/B}$ and a simple computation shows that we have $\Sigma_{P,d}=-2<0$ for those three examples. Hence, the degree $d$ is minimal but not $P$-admissible.

\end{ex}

\end{comment}

The following 
%lemmas (Lemma~\ref{lem:prep1}, \ref{lem:prep2}, \ref{lem:prep3}) 
statements until the end of this section
are meant to prepare the proof of the main result on quasi-homogeneity. They are less interesting on their own right.

\begin{lem}
\label{lem:prep1}

Let $d\in\Pi_P$. Let $e$ be the lifting of $d$. Then we have the following formula:
$$
\ell(z_d^P)=(c_1(G/B),e)-\mathrm{card}(\mathcal{B}_{R,e})-\mathrm{card}(R_P^+)\,.
$$

\end{lem}

\begin{proof}

Let $d$ and $e$ be as in the statement. By Fact~\ref{fact:lifting}\eqref{item:defoflifting}, \eqref{item:maxrep}, we know that $z_d^P$ is the minimal and that $z_e^B$ is the maximal representative in $z_d^PW_P=z_e^BW_P$. Hence, we have $\ell(z_d^P)=\ell(z_e^B)-\mathrm{card}(R_P^+)$. The result follows from this and Corollary~\ref{cor:gencascade:main2}.
\end{proof}

\begin{lem}
\label{lem:prep2}

Let $d\in\Pi_P$. Let $e$ be the lifting of $d$. Then we have the following equality:
$$
(c_1(G/B),e)-(c_1(X),d)=\Sigma_{P,d}+\mathrm{card}\coprod_{\alpha\in\mathcal{B}_{R,e}\cap R_P^+}I(s_\alpha)+\mathrm{card}(\mathcal{B}_{R,e}\cap R_P^+)\,.
$$

\end{lem}

\begin{proof}

By arguments very similar to those in the proof of the claim in Proposition~\ref{prop:padmissible}\eqref{item:pcosmalladmissible} and as $e_P=d$ by Fact~\ref{fact:lifting}\eqref{item:ePequalsd}, we have
$$
(c_1(G/B),e)-(c_1(X),d)=\sum_{\gamma\in R^+}(\gamma,e)-\sum_{\gamma\in R^+\setminus R_P^+}(\gamma,e)=\sum_{\gamma\in R_P^+}(\gamma,e)\,.
$$
The last sum can be split into two summands
$$
\Sigma_{P,d}+\sum_{\substack{\alpha\in\mathcal{B}_{R,e}\cap R_P^+\\\gamma\in R_P^+}}(\gamma,\alpha^\vee)
$$
since $e=\sum_{\alpha\in\mathcal{B}_{R,e}}\alpha^\vee$ by Fact~\ref{fact:lifting}\eqref{item:defoflifting} and \cite[Remark~8.4]{minimaldegrees}.

\begin{proof}[Claim: Let $\alpha\in R_P^+$ be a $B$-cosmall root. Then we have $\ell(s_\alpha)=\sum_{\gamma\in R_P^+}(\gamma,\alpha^\vee)-1$]\renewcommand{\qedsymbol}{$\triangle$}

Indeed, 

\noindent
since $\alpha$ is $B$-cosmall, we have $\ell(s_\alpha)=\sum_{\gamma\in R^+}(\gamma,\alpha^\vee)-1$ by \cite[Equation~(3), Theorem~6.1: $(a)\Leftrightarrow(b)$]{curvenbhd2}. In order to prove the claim, it therefore suffices to show that $\sum_{\gamma\in R^+\setminus R_P^+}(\gamma,\alpha^\vee)=0$. Since $\alpha\in R_P^+$, we have $I(s_\alpha)\subseteq R_P^+$ by \cite[5.5, Theorem~(b)]{humphreys3}. Thus, we have $s_\alpha(R_P)=R_P$ and $s_\alpha(R^+\setminus R_P^+)=R^+\setminus R_P^+$. Since $(-,-)$ is $W$-invariant, we know that $(\gamma,\alpha)=-(s_\alpha(\gamma),\alpha)$. All facts together yields that the sum $\sum_{\gamma\in R^+\setminus R_P^+}(\gamma,\alpha)$ is equal to its negative. Hence, we find the desired vanishing.\footnote{It is actually easy to work out a more conceptional proof of the claim than we gave by using local notions as introduced in \cite[Section~6]{minimaldegrees}.}
\end{proof}

By Theorem~\ref{thm:gencascade}\eqref{item:gencascadebcosmall}, all roots in $\mathcal{B}_{R,e}\cap R_P^+$ are $B$-cosmall. If we apply the claim to each of them, we find that
$$
\sum_{\substack{\alpha\in\mathcal{B}_{R,e}\cap R_P^+\\\gamma\in R_P^+}}(\gamma,\alpha^\vee)=\sum_{\alpha\in\mathcal{B}_{R,e}\cap R_P^+}\ell(s_\alpha)+\mathrm{card}(\mathcal{B}_{R,e}\cap R_P^+)\,.
$$
All in all, the desired equality now follows from the last three displayed equations and Theorem~\ref{thm:gencascade:main2}.
\end{proof}

\begin{cor}

Let $d$ be a $P$-admissible degree. Let $e$ be the lifting of $d$. Then we have the following inequality:
$$
(c_1(G/B),e)-(c_1(X),d)\geq 0\,.
$$

\end{cor}

\begin{proof}

This follows directly from the assumption on $d$ since Lemma~\ref{lem:prep2} expresses the difference in question as a sum of three nonnegative summands.
\end{proof}

\begin{rem}

Let $d\in\Pi_P$. Let $e$ be the lifting of $d$. In general, if $d$ is not $P$-admissible, it may happen that the difference $(c_1(G/B),e)-(c_1(X),d)$ is negative.
%We give three examples for this behavior.
The reader can find examples for this behavior in type $\mathsf{D}_4$ or type $\mathsf{E}_6$ for the degree $d=d_X\in\Pi_P$ and a suitable parabolic subgroup $P$.

\end{rem}

\begin{comment}

{\color{black}\begin{ex}

Let $R$ be of type $\mathsf{D}_4$, $\mathsf{E}_6$ or type $\mathsf{G}_2$. Let $P$ be the parabolic subgroup of $G$ such that $\Delta_P$ is as in the second row of Table~\ref{table:example}. Let $d=d_X\in\Pi_P$. The lifting of $d$ is $e=d_{G/B}$ and a simple computation shows that $\mathcal{B}_{R,e}\cap R_P^+=\emptyset$. By Example~\ref{ex:notPadmissible} and Lemma~\ref{lem:prep2}, we find that
$$
(c_1(G/B),e)-(c_1(X),d)=\Sigma_{P,d}=-2<0\,.
$$

\end{ex}}

\end{comment}

\begin{lem}
\label{lem:prep3}

Let $d\in\Pi_P$. Let $e$ be the lifting of $d$. Then we have the following equality:
$$
(c_1(X),d)-\ell(z_d^P)=-\Sigma_{P,d}+\mathrm{card}\coprod_{\alpha\in\mathcal{B}_{R,e}\setminus R_P^+}(I(s_\alpha)\cap R_P^+)+\mathrm{card}(\mathcal{B}_{R,e}\setminus R_P^+)\,.
$$

\end{lem}

\begin{proof}

By Lemma~\ref{lem:prep1}, \ref{lem:prep2} we have the identity
\begin{multline*}
(c_1(X),d)-\ell(z_d^P)=-\Sigma_{P,d}+\mathrm{card}(R_P^+)-\mathrm{card}\coprod_{\alpha\in\mathcal{B}_{R,e}\cap R_P^+}I(s_\alpha)+\\\mathrm{card}(\mathcal{B}_{R,e})-\mathrm{card}(\mathcal{B}_{R,e}\cap R_P^+)\,.
\end{multline*}
The last difference in this formula is clearly equal to $\mathrm{card}(\mathcal{B}_{R,e}\setminus R_P^+)$. As $I(s_\alpha)\subseteq R_P^+$ for all $\alpha\in\mathcal{B}_{R,e}\cap R_P^+$ by \cite[5.5, Theorem~(b)]{humphreys3}, it suffices to show that
\begin{equation}
\label{eq:asbefore}
R_P^+\setminus\coprod_{\alpha\in\mathcal{B}_{R,e}\cap R_P^+}I(s_\alpha)=\coprod_{\alpha\in\mathcal{B}_{R,e}\setminus R_P^+}(I(s_\alpha)\cap R_P^+)
\end{equation}
in order to prove the identity which is stated in Lemma~\ref{lem:prep3}. We argue very similar as in the proof of Theorem~\ref{thm:positivity2} to show Equation~\eqref{eq:asbefore}. By Theorem~\ref{thm:gencascade:main2} and Fact~\ref{fact:lifting}\eqref{item:maxrep}, we have
$$
R_P^+\subseteq I(z_e^B)=\coprod_{\alpha\in\mathcal{B}_{R,e}}I(s_\alpha)
$$
and thus
$$
R_P^+=\coprod_{\alpha\in\mathcal{B}_{R,e}\setminus R_P^+}(I(s_\alpha)\cap R_P^+)\amalg\coprod_{\alpha\in\mathcal{B}_{R,e}\cap R_P^+}I(s_\alpha)\,.
$$
The desired Equation~\eqref{eq:asbefore} now follows by reorganizing the previous displayed equation. This completes the proof of Lemma~\ref{lem:prep3}.
\end{proof}

\begin{lem}
\label{lem:prep4}

Let $d\in\Pi_P$. Let $e$ be the lifting of $d$. Then we have the following equality:
$$
(c_1(X),d)-\ell(z_d^P)=-\sum_{\alpha\in\mathcal{B}_{R,e}\setminus R_P^+}\sum_{\gamma\in R_P^+\setminus I(s_\alpha)}(\gamma,\alpha^\vee)+\mathrm{card}(\mathcal{B}_{R,e}\setminus R_P^+)\,.
$$

\end{lem}

\begin{proof}

We can split the sum $\Sigma_{P,d}$ into two summands as follows
$$
\Sigma_{P,d}=\sum_{\alpha\in\mathcal{B}_{R,e}\setminus R_P^+}\sum_{\gamma\in R_P^+\setminus I(s_\alpha)}(\gamma,\alpha^\vee)+\sum_{\alpha\in\mathcal{B}_{R,e}\setminus R_P^+}\sum_{\gamma\in I(s_\alpha)\cap R_P^+}(\gamma,\alpha^\vee)\,.
$$
Let $\alpha\in\mathcal{B}_{R,e}\setminus R_P^+$ and $\gamma\in I(s_\alpha)\cap R_P^+$. By Theorem~\ref{thm:gencascade}\eqref{item:gencascadebcosmall}, the root $\alpha$ is $B$-cosmall. Hence, we have $(\gamma,\alpha^\vee)=1$ since $\alpha\neq\gamma$ (cf. \cite[Theorem~6.1: $(a)\Rightarrow(c)$]{curvenbhd2}). This means that all terms in the second summand of the previous displayed equation are equal to one. In other words, we have
$$
\sum_{\alpha\in\mathcal{B}_{R,e}\setminus R_P^+}\sum_{\gamma\in I(s_\alpha)\cap R_P^+}(\gamma,\alpha^\vee)=\mathrm{card}\coprod_{\alpha\in\mathcal{B}_{R,e}\setminus R_P^+}(I(s_\alpha)\cap R_P^+)
$$
in view of Theorem~\ref{thm:gencascade:main2}. The lemma now follows from the last two displayed equations and Lemma~\ref{lem:prep3}.
\end{proof}

\section{The diagonal curve}
\label{sec:dia}

In this section, we introduce in full detail the diagonal curve associated to a minimal degree (cf. Definition~\ref{def:dia}). For this construction, we make use of generalized cascades of orthogonal roots and liftings of minimal degrees which we introduced in the sections before. We formulate an additional assumption on minimal degrees (Assumption~\ref{ass:main}). For minimal degrees satisfying this assumption, we prove in the next section that the morphism associated to the diagonal curve has a dense open orbit in the moduli space of stable maps. In the course of this proof, the set of tangent directions associated to a minimal degree plays a certain role. We introduce and investigate this set in this section (cf. Notation~\ref{not:tangentdirection}). It parametrizes certain tangent directions which result from the action of a suitable subgroup of $G$ on the tangent direction of the diagonal curve.

\begin{notation}

Let $X^T$ denote the fixed point set of the left $T$-action on $X$. The elements of $X^T$ are called $T$-fixed points. It is well known (cf. \cite[Lemma~1 and Lemma~2]{pand}) that we have a bijection
$
W/W_P\cong X^T
$. For any $w\in W$, we denote by $x(w)$ the image of $wW_P$ under this bijection, i.e. $x(w)$ is the $T$-fixed point given by the equation
$
x(w)=wP
$.

\end{notation}

\begin{notation}[{\cite[Lemma~4.2]{fulton}}]
\label{not:Calpha}

Let $\alpha\in R^+\setminus R_P^+$. We denote by $C_\alpha\subseteq X$ the unique irreducible $T$-invariant curve containing the points $x(1)$ and $x(s_\alpha)$. By \cite[Lemma~4.2]{fulton} such a unique curve exists. Moreover \cite[Lemma~4.2]{fulton} says that $C_\alpha$ is isomorphic to $\mathbb{P}^1$ and has degree $[C_\alpha]=d(\alpha)$. For an explicit construction of $C_\alpha$, see \cite[Section~3]{fulton}.

\end{notation}

\begin{defn}
\label{def:dia}

Let $d\in\Pi_P$. Let $e$ be the lifting of $d$. Then we define a morphism $f_{P,d}$ by the assignment 
$$
f_{P,d}\colon\mathbb{P}^1\hookrightarrow\prod_{\alpha\in\mathcal{B}_{R,e}\setminus R_P^+}C_\alpha\hookrightarrow X
$$
where the first morphism is the diagonal embedding of $\mathbb{P}^1$ into $\mathrm{card}(\mathcal{B}_{R,e}\setminus R_P^+)$ isomorphic copies of $\mathbb{P}^1$ and the second morphism is the embedding into $X$ which is well-defined due to Theorem~\ref{thm:gencascade}\eqref{item:stronglyorthogonal}. Again by Theorem~\ref{thm:gencascade}\eqref{item:stronglyorthogonal}, the definition of $f_{P,d}$ is independent of the ordering of the product $\prod_{\alpha\in\mathcal{B}_{R,e}\setminus R_P^+}$. Hence, the morphism $f_{P,d}$ is well-defined. We call the image $f_{P,d}(\mathbb{P}^1)$ the diagonal curve (associated to $d$).

\end{defn}

\begin{fact}
\label{fact:dia}

Let $d\in\Pi_P$. The morphism $f_{P,d}$ is a closed immersion of degree $(f_{P,d})_*[\mathbb{P}^1]=d$. Moreover, we have $f_{P,d}(0)=x(1)$ and $f_{P,d}(\infty)=x(z_d^P)$.

\end{fact}

\begin{proof}

The morphism $f_{P,d}$ is defined as the composition of two closed immersion. Hence, $f_{P,d}$ is itself a closed immersion. We now compute the degree of $f_{P,d}$. Let $e$ be the lifting of $d$. By definition of $f_{P,d}$ and the result \cite[Lemma~3.4]{fulton} recalled in Notation~\ref{not:Calpha}, we have
$$
(f_{P,d})_*[\mathbb{P}^1]=[f_{P,d}(\mathbb{P}^1)]=\sum_{\alpha\in\mathcal{B}_{R,e}\setminus R_P^+}[C_\alpha]=\sum_{\alpha\in\mathcal{B}_{R,e}\setminus R_P^+}d(\alpha)\,.
$$
As in the proof of Theorem~\ref{thm:splitting1}, by Fact~\ref{fact:lifting}\eqref{item:ePequalsd} and \cite[Remark~8.4]{minimaldegrees}, we also have
$$
\sum_{\alpha\in\mathcal{B}_{R,e}\setminus R_P^+}d(\alpha)=e_P=d\,.
$$
Altogether, this leads to the equation $(f_{P,d})_*[\mathbb{P}^1]=d$ -- as claimed. The equation $f_{P,d}(0)=x(1)$ is immediately clear from the definition of $f_{P,d}$. By Theorem~\ref{thm:gencascade:main} and Fact~\ref{fact:lifting}\eqref{item:defoflifting}, we have
$$
z_d^PW_P=z_e^BW_P=\prod_{\alpha\in\mathcal{B}_{R,e}}s_\alpha\cdot W_P=\prod_{\alpha\in\mathcal{B}_{R,e}\setminus R_P^+}s_\alpha\cdot W_P\,.
$$
This equation basically shows that $f_{P,d}(\infty)=x(z_d^P)$. (Note that $\infty$ is sent to $x(s_\alpha)$ under the isomorphism $\mathbb{P}^1\cong C_\alpha$ for all $\alpha\in R^+\setminus R_P^+$.)
\end{proof}

\begin{rem}

Let $d\in\Pi_P$. One may conjecture that the morphism $f_{P,d}$ has a dense open orbit in $\overline{M}_{0,3}(X,d)$ under the action of the automorphism group $\mathrm{Aut}(X)$. In this generality, we are not able to prove a quasi-homogeneity result, mainly because we do not have a type independent description of $\mathrm{Aut}(X)$.
%Nevertheless, we proposed a general framework in which other authors may extend the result. 
In this paper, we will prove the sharper and more restrictive result that the morphism $f_{P,d}$ has a dense open orbit in $\overline{M}_{0,3}(X,d)$ under the action of $G$ as long as $d$ satisfies Assumption~\ref{ass:main}. Note that all minimal degrees satisfy Assumption~\ref{ass:main} if $R$ is simply laced or if $P=B$. In this way, we obtain quasi-homogeneity of the moduli space $\overline{M}_{0,3}(X,d)$ for a large class of degrees $d$.

%$d$ is $P$-admissible (Theorem~\ref{thm:main}). Since all minimal degrees in $\Pi_B$ are $B$-admissible by Proposition~\ref{prop:padmissible}\eqref{item:badmissible}, the result is optimal for generalized complete flag varieties $G/B$.

\end{rem}

\begin{fact}
\label{fact:zinverse}

Let $d$ be a degree. Then we have $(z_d^P)^{-1}=w_Pz_d^Pw_P$.

\end{fact}

\begin{proof}

Let $e\in H_2(G/B)$ be a sufficiently large degree such that $e_P=d$ and such that $z_d^Pw_P=z_e^B$. We can choose such a degree $e\in H_2(G/B)$ by \cite[Corollary~4.12(d)]{curvenbhd2}. By \cite[Corollary~4.9]{curvenbhd2}, we know that $z_e^B$ is an involution. Hence, we have $(z_d^Pw_P)^{-1}=z_d^Pw_P$ or equivalent $(z_d^P)^{-1}=w_Pz_d^Pw_P$ -- as claimed.
\end{proof}

\begin{notation}

For a root $\alpha\in R$, we denote by $U_\alpha$ the associated root group as defined in \cite[26.3, Theorem~(a)]{humphreys}. 

\end{notation}

\begin{notation}
\label{not:R(P)}

To simplify notation, we write $R(P)=R^+\cup R_P$ for short. The set of roots $R(P)$ is precisely the set of roots $\alpha\in R$ such that $U_\alpha\subseteq P$.

\end{notation}

\begin{notation}

For a Weyl group element $z\in W$, we denote by $P^z$ the conjugate of $P$, i.e. $P^z=zPz^{-1}$.

\end{notation}

\begin{lem}
\label{lem:act}

Let $d$ be a degree. Let $z=z_d^P$ for short. Then we have $U_{-\gamma}\subseteq P\cap P^z$ for all $\gamma\in R_P^+$.

\end{lem}

\begin{proof}

%If we replace $z$ with the minimal representative in $zW_P$, the group $P\cap P^z$ does not change. Hence, we may assume that $z$ is the minimal representative in $zW_P$. By definition, we have $U_\gamma\subseteq P$ for all $\gamma\in R_P^+$. Therefore, it suffices to show that $U_\gamma\subseteq P^z$ for all $\gamma\in R_P^+$ which is equivalent to $z^{-1}(\gamma)\in R^+\cup R_P$

By Notation~\ref{not:R(P)}, we can describe the set of roots $\alpha\in R$ such that $U_\alpha\subseteq P\cap P^z$ as $\{\gamma\in R(P)\mid U_\gamma\subseteq P^z\}$. By \cite[26.3, Theorem~(b)]{humphreys} and Notation~\ref{not:R(P)}, this set equals $\{\gamma\in R(P)\mid z^{-1}(\gamma)\in R(P)\}$. Since $w_P\in W_P$, it is clear that $w_P(R(P))=R(P)$. Fact~\ref{fact:zinverse} therefore yields that the aforementioned set is equal to $\{\gamma\in R(P)\mid zw_P(\gamma)\in R(P)\}$. 

We have to show that this set contains $R_P^-$. Since $R_P^-\subseteq R(P)$, it clearly suffices to show that $zw_P(R_P^-)\subseteq R(P)$. Since $w_P$ is the longest element of the Weyl group $W_P$, we know that $I(w_P)=R_P^+$. But this means that $w_P(R_P^+)=R_P^-$. Since $w_P$ is an involution, the equation $w_P(R_P^+)=R_P^-$ is equivalent to $w_P(R_P^-)=R_P^+$. Therefore, the claim $zw_P(R_P^-)\subseteq R(P)$ is equivalent to the statement $z(R_P^+)\subseteq R(P)$.

We now show the latter statement. By definition, the element $z$ is the minimal representative in $zW_P$, so that we have $I(z)\cap R_P^+=\emptyset$. In other words, we have $z(R_P^+)\subseteq R^+\subseteq R(P)$ which suffices to prove the lemma.
%
%We have to show that this set contains $R_P^+$. Since $R_P^+\subseteq R(P)$, it clearly suffices to shows that $z(\gamma)\in R(P)$ for all $\gamma\in R_P^+$. But by definition, the element $z$ is the minimal representative in $zW_P$, so that we have $I(z)\cap R_P^+=\emptyset$. In other words, we have $z(\gamma)\in R^+\subseteq R(P)$ for all $\gamma\in R_P^+$ which suffices to prove the lemma.
\end{proof}

\begin{notation}
\label{not:tangentdirection}

Let $d\in\Pi_P$. Let $e$ be the lifting of $d$. Then we denote by $\mathrm{TD}_{P,d}$ the set of roots defined by
$$
\mathrm{TD}_{P,d}=\{-\alpha-\gamma\in R^-\setminus R_P^-\mid\alpha\in\mathcal{B}_{R,e}\setminus R_P^+,\gamma\in R_P^+\cup\{0\}\}\,.
$$
We call this set of roots the set of tangent directions (associated to $d$).

\end{notation}

\begin{thm}
\label{thm:disjointness_new}

Let $d\in\Pi_P$. Let $e$ be the lifting of $d$. Then we have an injective map
$$
\{(\alpha,\gamma)\in(\mathcal{B}_{R,e}\setminus R_P^+)\times R_P^+\mid(\gamma,\alpha^\vee)=-1\}\hookrightarrow\mathrm{TD}_{P,d}\setminus(-(\mathcal{B}_{R,e}\setminus R_P^+))
$$
defined by the assignment $(\alpha,\gamma)\mapsto-\alpha-\gamma$. Here we denote by $-(\mathcal{B}_{R,e}\setminus R_P^+)$ the set of roots consisting of all $-\alpha$ where $\alpha\in\mathcal{B}_{R,e}\setminus R_P^+$.

\end{thm}

\begin{proof}

First, we prove that the map defined by the assignment as in the statement is well-defined. Let $\alpha\in\mathcal{B}_{R,e}\setminus R_P^+$ and $\gamma\in R_P^+$ such that $(\gamma,\alpha^\vee)=-1$. Then, we have $s_\alpha(\gamma)=\alpha+\gamma$. Hence, $\alpha+\gamma\in R^+\setminus R_P^+$ and $-\alpha-\gamma\in R^-\setminus R_P^-$. This proves that $-\alpha-\gamma\in\mathrm{TD}_{P,d}$.

Suppose for a contradiction that $\alpha+\gamma=\alpha'$ for some $\alpha'\in\mathcal{B}_{R,e}\setminus R_P^+$. We clearly have $\alpha\neq\alpha'$ and that $\alpha$ and $\alpha'$ are orthogonal (cf. Theorem~\ref{thm:gencascade}\eqref{item:stronglyorthogonal}). If we now apply $(-,\alpha^\vee)$ to the equation $\alpha+\gamma=\alpha'$, we find that $0=(\alpha',\alpha^\vee)=(\alpha,\alpha^\vee)+(\gamma,\alpha^\vee)=2-1=1$ -- a contradiction. This proves that $-\alpha-\gamma\in\mathrm{TD}_{P,d}\setminus(-(\mathcal{B}_{R,e}\setminus R_P^+))$. This completes the proof of the well-definedness of the map.

We now prove injectivity of the map. To this end, suppose that $\alpha+\gamma=\alpha'+\gamma'$ where $\alpha,\alpha'\in\mathcal{B}_{R,e}\setminus R_P^+$ and $\gamma,\gamma'\in R_P^+$ such that $(\gamma,\alpha^\vee)=(\gamma',\alpha'^\vee)=-1$. Suppose for a contradiction that $\alpha\neq\alpha'$.

\begin{proof}[Claim: We have $(\gamma',\alpha^\vee)=(\gamma,\alpha'^\vee)=1$]\renewcommand{\qedsymbol}{$\triangle$}

\begin{comment}

Indeed, by what we saw in the first paragraph of the proof of the theorem, we can rewrite the equation $\alpha+\gamma=\alpha'+\gamma'$ as $s_\alpha(\gamma)=s_{\alpha'}(\gamma')$. The very last equation can also be written as
$
\gamma=s_\alpha s_{\alpha'}(\gamma')
$. 
Since $\alpha\neq\alpha'$ by assumption, the roots $\alpha$ and $\alpha'$ are orthogonal (cf. Theorem~\ref{thm:gencascade}\eqref{item:stronglyorthogonal}). Hence, the last equation can be expanded as
$$
\gamma=\gamma'+\alpha'-(\gamma',\alpha^\vee)\alpha\,.
$$
On the other hand, we know that $\gamma=\gamma'+\alpha'-\alpha$

\end{comment}

Since $\alpha\neq\alpha'$ by assumption, the roots $\alpha$ and $\alpha'$ are orthogonal (cf. Theorem~\ref{thm:gencascade}\eqref{item:stronglyorthogonal}). In view of this fact, if we apply $(-,\alpha^\vee)$ to the equation $\alpha+\gamma=\alpha'+\gamma'$, we find that $(\gamma',\alpha^\vee)=1$. Similarly, if we apply $(-,\alpha'^\vee)$, we find $(\gamma,\alpha'^\vee)=1$. This completes the proof of the claim.
\end{proof}

By the previous claim, $s_{\alpha}(\gamma')=\gamma'-\alpha$ and $s_{\alpha'}(\gamma)=\gamma-\alpha'$ are roots. Hence, $\rho=\alpha-\gamma'=\alpha'-\gamma$ is also a root which lies in $R^+\setminus R_P^+$. The root $\rho$ is obviously an element of $R(\alpha)\cap R(\alpha')$. In particular, we have $R(\alpha)\cap R(\alpha')\neq\emptyset$. By Theorem~\ref{thm:gencascade}\eqref{item:totallydisjoint2}, there exists a root $\varphi\in R^+$ such that $\alpha,\alpha'\in C_{R,e}(\varphi)$. In particular, this means that $\alpha$ and $\alpha'$ are comparable (cf. Theorem~\ref{thm:gencascade}\eqref{item:totallyordered}). Without loss of generality, we may assume that $\alpha\geq\alpha'$. In view of the assumption $\alpha\neq\alpha'$, this means that we even have $\alpha>\alpha'$.

\begin{proof}[Claim: We have $(\alpha,\rho)=0$]\renewcommand{\qedsymbol}{$\triangle$}

The arguments in the proof of this claim are very similar to those in the proof of Theorem~\ref{thm:positivity}. For the convenience of the reader, we spell them out once more. 

Let $\hat{e}=\sum_{\mu\in\mathcal{B}_{R,e}\colon\mu\leq\alpha}\mu^\vee$. By \cite[Proposition~4.4(9)]{minimaldegrees}, we have $\hat{e}\in\Pi_B$. By \cite[Proposition~3.10(7)]{minimaldegrees}, it follows that $\Delta(\hat{e})=\Delta(\alpha)$. Thus, $\hat{e}$ is a connected degree with $\alpha(\hat{e})=\alpha$. By \cite[Proposition~3.10(7), 3.16, Theorem~8.1]{minimaldegrees}, it follows that
$$
\Delta(\rho)\subseteq\Delta(\alpha')\subseteq\Delta(e-\alpha^\vee)\subseteq\Delta_\alpha^\circ\,.
$$
Therefore, we have $(\alpha,\rho)=0$ -- as claimed.
\end{proof}

In one of the claims before, we figured out that $(\gamma',\alpha^\vee)=1$. By the very definition of $\rho=\alpha-\gamma'$, we therefore find that $(\rho,\alpha^\vee)=2-1=1$. This equation obviously contradicts the statement of the previous claim. This contradiction shows that we must have $\alpha=\alpha'$ and thus $(\alpha,\gamma)=(\alpha',\gamma')$. This shows that the defined map is injective and completes the proof of the theorem.
\end{proof}

\begin{assumption}
\label{ass:main}

Let $d\in\Pi_P$. Let $e$ be the lifting of $d$. We say that $d$ satisfies Assumption~\ref{ass:main} if one of the following conditions holds:

\begin{enumerate}

\item
\label{item:long}

All roots in $\mathcal{B}_{R,e}\setminus R_P^+$ are long.

\item
\label{item:g/b}

We have $P=B$, i.e. $X$ is a generalized complete flag variety.

\item
\label{item:ass_pcosmall}

There exist $P$-cosmall roots $\alpha_1,\ldots,\alpha_k$ such that $\Delta(\alpha_1),\ldots,\Delta(\alpha_k)$ are pairwise totally disjoint and such that $d=\sum_{i=1}^k d(\alpha_i)$.

\end{enumerate}

\end{assumption}

\begin{lem}
\label{lem:long}

Let $\alpha$ be a long root and let $\gamma$ be a root non-proportional to $\alpha$ (i.e. $\gamma\neq\pm\alpha$). Then we have $(\gamma,\alpha^\vee)\in\{-1,0,1\}$.

\end{lem}

\begin{proof}

Since $\alpha$ and $\gamma$ are non-proportional, they are linearly independent. Thus, the Cauchy-Schwarz inequality applied to $\alpha$ and $\gamma$ is strict and reads as
$$
|(\gamma,\alpha)|<\|\gamma\|\|\alpha\|
$$
where $\|\cdot\|$ is the norm associated to the scalar product $(-,-)$. Using this strict inequality, we find that
$$
|(\gamma,\alpha^\vee)|=\frac{2|(\gamma,\alpha)|}{(\alpha,\alpha)}<\frac{2\|\gamma\|\|\alpha\|}{\|\alpha\|^2}\leq 2
$$
where the last inequality follows since $\alpha$ is long. This inequality gives the claimed values of $(\gamma,\alpha^\vee)$.
\end{proof}

\begin{lem}
\label{lem:prep5}

Let $d\in\Pi_P$. Let $e$ be the lifting of $d$. Assume that $d$ satisfies Assumption~\ref{ass:main}. Then we have the following equality:
$$
-\sum_{\alpha\in\mathcal{B}_{R,e}\setminus R_P^+}\sum_{\gamma\in R_P^+\setminus I(s_\alpha)}(\gamma,\alpha^\vee)=\mathrm{card}\{(\alpha,\gamma)\in(\mathcal{B}_{R,e}\setminus R_P^+)\times R_P^+\mid(\gamma,\alpha^\vee)=-1\}\,.
$$

\end{lem}

\begin{proof}

Assume first that Assumption~\ref{ass:main}\eqref{item:long} is satisfied. Let $(\alpha,\gamma)$ be a pair as it occurs in the index set of the double sum in the statement. By definition, we know that $(\gamma,\alpha^\vee)\leq 0$ -- otherwise $\gamma\in I(s_\alpha)$. (We use here that $\alpha\in R^+\setminus R_P^+$ and $\gamma\in R_P^+$.) Since $\alpha$ is long by assumption, Lemma~\ref{lem:long} applies and we find that $(\gamma,\alpha^\vee)\in\{-1,0\}$. If we discard all summands from the double sum which are zero, we are counting precisely the pairs $(\alpha,\gamma)$ such that $(\gamma,\alpha^\vee)=-1$. In this counting, we can equally well allow $\gamma$ to range over all of $R_P^+$ (instead of $\gamma\in R_P^+\setminus I(s_\alpha)$), since all elements $\gamma\in I(s_\alpha)$ satisfy $(\gamma,\alpha^\vee)=1$ by \cite[Theorem~6.1: $(a)\Rightarrow(c)$]{curvenbhd2}. (Here we use that $\alpha$ is $B$-cosmall by Theorem~\ref{thm:gencascade}\eqref{item:gencascadebcosmall}.) This proves the equality in this case.

If Assumption~\ref{ass:main}\eqref{item:g/b} is satisfied, both sides of the claimed equality are obviously zero. There is nothing to prove.

Finally, assume that Assumption~\ref{ass:main}\eqref{item:ass_pcosmall} is satisfied. Let $\alpha_1,\ldots,\alpha_k$ be $P$-cosmall roots such that $\Delta(\alpha_1),\ldots,\Delta(\alpha_k)$ are pairwise totally disjoint and such that $d=\sum_{i=1}^k d(\alpha_i)$. By Theorem~\ref{thm:splitting2}, we know that $\mathcal{B}_{R,e}\setminus R_P^+=\{\alpha_1,\ldots,\alpha_k\}$, in particular all elements of $\mathcal{B}_{R,e}\setminus R_P^+$ are $P$-cosmall. In view of this fact, Theorem~\ref{thm:pcosmall} shows that the left side of the claimed equality is zero. But Theorem~\ref{thm:pcosmall} also shows that the right side of the equality is zero. (For $\alpha\in\mathcal{B}_{R,e}\setminus R_P^+$, all elements $\gamma\in R_P^+\setminus I(s_\alpha)$ satisfy $(\alpha,\gamma)=0$ and all elements $\gamma\in I(s_\alpha)$ satisfy $(\gamma,\alpha^\vee)=1$ as it was explained in the first paragraph of this proof.) This shows the claimed equality in the last case.
\end{proof}

\begin{thm}
\label{thm:prep}

Let $d\in\Pi_P$. Assume that $d$ satisfies Assumption~\ref{ass:main}. Then we have the following inequality:
$$
(c_1(X),d)-\ell(z_d^P)\leq\mathrm{card}(\mathrm{TD}_{P,d})\,.
$$

\end{thm}

\begin{proof}

The theorem follows by combining Lemma~\ref{lem:prep4}, \ref{lem:prep5} and Theorem~\ref{thm:disjointness_new}.
\end{proof}

\section{Main result on quasi-homogeneity}

This section is devoted to the proof of the main result on quasi-homogeneity. After all preparations, it remains to combine the results proved until now and to relate them to the geometric question of quasi-homogeneity of the moduli space of stable maps.

{\color{black}\begin{notation}

Let $d$ be a degree. The moduli space $\overline{M}_{0,3}(X,d)$ comes equipped with three evaluation maps. For each $i\in\{1,2,3\}$, the $i$th evaluation map $\mathrm{ev}_i\colon\overline{M}_{0,3}(X,d)\to X$ is defined by 
$$
\mathrm{ev}_i([C,p_1,p_2,p_3,\mu\colon C\to X])=\mu(p_i)\,.
$$

\end{notation}}

\begin{thm}
\label{thm:main}

%Let $d$ be a $P$-admissible degree. 
Let $d\in\Pi_P$. Assume that $d$ satisfies Assumption~\ref{ass:main}.
Then the morphism $f_{P,d}$ has a dense open orbit in $\overline{M}_{0,3}(X,d)$ under the action of $G$. In particular, the moduli space $\overline{M}_{0,3}(X,d)$ is quasi-homogeneous under the action of $G$ for all 
%$P$-admissible degrees $d$.
minimal degrees $d\in\Pi_P$ which satisfy Assumption~\ref{ass:main}.

\end{thm}

\begin{comment}

\begin{rem}

In view of Proposition~\ref{prop:padmissible}, Theorem~\ref{thm:main} applies for three interesting subclasses of minimal degrees in $\Pi_P$.

\end{rem}

\end{comment}

\begin{proof}%[Proof of Theorem~\ref{thm:main}]

%We fix a $P$-admissible degree $d$. 
We fix a minimal degree $d\in\Pi_P$ which satisfies Assumption~\ref{ass:main}.
Let $z=z_d^P$ for short. Let $\overline{M}=\overline{M}_{0,3}(X,d)$ for short. We denote by $\overline{M}(2)$ the fiber of the total evaluation map $\mathrm{ev}_1\times\mathrm{ev}_2\colon\overline{M}\to X\times X$ over the point $(x(1),x(z))$.

\begin{proof}[Claim: We have $\mathrm{dim}(\overline{M}(2))=(c_1(X),d)-\ell(z)$]\renewcommand{\qedsymbol}{$\triangle$}

Indeed, the evaluation map $\mathrm{ev}_1\colon\overline{M}\to X$ is flat (cf. \cite[Lemma~2.5.1]{invitation}) and obviously surjective. Denote by $\overline{M}(1)$ the fiber of $\mathrm{ev}_1$ over $x(1)$. By the dimension formula for flat morphisms, we find that $\mathrm{dim}(\overline{M}(1))=\mathrm{dim}(\overline{M})-\mathrm{dim}(X)$. Similarly, we next consider the evaluation map $\mathrm{ev}_2\colon\overline{M}(1)\to X_z$. This maps is again flat (cf. \cite[loc. cit.]{invitation}) and surjective (cf. \cite[Theorem~5.1]{curvenbhd2}). The fiber of $\mathrm{ev}_2$ over $x(z)$ is precisely $\overline{M}(2)$. Again, by the dimension formula for flat morphisms, we find that $\mathrm{dim}(\overline{M}(2))=\mathrm{dim}(\overline{M})-\mathrm{dim}(X)-\ell(z)$. If we plug in the dimension formula for $\overline{M}$ (cf. \cite[Theorem~2(i)]{fultonpan}) in the last equation, we see that the claim is true.
\end{proof}

Let $L=P\cap P^z$. It is clear that $L$ acts on $\overline{M}(2)$. Moreover, the morphism $f_{P,d}$ has a dense open orbit in $\overline{M}$ under the action of $G$ if and only if it has a dense open orbit in $\overline{M}(2)$ under the action of $L$. We rather prove the latter equivalent formulation of the theorem. 

Let $f=f_{P,d}$ for short. Let $T_f$ be the tangent space at $f$ of the orbit $Lf\subseteq\overline{M}(2)$ of $f$ under the action of $L$. In order to prove that $f$ has a dense open orbit in $\overline{M}(2)$ under the action of $L$, it suffices to prove that $\mathrm{dim}(\overline{M}(2))\leq\mathrm{dim}(T_f)$. 

\begin{proof}[Claim: We have $\mathrm{card}(\mathrm{TD}_{P,d})\leq\mathrm{dim}(T_f)$]\renewcommand{\qedsymbol}{$\triangle$}

Let $\mathfrak{g}$ be the Lie algebra of $G$, let $\mathfrak{p}$ be the Lie algebra of $P$, and let $\mathfrak{t}$ be the Lie algebra of $T$. For a root $\alpha\in R$, we denote by $\mathfrak{g}_\alpha$ the Lie algebra of $U_\alpha$.
By \cite[Proposition~1.1]{Timashev}, the tangent space of $X$ at the point $x(1)$ can be identified with $\mathfrak{g}/\mathfrak{p}$. Hence, the tangent space $T_f$ can be identified with a vector subspace of $\mathfrak{g}/\mathfrak{p}$. We will now prove an inclusion of vector spaces as follows:
$$
\bigoplus_{\mu\in\mathrm{TD}_{P,d}}(\mathfrak{g}_\mu+\mathfrak{p})/\mathfrak{p}\subseteq T_f\,.
$$
The claim follows immediately from this by taking dimensions. (Note that $\mathrm{TD}_{P,d}\subseteq R^-\setminus R_P^-$ by definition.)

%For a root $\alpha\in R$, we denote by $x_\alpha\in\mathfrak{g}_\alpha$ the associated root vector. In order to prove the desired inclusion of vector spaces, it clearly suffices to prove that $x_\mu+\mathfrak{p}\in T_f$ for all $\mu\in\mathrm{TD}_{P,d}$.

Let $e$ be the lifting of $d$. For a root $\alpha\in R$, we denote by $x_\alpha\in\mathfrak{g}_\alpha$ the associated root vector. For a root $\alpha\in R$, we also write $h_\alpha=[x_\alpha,x_{-\alpha}]\in\mathfrak{t}$.
%The root vector $x_\alpha$ should be probably called a Chevalley basis but I am not sure. I think root vector does not specify which element of $\mathfrak{g}_\alpha$ -- is defined only up to nonzero scalar.
In order to prove the desired inclusion of vector spaces, we prove first that $x_{-\alpha}+\mathfrak{p}\in T_f$ for all $\alpha\in\mathcal{B}_{R,e}\setminus R_P^+$. By definition of $f$, the tangent direction of $f$ at the point $x(1)$ is the image of $x=\sum_{\alpha\in\mathcal{B}_{R,e}\setminus R_P^+}x_{-\alpha}$ in $\mathfrak{g}/\mathfrak{p}$. Hence, we have $x+\mathfrak{p}\in T_f$. 
%$T_f$ parametrizes tangent directions of deformations of $f$, in particular the tangent direction of $f$ at the point $x(1)$ (this point because all identifications work only for this point, like the one for the tangent space of $X$) -- as $f$ is trivially a deformation of itself.
In view of Theorem~\ref{thm:gencascade}\eqref{item:stronglyorthogonal}, we see that $[h_\alpha,x]=[h_\alpha,x_{-\alpha}]=-2x_{-\alpha}\in\mathfrak{g}_{-\alpha}\setminus\{0\}$ for all $\alpha\in\mathcal{B}_{R,e}\setminus R_P^+$. Since $T\subseteq B\cap B^z\subseteq L$ and since $\mathfrak{t}$ acts on $T_f$, we see from this that $x_{-\alpha}+\mathfrak{p}\in T_f$ for all $\alpha\in\mathcal{B}_{R,e}\setminus R_P^+$. 
%But this shows the inclusion $(\mathfrak{g}_{-\alpha}+\mathfrak{p})/\mathfrak{p}\subseteq T_f$ for all $\alpha\in\mathcal{B}_{R,e}\setminus R_P^+$. 

Let $\mu\in\mathrm{TD}_{P,d}$. The desired inclusion of vector spaces follows if we prove that $(\mathfrak{g}_\mu+\mathfrak{p})/\mathfrak{p}\subseteq T_f$. By the previous paragraph and the definition of $\mathrm{TD}_{P,d}$, we may assume that $\mu=-\alpha-\gamma$ for some $\alpha\in\mathcal{B}_{R,e}\setminus R_P^+$ and some $\gamma\in R_P^+$. Since $\mu$ is a root, we have $[\mathfrak{g}_{-\gamma},\mathfrak{g}_{-\alpha}]=\mathfrak{g}_\mu$. By Lemma~\ref{lem:act}, we know that $\mathfrak{g}_{-\gamma}$ acts on $T_f$. Hence, the previous paragraph and the previous equation show that we have indeed $(\mathfrak{g}_\mu+\mathfrak{p})/\mathfrak{p}\subseteq T_f$.
\end{proof}

After all, the desired inequality $\mathrm{dim}(\overline{M}(2))\leq\mathrm{dim}(T_f)$ is now easy to deduce from our previous results. Indeed, we have
\begin{align*}
\mathrm{dim}(\overline{M}(2))&=(c_1(X),d)-\ell(z)&&\text{by the first claim}\\
%
%&=-\Sigma_{P,d}+\mathrm{card}\coprod_{\alpha\in\mathcal{B}_{R,e}\setminus R_P^+}(I(s_\alpha)\cap R_P^+)+\mathrm{card}(\mathcal{B}_{R,e}\setminus R_P^+)&\text{by Lemma~\ref{lem:prep3}}\\
%&\leq\mathrm{card}\coprod_{\alpha\in\mathcal{B}_{R,e}\setminus R_P^+}(I(s_\alpha)\cap R_P^+)+\mathrm{card}(\mathcal{B}_{R,e}\setminus R_P^+)&\text{since $d$ is $P$-admissible}\\
%&\leq\mathrm{card}(\mathrm{TD}_{P,d})&\text{by Theorem~\ref{thm:disjointness}}\\
%
&\leq\mathrm{card}(\mathrm{TD}_{P,d})&&\text{by Theorem~\ref{thm:prep}}\\
&\leq\mathrm{dim}(T_f)&&\text{by the second claim.}
\end{align*}
% The only interpretation I have for this inequality being strict is that the moduli space stays quasi-homogeneous even if there are more than three marked points.
This completes the proof of the desired inequality and hence the proof of the theorem.
\end{proof}

\begin{comment}

\appendix

\section{The graph associated to a degree}
\label{appendix:graph}

\begin{ex}
\label{ex:computationndtyped}

\end{ex}

\section{Characterization of \texorpdfstring{$P$-cosmall}{P-cosmall} roots}

\end{comment}

%\vskip.5in

\bibliographystyle{amsplain}
\bibliography{diplom}

%\begin{thebibliography}{100}
%\bibitem{Test} J. Smith, \emph{Test Title}, Test Publisher, New York, 19--. 
%\end{thebibliography}

\end{document}